\documentclass{article}
\usepackage{amsmath,amssymb,amscd, float, verbatim, latexsym,cite,wrapfig}
\usepackage{amsthm}
\usepackage{graphicx}
\usepackage[all]{xy}
\usepackage{lscape,subcaption,authblk}
\usepackage{color,enumerate,setspace,url}		
\usepackage{mathtools}
\mathtoolsset{showonlyrefs=true}
\usepackage{bm}

\usepackage{setspace}

\usepackage[margin=1.in]{geometry}

\usepackage{wrapfig}

\usepackage{color}

\newtheorem{theorem}{Theorem}[section]
\newtheorem{lemma}[theorem]{Lemma}

\newtheorem{proposition}[theorem]{Proposition}
\newtheorem{corollary}[theorem]{Corollary}

\newtheorem{remark}{Remark}[section]
\newtheorem{definition}[theorem]{Definition}

\newcommand{\C}{\mathbb{C}}

\newcommand{\N}{\mathbb{N}}
\newcommand{\R}{\mathbb{R}}
\newcommand{\bS}{\mathbb{S}}
\newcommand{\T}{\mathbb{T}}
\newcommand{\Z}{\mathbb{Z}}

\newcommand{\expig}{e^{i \theta}}
\newcommand{\bydef}{\,\stackrel{\mbox{\tiny\textnormal{\raisebox{0ex}[0ex][0ex]{def}}}}{=}\,} 
\newcommand{\bx}{\bar{x}}
\newcommand{\tx}{\tilde{x}}
\newcommand{\ta}{\tilde{a}}
\newcommand{\tp}{\tilde{p}}
\newcommand{\tb}{\tilde{b}}
\newcommand{\ba}{\bar{a}}
\newcommand{\bb}{\bar{b}}
\newcommand{\bp}{\bar{p}}
\newcommand{\bU}{\bar{U}}
\newcommand{\dagA}{A^\dagger}
\newcommand{\hv}{\widehat{v}}

\newcommand{\cB}{\mathcal{B}}

\newcommand{\cO}{\mathcal{O}}
\newcommand{\cR}{\mathcal{R}}
\newcommand{\cS}{\mathcal{S}}

\newcommand{\pp}{\tfrac{\pi}{2}}

\title{
	Global dynamics in nonconservative nonlinear Schr\"odinger equations
 }

\author[a]{	Jonathan~Jaquette}
\affil[a]{\small
	Department of Mathematics and Statistics, Boston University,
Boston, MA 02215, USA. 
}

\author[b]{	Jean-Philippe~Lessard \thanks{Email address of corresponding author:  	\texttt{jp.lessard@mcgill.ca} }}
\affil[b]{\small
		Department of Mathematics and Statistics, McGill University,
	Montreal, QC H3A 0B9, Canada
}

\author[c]{Akitoshi~Takayasu}
\affil[c]{\small
	Faculty of Engineering, Information and Systems, University of Tsukuba,
	Tsukuba, Ibaraki 305-8573, Japan
}

\begin{document}

\maketitle

\begin{abstract} 
In this paper, we study the global dynamics of a class of nonlinear Schr\"odinger equations using perturbative and non-perturbative methods. We prove the semi-global existence of solutions for initial conditions close to constant. That is, solutions will exist for all positive time or all negative time.  The existence of an open set of initial data which limits to zero in both forward and backward time is also demonstrated. This result in turn forces the non-existence of any real-analytic conserved quantities. For the quadratic case, we prove the existence of two (infinite) families of nontrivial unstable equilibria and prove the existence of heteroclinic orbits limiting to the nontrivial equilibria in backward time and to zero in forward time. By a time reversal argument, we also obtain heteroclinic orbits limiting to the nontrivial equilibria in forward time and to zero in backward time. The proofs for the quadratic equation are computer-assisted and rely on three separate ingredients: an enclosure of a local unstable manifold at the equilibria, a rigorous integration of the flow (starting from the unstable manifold) and a proof that the solution enters a validated stable set (hence showing convergence to zero).
\end{abstract}

\bigskip


\bigskip
\bigskip
\centerline{{\bf AMS subject classifications}}
\medskip
\centerline{
	35B40,	
	35Q55  	
	37K99  	
	37M21  	
} 

%


\section{Introduction}

Consider the  class of nonlinear Schr\"odinger (NLS) equations given by
\begin{equation} \label{eq:NLS}
-i  u_t = 
\triangle u + u^{p}  f(u) 
\end{equation}
for integer $ p \geq 2$, real-analytic $f: \C\to \C$ with $ f(0) \neq 0$,  and  initial data $u(0,\cdot) : \mathbb{T}^d \to \C$.
That is, $f$ may be written as a power series in $u$ and $u^*$, its complex conjugate.%
\footnote{We have opted to use $u^*$ to denote the complex conjugate, and, following some notational conventions in computer-assisted proof, have reserved the notation $\bar{u}$ to denote a numerically computed approximate solution.}
Equation \eqref{eq:NLS} notably does not have the property of gauge invariance. 
That is to say, if $ v = e^{i \theta} u$ then 
\[
\triangle v + v^{p}  f(v) \neq 
e^{i \theta}
\big( 
\triangle u + u^{p}  f(u)
\big)
\]
for generic $ \theta \in \R$.  
Moreover \eqref{eq:NLS} does not have an obvious Hamiltonian structure  so often present in the study of nonlinear Schr\"odinger equations. 

While the local well-posedness theory for \eqref{eq:NLS} is well established,  the global well-posedness and general dynamical structure of these equations are less understood.  
In the present paper, we use perturbative and non-perturbative methods to study the dynamics of \eqref{eq:NLS}.  
For the class of equations in \eqref{eq:NLS} we prove semi-global existence of close to constant initial data (see Theorem \ref{thm:GlobalExistenceIntro}).  That is, solutions will exist for all positive time or all negative time.
Furthermore, we prove the existence of an open set of initial data which limits to $0$ in both forward and backward time (see Theorem \ref{thm:OpenHomoclinics}). This result in turn forces the non-existence of any real-analytic conserved quantities (see Theorem \ref{thm:NoConservedQuantities}) in stark contrast to the case of Hamiltonian NLS.   

We then study in depth the  case of \eqref{eq:NLS} with a  quadratic nonlinearity, 
	\begin{align} \label{eq:NLS_Quad}
-i u_t &= 
\triangle u + u^2 
\end{align}
for $x \in \T^1 = \R / \Z$.  
In Theorem \ref{thm:Equilibria} we prove  the existence of   two (infinite) families of  nontrivial equilibria. These equilibria are all unstable, and in Theorem \ref{thm:Heteroclinics} we prove the existence of heteroclinic orbits limiting to the nontrivial equilibria in backward time and to zero in forward time. 
By a time reversal argument, we also obtain heteroclinic orbits limiting  to the nontrivial equilibria in forward time and to zero in backward time. 
All of these dynamics trivially extend to \eqref{eq:NLS_Quad} posed on $\T^d$. 
 
While Hamiltonian NLS has attracted considerable study from physicists and mathematicians alike, 
non-Hamiltonian NLS has also garnered considerable interest over the past several decades. 
Beyond their intrinsic appeal, NLS with non-gauge invariant terms have been used to study asymptotic behavior about planar waves to the Gross-Pitaevskii equation \cite{gustafson2006scattering}, 
and Raman amplification in a plasma \cite{colin2009stability,hayashi2013system}.
Equations  such as \eqref{eq:NLS_Quad} also arise as toy models of the NLS with an external electric field \cite{leger2018global,leger2019_3d}. 

One notable feature of NLS without gauge invariance is that their local well-posedness theory extends to negative Sobolev spaces. 	
The landmark work  \cite{kenig1996quadratic}  studied $ \partial_t u = i \triangle u + N_j(u,u^*)$ for nonlinearities $ c_1 u^2$ and  $c_2 |u|^2$ and $c_3 (u^*)^2$ with $ c_j \in \C$. They showed that the IVP with nonlinearities  $u^2$ and $(u^*)^2$ is locally well-posed in $H^s(\R)$ for $s > -3/4$ and in $H^s(\T)$ for $ s > -1/2$. 
However for the nonlinearity  $|u|^2$ they only showed local well-posedness for $H^s(\R)$ with $ s>-1/4$. 
Indeed, the structure of the nonlinearity and not just its degree often plays an important role, an observation enounced in   \cite{staffilani1997quadratic} where the local well-posedness with nonlinearity $(u^*)^2 $ for data in $H^s(\R^2)$ with $s> -1/2$ was proven. 
Achieving a sharp result,  in \cite{bejenaru2006sharp} the nonlinearity $u^2$ is shown to be locally well-posed for initial data in $  H^s(\R)$ for $ s \geq -1$ and ill-posed for $ s <-1$. 
For further references we refer the reader to  \cite{kishimoto2019remark}, which itself studies the problem of ill-posedness for more general nonlinearities and spatial domains.  

When the spatial domain is $\R^d$, there is a considerable literature on the global existence and scattering of small initial data, for which we do not attempt to provide a comprehensive review.  
Regarding just the quadratic NLS on $\R^3$, we briefly mention the work of \cite{hayashi2000quadratic,germain2009global} and for the addition of a potential we refer the reader to  \cite{leger2018global,pusateri2020bilinear}. 

However when the spatial domain is  $\T^d$ it is harder to prove global existence, even for  Hamiltonian NLS. 
Any dynamical behaviour that can exist in the spatially homogeneous dynamics  carries over to the PDE on $\T^d$ as a complex one-dimensional subsystem. 
If the NLS is not gauge invariant, as in \eqref{eq:NLS_p}, one may obtain an explicit formula for arbitrarily small initial data which blow-ups in finite time, cf \eqref{eq:PpowerSolution}. 
In  \cite{oh2012blowup}, for example, they show finite-time blowup in $i u_t + \triangle u = \lambda |u|^p$ on $ \T^d$ for all $ 1<p<\infty$  simply based on a condition on the phase of the initial data's zero-Fourier coefficient. 
In short, to show global existence of a solution on $\T^d$, assuming smallness of the initial data is by no means sufficient. 

To obtain a global existence result, in Section \ref{sec:CenterManifold} we perform a center manifold analysis about the $0$ equilibrium to \eqref{eq:NLS}. 
The linearization about this equilibrium has one zero eigenvalue (associated with the homogeneous dynamics) and infinitely many imaginary eigenvalues. 
Unlike the gauge invariant NLS, the spatially homogeneous dynamics of \eqref{eq:NLS} does not support planar waves, but instead admits a singular foliation of homoclinic orbits (see Figure \ref{fig:HomogeneousDynamics}). 
By pairing Gr\"onwall type estimates with a detailed analysis of the homogeneous dynamics, 
we show that initial data which is sufficiently close to a constant will exist globally forward or backward in time.

Before stating our results further, let us first fix some definitions. 
For $ k = ( k_1, \dots k_d) \in \Z^d$  let $ | k| \bydef  |k_1| + \dots |k_d|$. 
We define the following norms. 

\begin{definition}
	For $\omega \in \R^n$, the torus $\T^d = \R^d / \tfrac{2 \pi}{\omega} \Z^d$, a number $\nu \ge 1$, and $u \in L^1(\T^d)$ we define a norm  
	\begin{align}
	\| u \| &= \sum_{k \in \Z^d} | 		\hat{u}(k) | \nu^{|k|},
	&
	\hat{u}(k) &\bydef \frac{1}{\mathrm{vol}(\T^d)}\int_{\T^d} u(x) e^{- i \omega k x} dx .
	\end{align}
	We define the weighted Wiener algebra $ A_\nu( \T^d) \subseteq L^1(\T^d)$ to be the set of functions $u$ with $\| u \| < \infty$. 
\end{definition}
\begin{definition} \label{def:ell_nu^1}
	Given a (multi-indices) sequence $\phi=(\phi_k)_{k \in \Z^d}$ of complex numbers, and
	for $ \nu \geq 1$ we define the analytic norm 
	\[
	\| \phi \|  \bydef   \sum_{k \in \Z^d} |\phi_k| \nu^{|k|}.
	\]
	Denote by $\ell_{\nu,d}^1$ to be the set of sequences $\phi$ with $\| \phi \|<\infty$.

\end{definition}
The Fourier transform defines an isometric isomorphism between $ A_\nu( \T^d) $ and $\ell_{\nu,d}^1$. 
We further note that $ \mathrm{Lip}(\T^d) \subseteq A_1( \T^d) \subseteq C^0 ( \T^d)$, and if $\nu>1$ then the functions in $ A_\nu( \T^d) $ are real analytic. 

We obtain a cleaner result of global existence in the case of a pure power nonlinearity $f \equiv 1$, that is the equation 
\begin{align} 
\label{eq:NLS_p}
- i u_t &= \triangle u + u^p
\end{align}
for $u(0) \in A_\nu(\T^d)$. 
In particular, we show that arbitrarily large, close to constant, initial data will exist globally forward or backwards in time, with explicit hypotheses on the initial data for which the theorem applies. 
	\begin{theorem}\label{prop:PowerGlobalExistence}
		Fix $ p \geq 2$, $d \geq 1$, $ \nu \geq 1$ and define   $ C_p = \exp \left\{ - \frac{\pi}{2} \frac{2^p-p-1}{p-1} \right\}$. 
		Let $ z_0 \in \C$ 	 and fix $ u_0 \in  A_\nu( \T^d)$
		satisfying 
		$	\| u_0 \| < C_p  |z_0| $. 
		Let $u(t)$ be the solution  to \eqref{eq:NLS_p} with initial data $ u(0,x) = z_0 + u_0(x)$ and define $ \vartheta = (p-1) Arg(z_0) \mod 2 \pi$.  
		\begin{itemize}
			\item 		If $
			0 \leq \vartheta \leq \pi $, then the solution $u(t)$ exists for all $ t \geq 0$ and $ \lim_{t \to + \infty}  u(t) =0$. 
			\item  If $ \pi \leq \vartheta \leq 2\pi $, then the solution $u(t)$ exists for all $ t \leq 0$ and $ \lim_{t \to - \infty}  u(t) =0$. 
		\end{itemize}
	
	Furthermore let $ \zeta(t)$ denote the solution to $  \dot{z} = i z^p$ with the initial condition $\zeta(0) = z_0$.  If $ 	0 \leq \vartheta \leq \pi $ (respectively if $
	\pi \leq \vartheta \leq 2\pi $) then 
	\[
	\left| \frac{\zeta(t)}{z_0} \right|  \leq  \sqrt[2 (p-1)]{ \frac{1}{1+ t^2 \big| (p-1) z_0^{p-1}  \big|^2 }  } \,, 
	\qquad \qquad 
	\| u(t) - \zeta(t) \|  \leq |z_0| \left| \frac{\zeta(t)}{z_0}\right|^p 
	\]
	for $t \geq 0$ (respectively for $t \leq0$). 
	 
	\end{theorem}

For the general nonlinearity in \eqref{eq:NLS} we obtain a similar result, albeit with a restriction on the overall norm of the initial data. 

	\begin{theorem} 
		\label{thm:GlobalExistenceIntro}
		Consider the  equation \eqref{eq:NLS}; fix $ p \geq 2$,  $ d \geq 1$, $\nu \geq 1$,  $\omega \in \R^d$, $ \T^d = \R^d / \tfrac{2 \pi}{\omega}  \Z^d$ and fix a real analytic function $ f : \C\to \C$ such that $ f(0) \neq 0$. 
		There exists a constant $ \delta  >0$ depending on $p$ and $f$, such that for all $ u_0 \in  A_\nu( \T^d)$ and   $  z_0 \in \C$ satisfying 
		\begin{align}
		| z_0| &< \delta ,& 	\| u_0 \| &< \delta | z_0|^{p-1} ,
		\end{align}
		then (depending on  $ z_0$) the solution $u(t)$ with initial data $ u(0,x) = z_0 + u_0(x)$ will exist for all $ t \geq 0$ (or $t \leq 0$) and converge to $0$ as $t \to + \infty$ ( or as $ t \to - \infty$). 
		 
		Furthermore, let $ \zeta(t)$ denote the solution to $  \dot{z} = i z^p f(z)$ with the initial condition $\zeta(0) = z_0$. Then there exists a constant $K >0$ depending on $z_0$ and $u_0$ such that 
		\[
		| \zeta(t)|  \leq K |t|^{-1/(p-1)} , 
		\qquad \qquad 
		\| u(t) - \zeta(t) \|  \leq K |t|^{-p/(p-1)} 
		\]
		as $ t \to + \infty$ (or as $ t \to - \infty$). 
\end{theorem} 

Moreover, we are able to prove the existence of an open set of initial conditions which are homoclinic to the zero equilibrium. 
 Recall, a solution $u: \R\to A_\nu(\T^d)$ is said to be a {\em homoclinic orbit} to an equilibrium $ \tilde{u} \in A_\nu(\T^d)$ if $ \lim_{ t \to \pm \infty} u = \tilde{u}$. 
 Similarly, a solution $u: \R\to A_\nu(\T^d)$  is said to be a {\em heteroclinic orbit} between equilibria $ \tilde{u}, \tilde{w} \in A_\nu(\T^d)$ if $ \lim_{t \to -\infty} u(t) = \tilde{u}$ and $ \lim_{t \to +\infty} u(t) = \tilde{w}$.

\begin{theorem}  \label{thm:OpenHomoclinics} 
There exists an open set of initial data in   $U \subseteq  A_\nu( \T^d) $ of  homoclinic orbits limiting  to $u\equiv 0$ for the equation \eqref{eq:NLS}.  
\end{theorem}
As a consequence of Theorem \ref{thm:OpenHomoclinics}, any analytic functional which is preserved under \eqref{eq:NLS} must be globally constant, as they would necessarily be constant on a nontrivial open set $U$. For example, if $H: A_\nu(\T^d) \to \R$ is defined as 
\[
 H(u) = \mathrm{Re} \left\{ \int_{\T^d} h( u, u^*, \dots ,\nabla^n u ,\nabla^n u^*  ) \right\}
\]
  for polynomial $h: \C^{2n} \to \C$ and $ \nu >1$, and $H$ satisfies 
\[
H(u(t)) \geq H(u(t+s)) , \qquad \qquad s \geq 0, \; t \in \R
\]
for all solutions $ u$ to \eqref{eq:NLS}, then $H$ is necessarily constant. 
That is to say, there are neither  conserved quantities nor Lyapunov functionals. 
As  the functions in $ A_\nu(\T^d)$ are themselves analytic for $ \nu >1$, we obtain the following theorem.  

\begin{theorem} \label{thm:NoConservedQuantities} 
	If $ X \subseteq \{u: \T^d \to \C \}$ is a Banach space and $C^\omega(\T^d , \C) \hookrightarrow X$ is a dense, continuous embedding, 
	then the only real-analytic functionals 
	$F: X \to \R$  respecting an inequality under \eqref{eq:NLS}  are constant.
\end{theorem}

Our  results may be contrasted with  Hamiltonian NLS, which enjoys conservation of energy, mass, and momentum in various scenarios.  Note also that the cubic NLS in 1D is integrable, with infinitely many conserved quantities. 
There is a rich literature going back two centuries on integrable systems (see survey papers  \cite{kozlov1983integrability,ramani1989painleve,kruskal1997analytic}), and proofs for the non-existence of additional conserved quantities are perpetually harder than proofs which confirm that a specific quantity is conserved. 
For NLS equation without a gauge invariant nonlinearity, there are few results on the existence or non-existence of conserved quantities. 
For the nonlinearity $|u|^p$   conservation of mass fails (there are solutions whose $L^2$ norms blow up in finite time \cite{ikeda2013small}) and it is mentioned in \cite{ikeda2015some} that  conservation of energy is suspected not to hold, however it is unclear if other quantities may be conserved. 
In a system of NLS without gauge invariance, conservation laws are shown to hold only for a specific choice of parameters \cite{hayashi2013system}.   

%

We further note that there exist homoclinic orbits to \eqref{eq:NLS_p} which grow to an arbitrarily large size. 
That is, for any $K \gg1$, there is an open set of global solutions satisfying $ \| u(0)\|  > K$ and $\lim_{t \to \pm \infty} \| u(t)\| =0$. 
This can be seen as a direct corollary to  Theorem~\ref{prop:PowerGlobalExistence} by taking $ z_0 \in \R$.  
Alternatively, it can be seen to follow from Theorem \ref{thm:OpenHomoclinics} by rescaling solutions to \eqref{eq:NLS_p}; if $u_1(t,x)$ solves \eqref{eq:NLS_p} then so does 
\begin{align} \label{eq:RescalingSolutions}
u_n(t,x) &\bydef  n^{2/(p-1)} u_1 \left(n^2t,nx\right) ,\qquad \qquad \forall n \in \Z .
\end{align}
This arbitrarily large finite growth and decay  is not entirely surprising as it largely shadows a phenomenon already present in the spatially homogeneous dynamics. 
It seems a pertinent question indeed to ask what dynamical behavior can \eqref{eq:NLS} exhibit which shadows neither the homogeneous dynamics nor the linear dynamics? 
To narrow our focus, we   restrict our attention to the quadratic NLS in \eqref{eq:NLS_Quad} posed on $\T^1$. 

As is often the case in the study of dynamics, one may turn to numerical investigations for inspiration. 
The first numerical work studying \eqref{eq:NLS_Quad} we are aware of is  \cite{Cho2016}, which was motivated by studying the singularities of the  nonlinear heat equation $u_t = \triangle u + u^2$ in the complex plane of time. 
Equation \eqref{eq:NLS_Quad} results as a limiting case of purely imaginary time. 
After many simulations, Cho et al. reported that generic real initial data appears to converges to zero at a rate of $\cO(1/|t|)$, and conjectured that solutions to \eqref{eq:NLS_Quad}  with \emph{real} initial data, large or small, will exist globally in time.  
For close to constant real initial data, this convergence rate and global existence  is indeed confirmed by our Theorem \ref{prop:PowerGlobalExistence}.

By employing  computer-assisted proofs, we are able to rigorously establish nonperturbative dynamical features of the quadratic NLS without gauge invariance \eqref{eq:NLS_Quad}. 
Following the prime directive for understanding  global dynamics of dynamical systems, in this paper we study equilibria, their local stability, and the connecting orbits between them.  
By using rigorous numerics to prove the existence of a finite number of equilibria and applying the rescaling in \eqref{eq:RescalingSolutions}, we are able to establish the existence of two infinite families of equilibria.  
\begin{theorem}  \label{thm:Equilibria}
	There exist two analytic, spectrally unstable equilibria  $u_1^i$ and $u_1^{ii}$ to \eqref{eq:NLS_Quad}. 
	By the rescaling in \eqref{eq:RescalingSolutions}, these equilibria (and their complex conjugates) generate infinite families of spectrally  unstable equilibria $\{u_n^{i}\}_{n \in \N}$ and $\{u_n^{ii}\}_{n \in \N}$.  
\end{theorem}

\begin{figure}[h]
	\centering
	\includegraphics[width = .48 \textwidth]{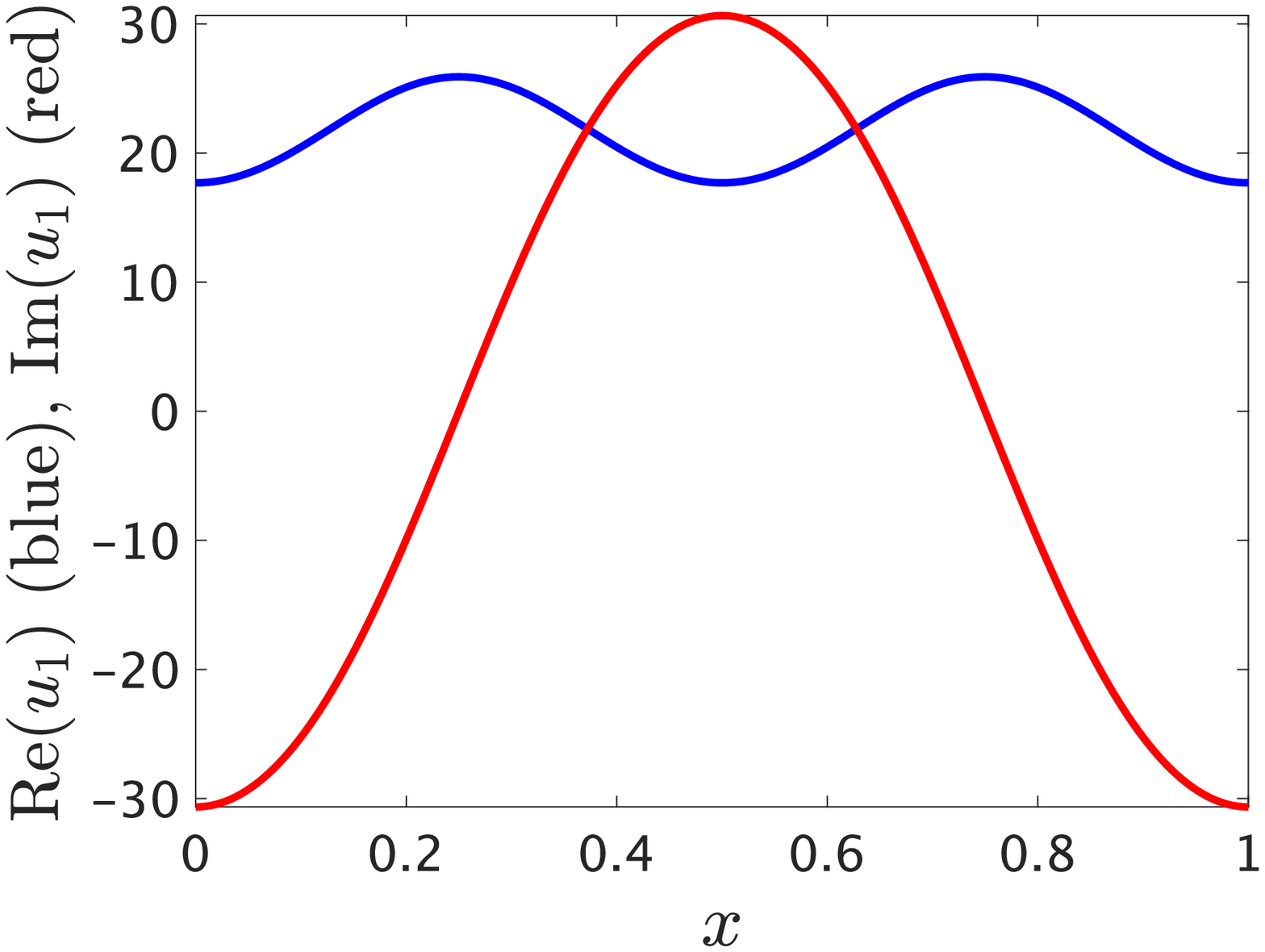}
	\includegraphics[width = .48 \textwidth]{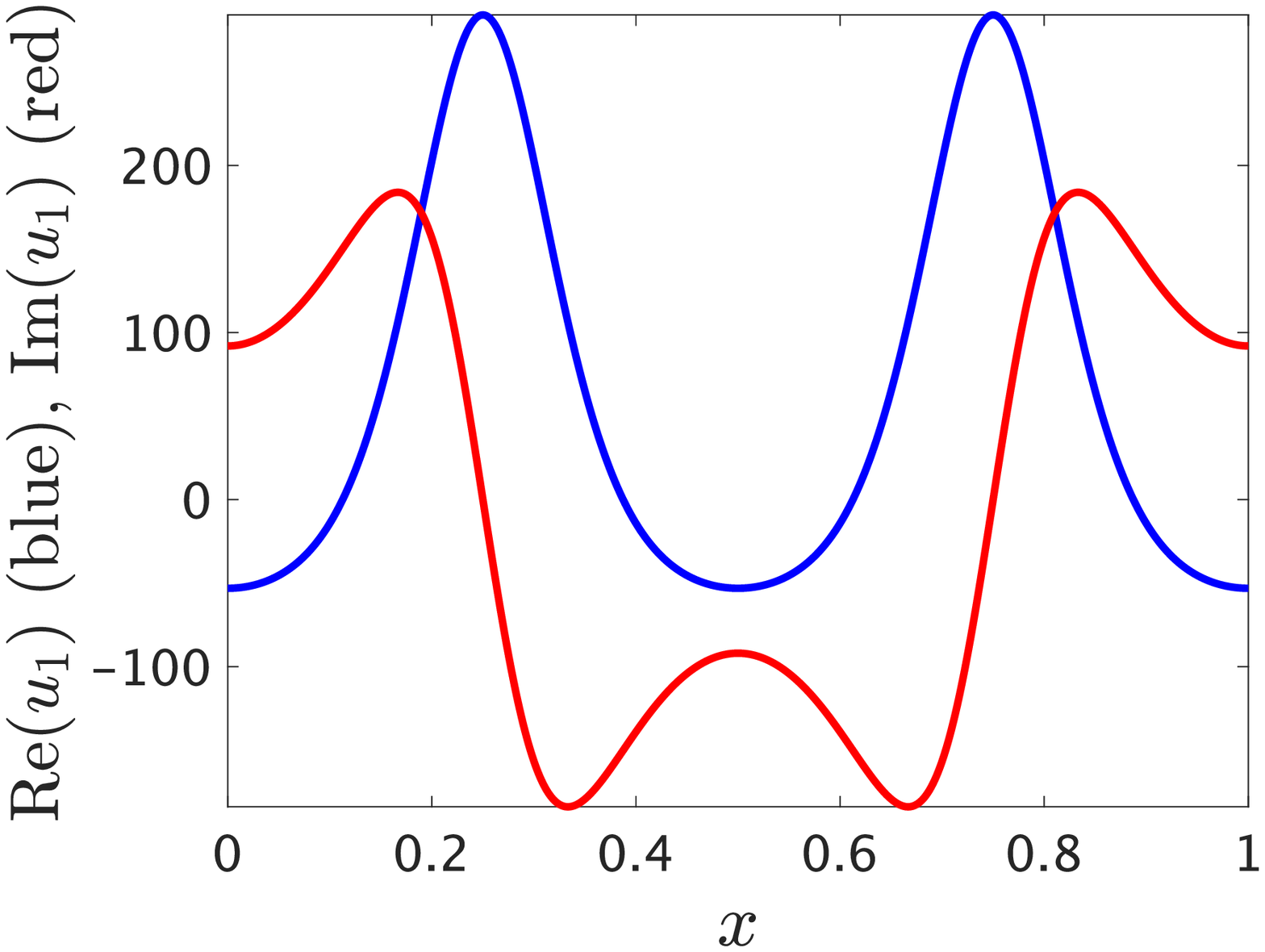}
	\caption{Profiles of two different equilibria: $u^i_1(x)$ (left), $u^{ii}_1(x)$ (right), the existence of which are established via computer-assisted proof.}\label{fig:equilibria}
\end{figure}

We are not aware of prior work demonstrating non-trivial equilibria to NLS without gauge invariance. 
In one sense, these equilibria could be thought of  a special case of standing waves of the form $u(t,x)=e^{i \omega t} \phi(x)$ wherein $ \omega = 0$. 
The existence and stability of standing waves for general $ \omega \in \R $ is well studied in the literature of nonlinear Schr\"odinger equations. 
The existence of such profile solutions can be found by solving an integrable Hamiltonian with two degrees of freedom \cite{bridges1994instability}, and for the stability of spatially periodic $\phi$ we refer to \cite{gallay2007orbital}. 
However, for our non-gauge invariant NLS, it turns out that equilibria are the only standing wave profiles we may expect to find. 
\begin{proposition} \label{prop:NoSeparableSolutions}
	Suppose that $p \geq 2$ and  $u(t,x) = \psi(t) \phi(x)$ is a solution to \eqref{eq:NLS_p}. Then either $ \psi(t)$ or $ \phi(x)$ is a constant function. 
\end{proposition}

The existence of these  unstable equilibria established by Theorem \ref{thm:Equilibria} reveals a hidden complexity in the global dynamics of \eqref{eq:NLS_Quad}, and presents an obstacle in the path of proving  global well-posedness.  In particular, any proof of global existence must contend with the possibility that solutions limit not to the zero equilibria, but instead to one of these equilibria, or even some other nontrivial invariant set.   
In Theorem \ref{thm:Heteroclinics} we prove just that, establishing the existence of several heteroclinic orbits to \eqref{eq:NLS_Quad} limiting from the nontrivial equilibria and to the zero solution, cf Figures \ref{fig:CO_from_minus_1}, \ref{fig:CO_from_1}, and  \ref{fig:CO_pt2_from_1}. 
Using the fact that  if $u(t)$ is a solution to \eqref{eq:NLS_Quad} then so too is $ (u(-t))^*$, we are able to establish the existence of heteroclinic orbits to \eqref{eq:NLS_Quad} limiting from the zero solution and to nontrivial equilibria.

%
%
%
%

\begin{figure}[htbp]
	\centering
	\includegraphics[width = .95 \textwidth]{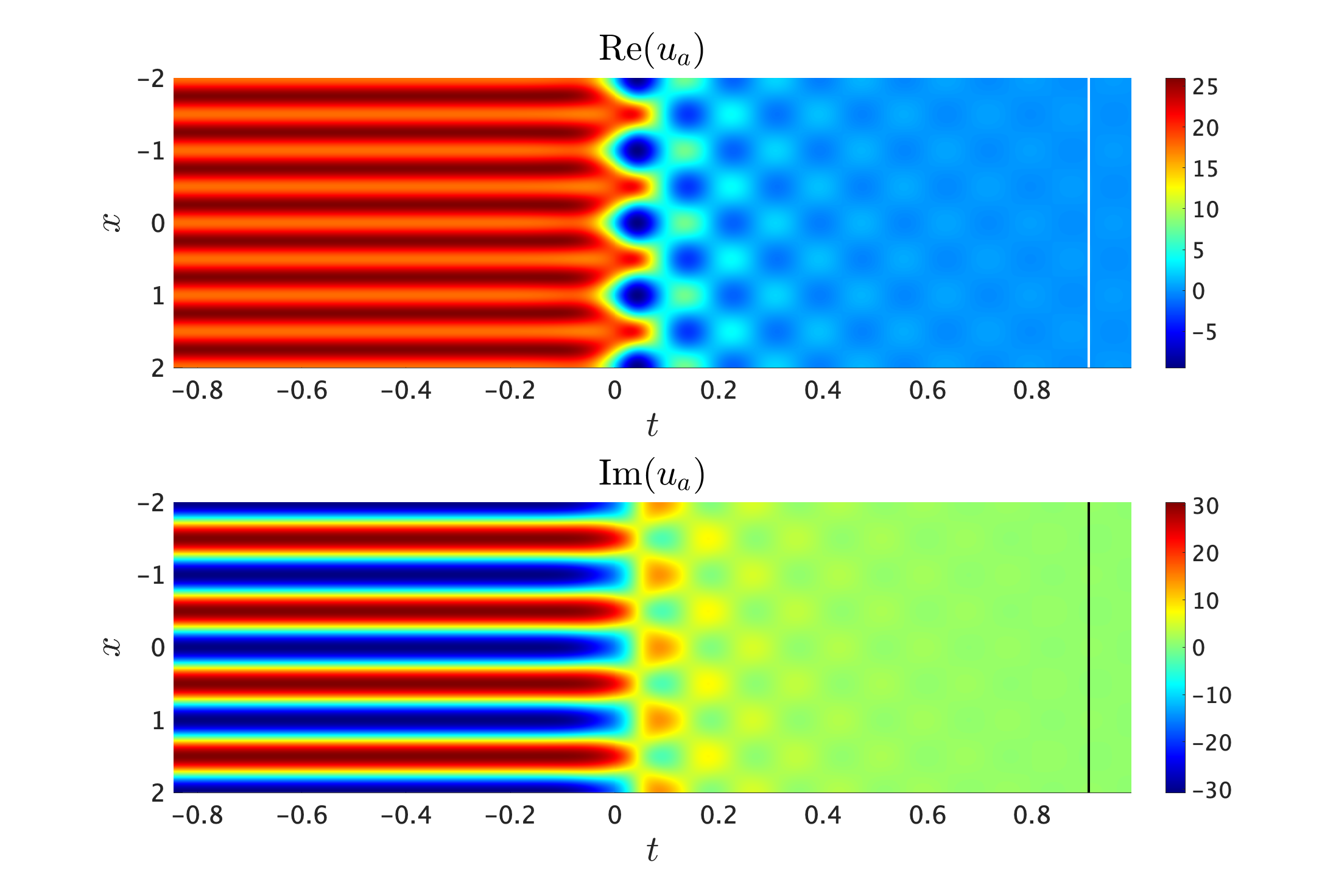}\\[-4pt]
	($a$) The heteroclinic solution $u_a$: connection from $u^i_1(x)$ to $0$.
	\includegraphics[width = .95 \textwidth]{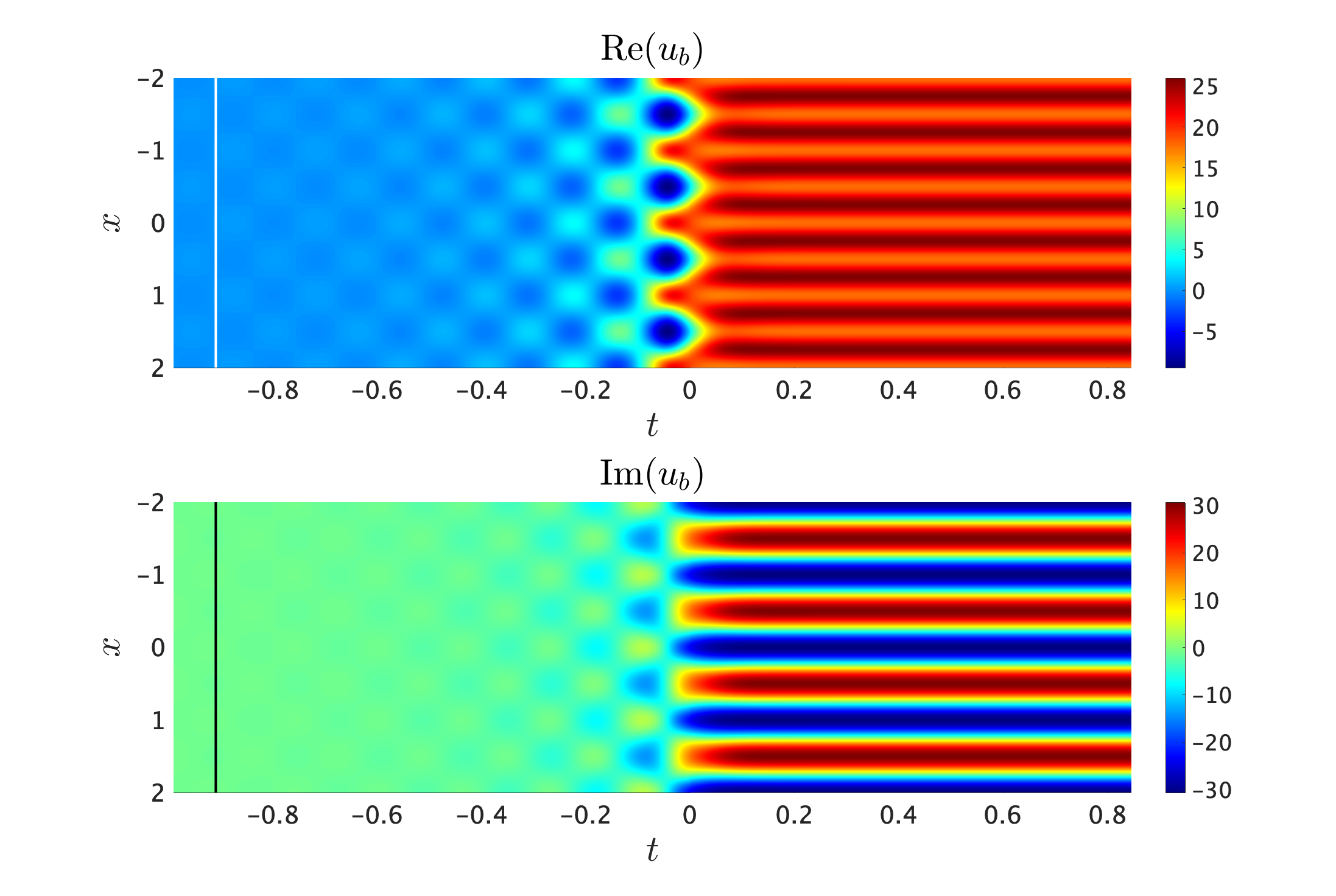}\\[-4pt]
	($b$) The heteroclinic solution $u_b$: connection from $0$ to $u^i_1(x)$.
	\caption{Extracted part of heteroclinic orbits of \eqref{eq:NLS_Quad} between the first equilibrium $u^i_1(x)$ and the zero function, validated with a computer-assisted proof. Here the $x$-variable is extended to the range $x\in [-2,2]$ in order to understand the time evolution of the solution.}\label{fig:CO_from_minus_1}
\end{figure}
 
 \begin{theorem}\label{thm:Heteroclinics}
 	Let $u$ be any one of the equilibria in Theorem \ref{thm:Equilibria}. 
 	There exist heteroclinic orbits  $u_a$ and $u_b$ to \eqref{eq:NLS_Quad} such that 
 	\begin{align*}
 		\lim_{t \to -\infty} u_a(t) &= u,&
 		\lim_{t \to +\infty} u_a(t) &= 0,
 		& & &
 		\lim_{t \to -\infty} u_b(t) &= 0,
		&
		\lim_{t \to +\infty} u_b(t) &=  u,
 	\end{align*}
 	converging exponentially fast to $u$ and algebraically fast to $0$. 
 \end{theorem}

These results may also be seen in parallel to the prodigious line of inquiry seeking to answer one of Bourgain's open problems\cite{bourgain2000problems}:  Do there exist global solutions to the cubic defocusing NLS on $\T^2$ which actually achieve unbounded growth in the higher Sobolev norms in infinite time? 
In the past decade, there have been beautiful results working towards answering this question, in particular showing arbitrarily large finite growth. 
While earlier work \cite{kuksin1997oscillations} had shown arbitrarily large growth for sufficiently large initial data, 
such growth is achieved in \cite{colliander2010transfer,guardia2015growth} for small initial data: for any $s>1$, $\delta >0$ and $ K \gg 1$ there is $0<T<K^c$ and initial data for which $\| u(0)\|_{H^s} <\delta$ and $ \| u(T)\| > K$ for some $c>0$. 
In \cite{guardia2016growth} this result was extended to more general NLS nonlinearities, albeit with an exponential upper bound on the growth time.
We also mention the work \cite{hani2014long} which demonstrates trajectories with unbounded $H^s$ norms in an NLS with a specially prepared cubic type nonlinearity. 

We also mention that Bourgain's paper politely sets numerics at arm's length apart from  pure mathematics. 
\begin{quote}
	\emph{Some of my coworkers believe today's availability of powerful computational means is partly responsible for a declining interest in the often difficult rigorous work.  [...] 
[E]vidence of certain phenomena gathered from extensive computation is often received by the pure mathematician with certain scepticism or dismissed as unreliable. 
		At this point, there does not seem to be such a thing as a truly certified numerical PDE experiment. } \cite{bourgain2000problems}
\end{quote}
While constructive computer-assisted proofs (CAPs) of existence of solutions of PDEs (i.e. truly certified numerical methods) began appearing as early as the turn of the 1990s, they remained rather isolated from mainstream mathematics at the time, were often published in specialized scientific computing journals or in Japanese journals (where most of the original pioneering work of M.T. Nakao appeared). It is therefore not surprising that Bourgain was not aware of this new field. 
Over the intervening years, however, the field of validated numerics and CAPs for PDEs have made incredible strides.  

In fact, the history of CAPs in nonlinear equations began long before the applications to PDEs. Indeed, in the 1960s, functional analytic methods of computer-assisted proof already existed exemplified by the work of Cesari on Galerkin projections for periodic solutions \cite{MR151678,MR173839}. In the field of dynamical systems, important open problems were settled with computer-assisted proofs, e.g. the universality of the Feigenbaum constant \cite{feigenbaum} and the existence of the strange attractor in the Lorenz system \cite{lorenz} (i.e.\ Smale's 14th problem). Other prominent examples outside dynamics are the proofs of the four-colour theorem \cite{fourcolor} and Kepler's densest sphere packing problem \cite{kepler}. We refer the interested to reader to the expository works \cite{MR1420838,MR1849323,Plu01,MR2652784,MR2807595,MR3444942,jay_konstantin_survey} 
and the references therein, for a more complete overview of the field of rigorously verified numerics. 

More specifically in the field of PDEs, the first CAPs may have appeared in the work of M.T. Nakao when he used the computer together with Schauder's fixed point theorem to prove the existence of weak solutions for linear elliptic second order BVPs \cite{MR944817} and nonlinear parabolic equations \cite{MR1146987}. Following this approach, a priori error estimates for finite element approximations in Sobolev spaces were used to prove the existence of solutions of elliptic problems \cite{MR2161437}. Independently, Plum proposed a method based on explicit Sobolev embeddings and eigenvalue bounds to prove existence of solutions of nonlinear elliptic BVPs \cite{MR1151060,MR1182440,MR2019251}. 
A method based on a computational version of the Banach fixed point theorem (a Newton-Kantorovich type theorem) was introduced in \cite{MR1639986} to prove existence and local uniqueness of solutions of BVPs. Newton-Kantorovich type theorems were then used in the study of equilibria \cite{MR2718657,MR3904424,MR3770054} and periodic orbits \cite{AriKoc10,FigLla17,GamLes17,MR2049869} and the ill-posed Boussinesq equation \cite{CasGamLes18}. 
Topological methods (e.g. Conley index, self-consistent a priori bounds) were also introduced to obtain CAPs in parabolic PDEs \cite{MR1838755,MR2049869,MR2136516,MR2329522}. 
We finally mention the work~\cite{CasCorGom16}, where the existence of a family of uniformly rotating solutions of the inviscid surface quasi-geostrophic equation is shown with computer-assistance. While a complete review of the field of rigorously verified numerics in PDEs falls outside the scope of the paper, we refer to the recent book \cite{MR3971222} and the review paper \cite{MR3990999} for more details.

The paper is organized as follows. 
In Section \ref{sec:CenterManifold} we study the general equation \eqref{eq:NLS}, proving 
Theorems \ref{prop:PowerGlobalExistence}--\ref{thm:NoConservedQuantities} and Proposition \ref{prop:NoSeparableSolutions}.
The remainder of the paper studies equation  \eqref{eq:NLS_Quad}, using computer-assisted proofs in particular. 
The proof of Theorem \ref{thm:Equilibria} uses standard methods of verified numerics and is left to Appendix \ref{sec:eigenpairs}. 
Section \ref{sec:UnstableManifold} develops a validated approximation of the (strong) unstable manifolds of these nontrivial equilibria, with rigorous \emph{a posteriori} error bounds. 
In Section \ref{sec:Integrator} we describe our rigorous numerical integrator for solving the initial value problem in \eqref{eq:NLS_Quad}. These results are all combined in Section \ref{sec:FinalTheorem} where we prove Theorem \ref{thm:Heteroclinics}. 
We note that the codes used to produce our computer-assisted proofs are freely available from \cite{bib:codes}, and some of the relevant computational details are listed in Appendix \ref{sec:ComputationalTable}. 
Finally in Section \ref{sec:OpenQuestions} we list several open questions.  


\section{Semi-global existence of close to constant initial data}
\label{sec:CenterManifold}

Let us begin by considering \eqref{eq:NLS_p}, a special case of \eqref{eq:NLS} where $ f(u)=1$, repeated  below  
\begin{align*} 
u_t &= i
\left( 
\triangle u + u^p 
\right)
\end{align*}
for integer $p \geq 2$ and  initial data $u(0,x) : \mathbb{T}^d \to \C$ 
with periodic boundary conditions $ \T^d = \R^d / \tfrac{2 \pi}{\omega} \Z^d$ for $\omega \in \R^d$. 
This has an equilibrium at $ u \equiv 0$ with eigenvalues $-i k^2 \omega ^2$  for $ k \in \Z^d$. 
In particular there is a  single 0-eigenvalue with complex multiplicity 1. 
This eigenvalue is associated with the spatially homogeneous dynamics, an invariant subsystem with internal dynamics given by  
\begin{align}\label{eq:NLS_hom}
	\dot{z} &= i z^p
\end{align}
The phase diagrams of these dynamics are reminiscent of the electric field of a dipole or multipole, see Figure 
\ref{fig:HomogeneousDynamics}, and we refer the reader to \cite{dumortier2006qualitative} for the qualitative dynamics of this normal form.
\begin{figure}[h]
	\centering
	\includegraphics[width=0.32\linewidth]{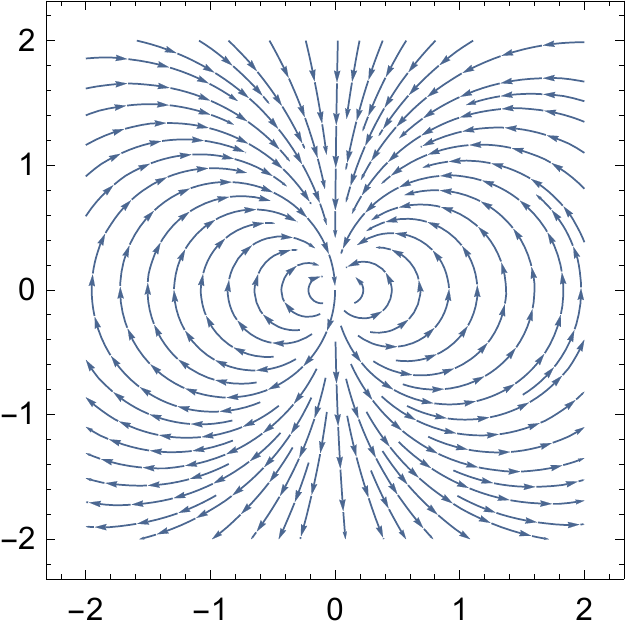} 
	\includegraphics[width=0.32\linewidth]{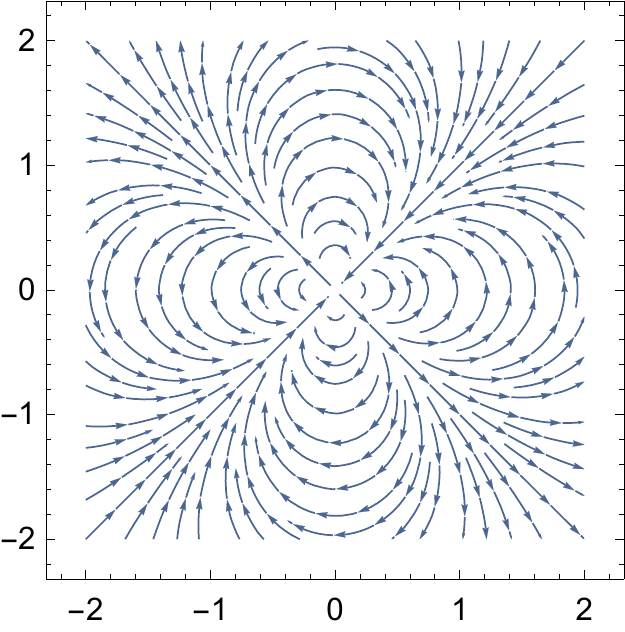}
	\includegraphics[width=0.32\linewidth]{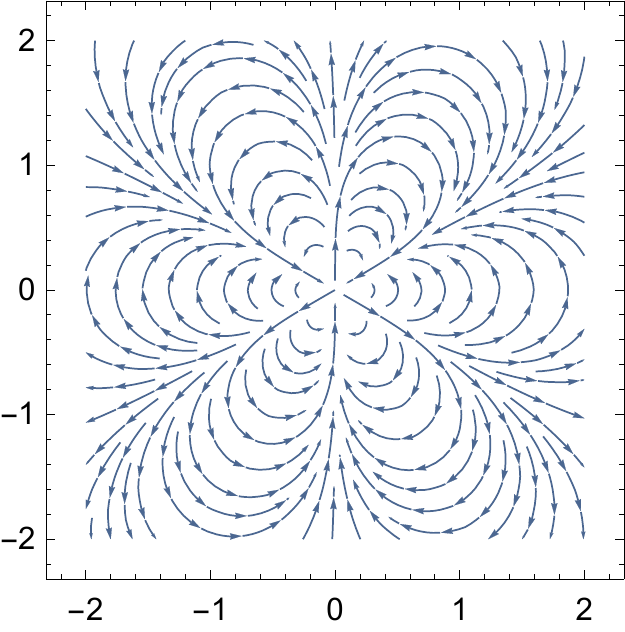}  
	\caption{
		Homogeneous dynamics for \eqref{eq:NLS_hom} with  $ p = 2,3,4$. 
	}
	\label{fig:HomogeneousDynamics}
\end{figure}

Note that \eqref{eq:NLS_hom} is a separable ordinary  differential equation, which we may solve exactly. 
Indeed,  for the branch $ 1 = \sqrt[p-1]{1}$ of $ z \mapsto z^{1/(p-1)}\in \C $, the function $ \zeta: \R_+ \to \C$ defined below solves \eqref{eq:NLS_hom} with initial condition $\zeta(0) = z_0$.   
\begin{align} \label{eq:PpowerSolution}
	\zeta(t) &= \frac{  z_0 }{(1 -i (p-1) z_0^{p-1} t)^{1/(p-1)}}.
\end{align}
We use this solution, and more generally the homogeneous solutions to \eqref{eq:NLS}, to provide a non-autonomous blow up of the spatio-temporal dynamics about the equilibrium. 

%
%
%
%
%

\begin{proposition} \label{prop:GeneralBlowUp}
 Suppose that $z:\R_+ \to \C$ solves  $ z' = i z^p f(z)$ and $z(0)\neq0$ for $ f : \C\to \C$ real analytic.  
 There exists a real-analytic function $h:\C^2 \to \C$ such that if  $ u$ solves \eqref{eq:NLS}  and $ \tilde{u}$ is defined by $ u = z(t) + z(t)^p \tilde{u}$, then 
 $\tilde{u}$ satisfies 
\begin{align}\label{eq:GeneralBlowUp}
	\partial_t  \tilde{u} = i \triangle \tilde{u} + i  h(z(t),\tilde{u}).
\end{align}
Furthermore there exist constants $ \delta , K >0$ such that $ | h(x,y)| \leq K |x^p y| $ whenever $ |x|,|y| < \delta$.  
\end{proposition}
\begin{proof}
Taking the time derivative of $u$, we obtain 
	\begin{align*}
	u_t 
	&= z'(t) + p z'(t) z(t)^{p-1} \tilde{u} + z(t)^{p} \tilde{u}_t \\
	&= i z(t)^p f(z) +i p z(t)^{2p-1} f(z) \tilde{u} + z(t)^p \tilde{u}_t.
	\end{align*}
	We expand $u^p f(u)$ from the RHS of \eqref{eq:NLS}, writing $ z=z(t)$, 
	\begin{align}
		u^p f(u) &= (z+ z^p \tilde{u})^p f(z + z^p \tilde{u}) \nonumber \\
		& = \left(
		z^p + p z^{2p-1} \tilde{u} + \sum_{m=2}^p {p\choose m} z^{p +m(p-1)} \tilde{u}^m
		\right)
		f(z+ z^p \tilde{u}). \label{eq:ExpansionUpFu}
	\end{align}
	We may write $ f(z+ z^p \tilde{u}) - f(z) =  h_1(z,\tilde{u})$ for a real-analytic function $ h_1 : \C^2 \to \C$  defined as 
	\[
	h_1(z,\tilde{u}) 
	\bydef 
	\sum_{m+n \geq 1}^\infty \frac{\partial_{z}^m  \partial_{z^*}^nf(z)}{m!n!} 
	\big(z^p \tilde{u}\big)^m \left((z^p \tilde{u})^*\right)^n .
	\]
	Note that there exist constants $ \delta_1 , K_1 >0$ such that $ | h_1(x,y)| \leq K_1 |x^p y| $ whenever $ |x|,|y| < \delta_1$.  
	We may continue expanding $ u^p f(u)$ from \eqref{eq:ExpansionUpFu} as  below  
	\begin{align}
 		u^p f(u) &=  \left(
 		z^p + p z^{2p-1} \tilde{u} 
 		\right)
\left(  		f(z ) +  h_1(z,\tilde{u})  \right)
 		+ 
z^{p} \left( \sum_{m=2}^p {p\choose m} z^{m(p-1)} \tilde{u}^m \right) 		 		f(z+ z^p \tilde{u}).
	\end{align}
	Equating $ u_t = i (\triangle u + u^p f(u))$ and canceling terms we obtain
	\begin{align}
		z^p \tilde{u}_t &= i z^p \triangle \tilde{u} + 
		i
		\left(
		z^p + p z^{2p-1} \tilde{u} 
		\right) h_1(z,\tilde{u}) 
		+ 
		i
		z^{p} \left( \sum_{m=2}^p {p\choose m} z^{m(p-1)} \tilde{u}^m \right) 		 		f(z+ z^p \tilde{u}).
	\end{align}
	Define the real analytic function $h:\C^2 \to \C$ by 
	\[
	h(z,
	\tilde{u})
	\bydef
	\left(
	1 + p z^{p-1} \tilde{u} 
	\right)
	h_1(z,\tilde{u}) 
	+ 
 \left( \sum_{m=2}^{p} {p\choose m} z^{m(p-1)} \tilde{u}^m \right) 		 		f(z+ z^p \tilde{u}).
	\]
	Dividing through by $ z^p$ we obtain 
	\[
	\tilde{u}_t = i \triangle \tilde{u} + i
	 h(z,\tilde{u})
	 .
	\]
	Note further that 
	 there exist constants $ \delta , K >0$ such that $ | h(x,y)| \leq K |x^p y| $ whenever $ |x|,|y| < \delta$.  
\end{proof}

	\begin{remark} \label{rem:Blowup}
		If  $f(u)= 1$ as in \eqref{eq:NLS_p}, then $h_1=0$ and  we have  
		\begin{align} \label{eq:BlowUp}
h(\zeta(t),\tilde{u}) &= 
 \sum_{m=2}^{p}   {p\choose m} \zeta(t)^{(p-1)m} \tilde{u}^m
 .
		\end{align}
		
	\end{remark}


To perform our analysis, we   work with functions with Fourier coefficients in $\ell_{\nu,d}^1$, see Definition \ref{def:ell_nu^1}.  
That is, we define $ a : \R\to \ell_{\nu,d}^1$ to be the Fourier coefficients of $ u$ according to 
	\begin{align}\label{eq:FourierSeries}
u(t, x ) \bydef \sum_{k \in \Z^d} a_k(t) e^{ i k \omega x}.
	\end{align}
	Associated with the multiplication of two functions $u,w:\T^d \to \C$ is the discrete convolution $ * : \ell_{\nu,d}^1 \times \ell_{\nu,d}^1 \to \ell_{\nu,d}^1$. 
	That  is, given two sequences $b=(b_k)_{k \in \Z^d  },c=(c_k)_{k \in \Z^d} \in \ell_{\nu,d}^1$, their discrete convolution $b*c=((b*c)_k )_{k \in \Z^d}$ is given component-wise by
	\[
	(b*c)_k = \sum_{k_1+k_2 = k \atop k_1,k_2 \in \Z^d} b_{k_1} c_{k_2}.
	\]
	We note that $ \ell_{\nu,d}^1$ is a   Banach algebra; if $ b , c \in \ell_{\nu,d}^1$ then $ \|b*c \| \leq \|b \| \|c\|$.  
	Moreover, for $ b \in \ell_{\nu,d}^1$ we recursively define $ b^1 \bydef b$ and $ b^{n} \bydef b^{n-1} * b$. 
	In this manner, if  $u_0:\T^d \to \C$ has Fourier coefficients $ b \in \ell_{\nu,d}^1$,  and $ f(z) = \sum_{n \in \N} d_n z^n$ is an analytic function with coefficients $d_n\in \C$, $n \geq 0$,  then 
	\[
	f(u_0(x)) 
	= \sum_{k \in \Z^d} \tilde{f}(b)_k e^{i \omega k x} , \qquad \qquad \tilde{f}(b) \bydef \sum_{n \in \N} d_n  b^n 
	\] 
Similarly, we are able to naturally extend a real  analytic function $ f: \C \to \C$ to a real analytic function $ \tilde{f} : \ell_{\nu,d}^1 \to \ell_{\nu,d}^1$.  
When it is notationally convenient (ie always) we will write $\tilde{f}$ as $f$.  

Furthermore, let us define the inclusion map $ \iota^0: \C\to \ell_{\nu,d}^1$ onto the zero\textsuperscript{th} Fourier coefficient; for $ z \in \C$  define $( \iota^0(z))_k $ as equal to $z$ if $k = 0$, and $0$ otherwise.

\begin{remark} \label{rem:FourierDifferentialEquation}
 Suppose $u$ is a solution to \eqref{eq:NLS} with time varying Fourier coefficients $a(t) \in \ell_{\nu,d}^1$. 
 Then for $z(t)$ as in Proposition \ref{prop:GeneralBlowUp} and $ \tilde{a}(t) \in \ell_{\nu,d}^1 $ denoting the time varying Fourier coefficients of $ \tilde{u}$ with $ a(t) =\iota^0(z(t)) + z(t)^p \tilde{a}(t)$, it follows that 
 	\[
 ( \partial_t \tilde{a})_k = - i k^2 \omega^2 \tilde{a} + i
 h(z(t),\tilde{a}),
  \qquad \qquad k \in \Z^d .
 \]
\end{remark}

To illustrate our approach for proving global existence, and to obtain sharper results, we first treat the case of equation \eqref{eq:NLS_p}. 
Note in Figure \ref{fig:HomogeneousDynamics} that each equilibrium has ``stable sectors' of solutions approaching zero with monotonically decreasing norm. 
From the explicit solution $ \zeta(t)$ in \eqref{eq:PpowerSolution} with $ z_0 = r_0 e^{i \theta_0}$ we may write 
\begin{align}\label{eq:EpsilonEstimate}
	|\zeta(t)|^{p-1} =  
	\left| 
	\frac{ z_0^{p-1} }{1 -i (p-1) z_0^{p-1} t} 
	\right| 
	=
	\left| 
	\frac{ r_0^{p-1} }{
	1 + t (p-1) r_0^{p-1} \big(
	\sin (p-1) \theta_0 - i \cos (p-1) \theta_0 
	\big)
} 
	\right| .
\end{align}
Hence, if $ \sin (p-1) \theta_0 \geq 0$ then $|\zeta(t)|$ monotonically decreases. 
We define these stable regions   below. 

\begin{definition}
	For any $ \rho_0 > 0$ define the region 
	\begin{align}\label{eq:InvariantBall}
	B(\rho_0) \bydef 
	\{
	z \in \C :  
	0 \leq (p-1) Arg(z) \mod 2 \pi \leq \pi ;
	| z| \leq \rho_0
	\} .
	\end{align}
\end{definition}

For the nonlinearity given in \eqref{eq:NLS_p}, we may obtain an explicit region within which points converge to $0$.

\begin{theorem}\label{thm:HomoclinicBlowup} 
	Consider \eqref{eq:NLS_p} with  $ p \geq 2$. 
	Fix $ 0 <  \rho_0, \rho_1$ and define the set 
	\begin{align} \label{eq:BallOfFunction}
	\cB(\rho_0,\rho_1) &\bydef 
	\left\{
 \phi + 	\iota^0(z_0)  \in \ell_{\nu,d}^1 \; 
	:\; 
	z_0 \in B(\rho_0)  \subseteq \C;
	\| \phi \|  \leq \rho_1 |z_0|^p
	\right\}.
	\end{align}
	Define $ P(r,\rho_0)$ as
	\begin{align} \label{eq:Pdef}
	P(r,\rho_0) &\bydef 
	\sum_{m=2}^{p}   {p\choose m}
	\left(
	r \rho_0^{p-1}
	\right)^{m-1} 
.
	\end{align}
	If there exists some $ r >0$ such that 
	\begin{align} \label{eq:RadiiExponential}
	\rho_1  \exp \left\{  \frac{\pi}{2} \frac{ P(r,\rho_0)}{(p-1)} \right\} < r,
	\end{align}
	then solutions of points  $ a(0) = 	  \phi + 	\iota^0(z_0)\in \cB(\rho_0,\rho_1) $ under \eqref{eq:NLS_p} will exist for all positive time,   converge to zero, and satisfy $ \| a(t) - \iota^0(\zeta(t)) \| \leq r | \zeta(t)|^p$ for $\zeta(t)$ in \eqref{eq:PpowerSolution}.
\end{theorem}
\begin{proof}
	
	Using the Fourier ansatz as discussed in Remark \ref{rem:FourierDifferentialEquation} and using the explicit form of $h$ given in  Remark \ref{rem:Blowup},  we obtain a differential equation on the Fourier coefficients. 
	\begin{align}  
	\partial_t \tilde{a}_k &= 
	- i k^2 \omega^2 \tilde{a}_k + i \sum_{m=2}^{p}   {p\choose m} \zeta(t)^{(p-1)m} ( \tilde{a}^m)_k ,
	\label{eq:FourierODE}
	\end{align}
	for $\zeta$ defined in \eqref{eq:PpowerSolution}.
		Define $ \tilde{\phi} = z_0^{-p} \phi$. 
	By variation of constants 
	we obtain:  
	\[
	\tilde{a}_k(t) = 
	e^{-i k^2 \omega^2 t} \tilde{\phi}_k
	+ \int_0^t  e^{-i k^2 \omega^2 (t-s)} 
	i \sum_{m=2}^{p}   {p\choose m} \zeta(s)^{(p-1)m}
	( \tilde{a}^{m}  )_k  ds .
	\]
	Note from  \eqref{eq:EpsilonEstimate} that for $ z_0 \in B_0(\rho_0)$ we have 	$ | \zeta(t)|^{m(p-1)} \leq \rho_0^{(m-2)(p-1)} |\zeta(t)|^{2(p-1)}$ for $ m \geq 2$. 
	Moreover, for $ Arg(z_0)=\theta_0$ 
	 the norm of $ \zeta(t)$ is maximized when $0 = \sin (p-1) \theta_0 $,  hence 
\begin{align} \label{eq:zetaBound}
		|\zeta(t)|^{2(p-1)} 
	\leq  \frac{|z_0^{p-1}|^2 }{1 + |(p-1) z_0^{p-1}|^2 t^2} .
\end{align}
	For $ r \geq \| \tilde{\phi} \| $ define $ T = \sup  \{ t  \geq 0 :  \|\tilde{a}(t)\| \leq r \}$. 
	Taking norms in the variation of constants formula  and applying Banach algebra estimates, we obtain  for $ t \in [ 0, T]$ that 
	\begin{align*}
	\| \tilde{a}(t) \| & \leq \| \tilde{\phi}\| + \int_0^t  \sum_{m=2}^{p}   {p\choose m} |\zeta(s)|^{m(p-1)} \|\tilde{a}(s)\|^{m}    ds \\
	& \leq \| \tilde{\phi}\| + \int_0^t  \sum_{m=2}^{p}   {p\choose m} |\zeta(s)|^{2(p-1)}
	\rho_0^{(m-2)(p-1)}  \|\tilde{a}(s)\|r^{m-1}   ds .
	\end{align*}
	Then for $P(r,\rho_0)$ defined in \eqref{eq:Pdef} we have 
	\begin{align*}
	\| \tilde{a}(t) \| &\leq  \int_0^t  \rho_0^{-(p-1) }
	P(r,\rho_0) |\zeta(s)|^{2(p-1)} \|\tilde{a}(s)\| ds .
	\end{align*} 
	In anticipation of applying Gr\"onwall's inequality we compute 
	\begin{align*}
	\int_0^t  \rho_0^{-(p-1) }
	P(r,\rho_0) |\zeta(s)|^{2(p-1)}  d s 
	&\leq \frac{ P(r,\rho_0)}{\rho_0^{p-1}}  \int_0^t  \frac{|z_0^{p-1}|^2 }{1 + |(p-1) z_0^{p-1}|^2 s^2} d s \\
	&\leq \frac{ P(r,\rho_0)}{(p-1)}     \arctan( |(p-1)z_0^{p-1}| t )   \\
	& < \frac{\pi}{2} \frac{  P(r,\rho_0)}{(p-1)}   .
	\end{align*}
	Thus from Gr\"onwall's inequality we obtain:
	\[
	\| \tilde{a}(t) \| < \| \tilde{\phi}\|  \exp \left\{  \frac{\pi}{2} \frac{  P(r,\rho_0)}{(p-1)} \right\} 
	\leq
	\rho_1 \exp \left\{  \frac{\pi}{2} \frac{  P(r,\rho_0)}{(p-1)} \right\} 
	\qquad  \qquad \forall t \in [0,T].
	\]
	Assuming that \eqref{eq:RadiiExponential} holds, then $ \| \tilde{a}(t) \| < r$ for all $ t \in [0,T]$, whence $ T = + \infty$.  
	Hence solutions $u$ to \eqref{eq:NLS_p}  with Fourier coefficients $ a(t) = \iota^0(\zeta(t)) + \zeta(t)^p \tilde{a}(t)$ exist for all positive time, and its Fourier coefficients satisfy $\| a(t) - \iota^0(\zeta(t))\| \leq r |\zeta(t)|^p $, with $ \zeta \to 0 $ as $ t \to \infty$.

\end{proof}

\begin{proof}[Proof of Theorem \ref{prop:PowerGlobalExistence}] 
	Letting $\phi$ denote the Fourier coefficients of $ u_0$. The result follows from Theorem \ref{thm:HomoclinicBlowup} with a particular choice of constants. 
	Note that as $ \| \phi  \| < C_p |z_0|$, we may choose some $\rho_1 < C_p |z_0|^{-(p-1)}$ such that $\| \phi   \| \leq \rho_1 | z_0|^p$.
	Fix $ \rho_0 = |z_0|$ and for $ \vartheta =  Arg(z_0) \mod 2 \pi$ suppose $ 0 \leq \vartheta \leq \pi$, hence  $ \iota^0 (z_0) + \phi \in \cB(\rho_0,\rho_1)$. 
	If we choose $ r = 1/\rho_0^{p-1}$ then $ P(r,\rho_0) = P(1,1) = 2^p - (p+1)$.
	Thus   condition 
	\eqref{eq:RadiiExponential} reduces to  
	\[
\rho_1 < 
	 \exp \left\{ - \frac{\pi}{2} \frac{P(1,1)}{p-1} \right\} r = \frac{C_p}{\rho_0^{p-1}}
	\] 
	which is indeed satisfied for our choice of $\rho_1$. By Theorem \ref{thm:HomoclinicBlowup} the Fourier coefficients $a(t)$ of $ u(t)$ exist for all $t\geq 0$  and $ \lim_{t \to + \infty} a(t) = 0$. 
	The bound on the rate of decay follows from the estimate in \eqref{eq:zetaBound}, combined with the estimate $\| a(t) - \iota^0(\zeta(t)) \| \leq r | \zeta(t)|^p $ from Theorem \ref{thm:HomoclinicBlowup}. 
	The result for $ \pi \leq \vartheta \leq 2 \pi$ follows with parity. 
	
\end{proof}
\begin{remark} \label{rem:StableSetP2}
	In the case $p=2$, then inequality  \eqref{eq:RadiiExponential} becomes $	\rho_1  \exp \left\{  \tfrac{\pi}{2} \rho_0 r \right\} < r$. 
\end{remark}

To generalize this to nonlinearities of the form in  \eqref{eq:NLS}, it is generally not possible to write down explicit solutions to the homogeneous equation $z'=iz^p f(z)$. 
Instead, we use the estimate below, which shows  that solutions in the stable sectors converge to zero sufficiently fast.

\begin{proposition} \label{prop:Integrability}
For real analytic $ f: \C \to \C$ with $ f(0) \neq 0$ let the flow $ z(t, z_0) = r(t,z_0) e^{i \theta(t,z_0)}$ denote the solution to the differential equation
 $ \dot{z} = i z^p f(z)$ with initial condition $ z_0$. Define the set 
	\begin{align}\label{eq:GeneralStableSector}
		B_0(\rho_0) &\bydef
	\left\{ 
	z_0 \in \C: 
	|z_0|\leq \rho_0, \, \tfrac{d}{dt} r(0,z_0) \leq 0
	\right\}	.
	\end{align} 
	For sufficiently small  $ \rho_0 > 0$, there exists a constant $ C >0$ such that all 
	solutions $ z(t)$  
	with initial condition $ z(0) = z_0 \in 		B_0(\rho_0)$ satisfy 
	\begin{align} \label{eq:pIntegrability}
	\int_0^\infty  
	|z(t)|^p dt \leq C 	|z_0| .
	\end{align}
\end{proposition}
\begin{proof}
	We first write the differential equation $ \dot{z} = i z^p f(z)$ in polar coordinates. 
	Writing $z = r e^{i \theta}$, the equations $ z' = i z^p f(z)$ and $ (z^*)' = (i z^p f( z))^*$ become
	\begin{align*}
	\dot{r} e^{ i \theta } +  r \dot{\theta} e^{i \theta}
	 &= i r^p e^{i p \theta } f( r e^{i \theta} ) 
	 \\
	\dot{r} e^{- i \theta } -  r \dot{\theta} e^{-i \theta} 
	&= -i r^p e^{-i p \theta } (f( r e^{i \theta} ))^* .
	\end{align*}
	From this we obtain expressions for $ \dot{r} $ and $\dot{\theta}$ below
	\begin{align*}
	\dot{r} &= \frac{i r^p}{2} \left( e^{i (p-1) \theta} f(r e^{i \theta })
	-
	e^{-i (p-1) \theta} (f( r e^{i \theta} ))^*  
	\right)\\
	\dot{\theta} &= \frac{r^{p-1}}{2} \left( e^{i (p-1) \theta} f(r e^{i \theta })
	+
	e^{-i (p-1) \theta} (f( r e^{i \theta} ))^*  
	\right).
	\end{align*}
	As a real analytic function, we may write $f$ as a convergent power series $ f(z) = \sum_{m,n=0}^\infty b_{m,n} z^{m} (z^*)^n$	 for coefficients $ \{ b_{m,n}\}_{m,n\in\N}$, thus obtaining 
	\begin{align*}
	\dot{r} &= - r^p \sum_{m,n=0}^\infty  r^{m+n} \left(\frac{ b_{m,n} e^{i (m-n+p-1) \theta} 
		-
b_{m,n}^*
		e^{-i (m-n+ p-1) \theta} }{2i}
	\right)\\
	\dot{\theta} &= r^{p-1}
	\sum_{m,n=0}^\infty  r^{m+n}
	\left( \frac{b_{m,n} e^{i (m-n+p-1) \theta} 
		+
		b^*_{m,n}
		e^{-i (m-n+ p-1) \theta} }{2}
	\right).
	\end{align*}
	Let us write $ b_{m,n} = R_{m,n} e^{ i \vartheta_{m,n}}$. 
	Using the identities  for $\sin$ and $\cos$ we obtain 
	\begin{align}
	\dot{r} &= - r^p \sum_{m,n=0}^\infty  r^{m+n} R_{m,n}
		\sin \big(  (m-n+p-1) \theta + \vartheta_{m,n} \big) 
		\label{eq:RdifferentialEquation}
	\\
	\dot{\theta} &= r^{p-1}
	\sum_{m,n=0}^\infty  r^{m+n}  
	R_{m,n}
	\cos \big(  (m-n+p-1) \theta + \vartheta_{m,n}  \big) .
		\label{eq:ThetadifferentialEquation}
	\end{align}
	Thus, we have a nice expression for the dynamics on the space $ \R\times \bS^1$. 
	Let us define the two spaces $B_1,B_2$ such that $B_1 \cup B_2 = B_0(\rho_0)$ under the identification $ \C \equiv \R_+ \times \bS^1$, see Figure \ref{fig:Blowup}.  
	\begin{align*}
	B_1 &\bydef \left\{
	(r,\theta) \in \R\times \bS^1 : r \leq \rho_0 , \;
	\left|
 (p-1) \theta + \vartheta_{0,0}
 - \pp \mod 2 \pi \right|  \leq \tfrac{\pi}{4}  
	\right\} \\
	B_2 &\bydef \left\{
	(r,\theta) \in \R\times \bS^1 : r \leq \rho_0 ,\; 
		\left|
	(p-1) \theta + \vartheta_{0,0}
	- \pp \mod 2 \pi \right|  \geq \tfrac{\pi}{4}  
	,\;  \dot{r} \leq 0
	\right\} .
	\end{align*}

	\begin{figure}[h]
		\centering
		\includegraphics[width=1\linewidth]{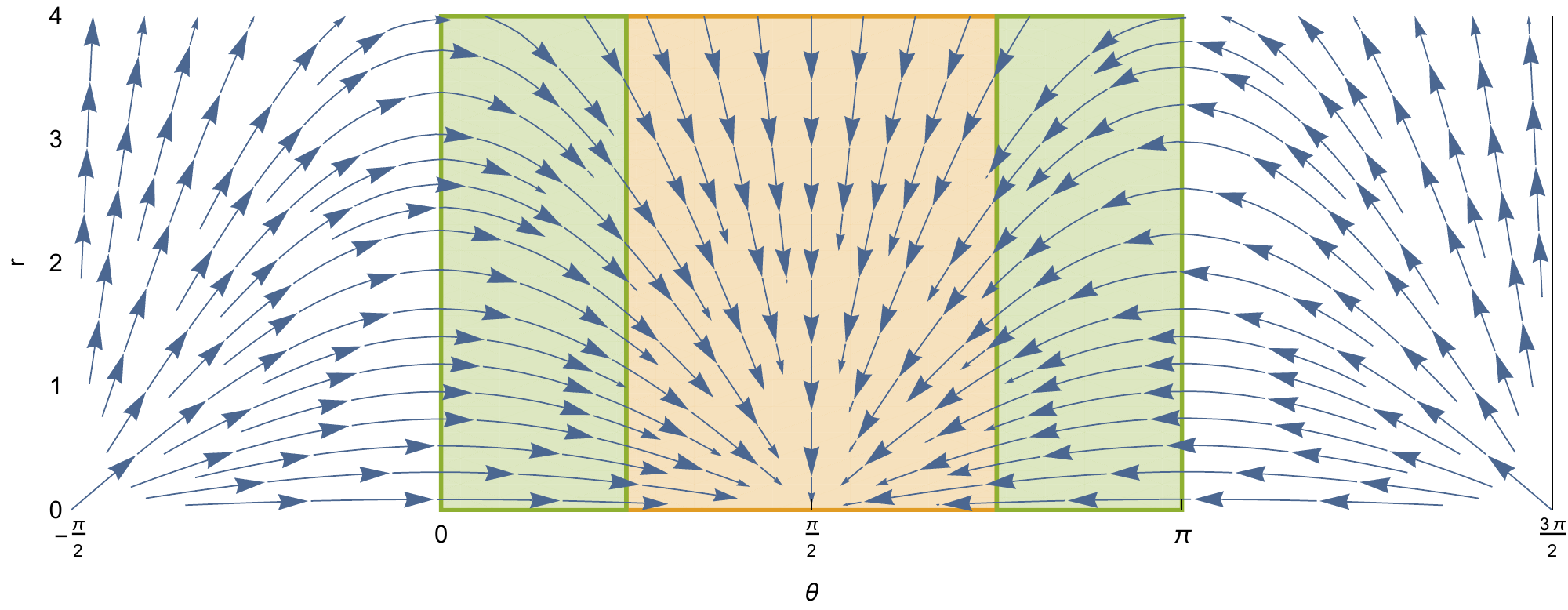} 
		\caption{The blown-up dynamics on $ \R_+ \times \bS^1$ for $p=2$ and  $ \vartheta_{0,0}= 0$ whence \eqref{eq:NLS_hom} becomes $ \dot{z} = i z^2$. 
			The set $ B_1$ is represented by the inner orange region, and the set $B_2$ is represented by the outer green region. 
		}
		\label{fig:Blowup}
	\end{figure}

	We show that for $ \rho_0$ small enough, every initial condition $ (r_0,\theta_0) \in B_2$  will enter $B_1$ within time $ T = \cO( |r_0|^{-(p-1)})$. 
To do so, let us define the set of points in $\R_+\times \bS^1$ with zero radial derivative, 
\begin{align}
\label{eq:ZeroRadialDerivative}
	\mathcal{R} = \cR(\rho_0)= \{ (r,\theta ) \in \R_+ \times  \bS^1 : \dot{r}=0, r \leq \rho_0 \}.
\end{align}
As $f(0) \neq 0$ and thereby $ R_{0,0} \neq 0$, by the implicit function theorem $\cR$  may  be locally written as the union of $2(p-1)$ smooth curves $ (r,\Theta_j(r))$ with   $\Theta_{j}(0)  =  \frac{j \pi - \vartheta_{0,0}}{p-1}$ for $1 \leq j \leq 2(p-1)$. 
Hence, there exists some $ \rho_0$ sufficiently small   and finite  $ \eta = \Theta'(0) + \cO(\rho_0)$  such that $ B_2 \subseteq B_2'$ where 	
\[
		B_2' \bydef \left\{
(r,\theta) \in \R\times \bS^1 : r \leq \rho_0 ,\;
\tfrac{\pi}{4} \leq 
\left|
(p-1) \theta + \vartheta_{0,0}
- \pp \mod 2 \pi \right|  \leq \tfrac{\pi}{2} + \eta r
\right\} .
\]
	Recall from \eqref{eq:RdifferentialEquation} that 
	\begin{align}
	\label{eq:rDotApprox}
	\dot{r} & =\,\,
	-r^p 
	R_{0,0} \sin( (p-1) \theta + \vartheta_{0,0} ) + \cO( |r|^{p+1})  \\
		\label{eq:thetaDotApprox}
	\dot{\theta} & =
	r^{p-1} 
	R_{0,0} \cos( (p-1) \theta + \vartheta_{0,0}  ) + \cO(|r|^p)
	\end{align}
	and note that for $(r,\theta) \in B_2'$ and $\rho_0$ sufficiently small we have 
	\begin{align*}
	\sin \left( (p-1) \theta + \vartheta_{0,0}\right) &\in \left\{\;\;\;
	\cos(\varphi) : \varphi \in [-\pp-\eta \rho_0,-\tfrac{\pi}{4}]
	\cup[ \tfrac{\pi}{4},\tfrac{\pi}{2} + \eta \rho_0]
	\right\} \subseteq  
	\left[ 
	-\sin( \eta \rho_0) , \tfrac{1}{\sqrt{2}}
	\right]
	\\
	\cos \left( (p-1) \theta + \vartheta_{0,0}\right) &\in \left\{
	-\sin(\varphi) : \varphi \in [-\pp-\eta \rho_0,-\tfrac{\pi}{4}]
	\cup[ \tfrac{\pi}{4},\tfrac{\pi}{2} + \eta \rho_0]
	\right\}
	\subseteq 
	\left[\tfrac{1}{\sqrt{2}} ,1 \right] \cup 
	\left[\tfrac{-1}{\sqrt{2}} ,-1 \right].
	\end{align*}
	To bound the higher order terms in \eqref{eq:rDotApprox} and \eqref{eq:thetaDotApprox} we  define
	 $$K= K(\rho_0) \bydef \sum_{m+n\geq 1}^\infty R_{m,n} \rho_0^{m+n}.$$ 
	 Note that  $ K \to 0$ as $ \rho_0 \to 0$.  
	 Computing from \eqref{eq:RdifferentialEquation} we thus obtain a lower bound on $ \dot{r}$ below
	\begin{align*}
\dot{r} 
&\geq 
	-r^p
\left(
\frac{R_{0,0}}{\sqrt{2}} + K 
\right) .
	\end{align*}
 Define  $\delta_1 = \left(
 \frac{R_{0,0}}{\sqrt{2}} + K  
 \right)$.
By solving the differential equation $ r' = -\delta_1  r^p  $ we obtain a lower bound on $ r(t)$ as below  
	\begin{align*}
	r(t) &\geq \frac{r_0}{\left(
		1 + r_0^{p-1} \delta_1 t 
		\right)^{1/p-1}} .
	\end{align*}

We can now estimate how long it takes points in $B_2$ to enter $B_1$. 
	Without loss of generality suppose that $ (r,\theta)$ is in the left half of $B_2$, that is $  - \eta \rho_0 \leq (p-1) \theta + \vartheta_{0,0}   \leq \tfrac{\pi}{4}$, 
		whereby 
	$ \cos( (p-1) \theta + \vartheta_{0,0}  ) \geq 1/\sqrt{2}   $.  
	We estimate how fast $ \theta $ increases, with the goal of estimating when $ (p-1) \theta (t) + \vartheta_{0,0} \geq \tfrac{\pi}{4}$. 
	Suppose $ \rho_0 >0$ is  small enough so that $K <  \frac{R_{0,0}}{\sqrt{2}} $ and define $\delta_2 = \left(
	\frac{R_{0,0}}{\sqrt{2}} - K  
	\right)$, whence   $ \delta_2 >0$.
It follows that 
	\begin{align*}
	\dot{\theta} 
	&\geq 
	r^{p-1} 
	R_{0,0} \cos( (p-1) \theta + \vartheta_{0,0}  ) - r^{p-1} K \\
	&\geq 
	 r(t)^{p-1} \left(
	\frac{R_{0,0}}{\sqrt{2}} - K 
	\right) \\
	&\geq 
	\frac{r_0^{p-1} \delta_2 }{1+r_0^{p-1} \delta_1 t}.
	\end{align*}
Integrating,  we obtain the estimate 
	\begin{align*}
	\theta(t) &\geq 
	\theta_0 + \int_0^t
	\frac{r_0^{p-1} \delta_2 }{1+r_0^{p-1} \delta_1 s} ds 
	= \theta_0 + \frac{\delta_2}{\delta_1} \log \left(
	1+ r_0^{p-1} \delta_1 t
	\right).
	\end{align*}
		We wish to know when   solutions of initial conditions  $(r,\theta) \in B_2$ will enter into region $B_1$, that is, a value of $T$ for which 
	$ (p-1) \theta(T) + \vartheta_{0,0} 
	\geq \frac{\pi}{4} 
	$.
	This occurs when  $  \theta(t)-\theta_0 \geq \frac{\frac{\pi}{4} + \eta \rho_0}{p-1}$, thus we obtain the following 
		\begin{align*}
	T \bydef \frac{\exp \big\{ 
		\frac{\delta_1(\frac{\pi}{4} + \eta \rho_0) }{\delta_1(p-1)}
		 \big\}-1}{  r_0^{p-1} \delta_2 } 
	= \frac{T_0}{r_0^{p-1}}. 
	\end{align*}
	for $T_0>0$ appropriately defined. 
	
	Since $ \dot{r} \leq 0$ for points in $B_2$, then $ r(T) \leq r_0$. 
	For points in $B_1$, 
	we have that $\dot{r} \leq - r^p \left( \frac{R_{0,0} }{\sqrt{2}} - K\right) = - \delta_2 r^p$. 
	Again, by solving this differential equation we obtain the estimate 
	\begin{align} \label{eq:zetaBound2}
	r(t) & \leq 
	\begin{cases}
	r_0 & 0 \leq t \leq T \\
	\frac{r_0}{\left(1 + r_0^{p-1} \delta_2 (t-T)\right)^{1/p-1}} 
	&
	T \leq t .
	\end{cases}
	\end{align}
	
	Integrating the $p$-th power, we obtain:
	\begin{align*}
	\int_0^\infty 
	|z(t)|^p dt &\leq
	r_0^p T + \int_T^\infty 
	\frac{r_0^p}{\left(1 + r_0^{p-1} \delta_2 (t-T)\right)^{p/p-1}} dt
	= r_0 \left( T_0  + \frac{p -1}{\delta_2 }\right).
	\end{align*}
	Hence, for $C =  T_0  + \frac{p -1}{\delta_2 }$ 
	we obtain the inequality in \eqref{eq:pIntegrability}.

\end{proof}

\begin{theorem} \label{thm:HomoclinicBlowupGeneral}   
	Consider \eqref{eq:NLS} with  $ p \geq 2$ and  
	define $B_0(\rho_0)$ as in \eqref{eq:GeneralStableSector} 
	and define 
	\begin{align} \label{eq:BallOfFunction2}
	\cB(\rho_0,\rho_1) &\bydef 
	\left\{
	 \phi + \iota^0 (z_0) \in \ell_{\nu,d}^1 \; 
	:\; 
	z_0 \in B_0(\rho_0)  \subseteq \C;
	\| \phi \|  \leq \rho_1 |z_0|^p
	\right\}.
	\end{align}
	There exists some $  \rho_0 , \rho_1 >0 $ such that  solutions of \eqref{eq:NLS} whose initial data has Fourier coefficients in $\cB(\rho_0,\rho_1) $   will exist for all positive time and  converge to zero as $t \to \infty$.

\end{theorem}

\begin{proof}

Let us consider $\rho_0, \rho_1 >0$ and reserve the freedom to take these constants   sufficiently small later in the proof. 
For $z_0 \in B_0(\rho_0)$  consider $ \phi + \iota^0 ( z_0)  \in \cB(\rho_0,\rho_1)$, and let $ a(t) \in \ell_{\nu,d}^1$ denote the time varying Fourier coefficients of the solution of $\eqref{eq:NLS}$ with Fourier coefficients having the initial condition  $a(0) = \phi + \iota^0(z_0)$. 
	Define $ z(t) : \R_+ \to  \C$ to be the solution of $ z' = i z^pf(z)$ with initial condition $ z_0$. 
		Define time varying Fourier coefficients $\tilde{a}(t) \in \ell_{\nu,d}^1$ by $ a(t) = z(t) + z(t)^p \tilde{a}(t)$, and define 	$ \tilde{\phi} = z_0^{-p} \phi_k$  whereby $ \tilde{a}(0) = \tilde{\phi}$ and $ \|\tilde{\phi}\| \leq \rho_1$.  
	By Proposition \ref{prop:GeneralBlowUp} and Remark \ref{rem:FourierDifferentialEquation} there exists an analytic function $h$ such that the solution of $\tilde{a}$ satisfies 
\[
		\tilde{a}_t = i \triangle \tilde{a} + i
h(z,\tilde{a}).
\] 
Moreover,  there exists some $\delta,K>0$ such that $\|h(z,\tilde{a})\|< \|z^p \tilde{a} \| K$ for all $ |z(t)|,\| \tilde{a}\| \leq \delta$.  
Fix $  \rho_0 <  \delta/2$, and also sufficiently small such that Proposition \ref{prop:Integrability} applies.  
 Further note that $| z(t)| \leq \rho_0$ is monotonically decreasing in $t$.  
For any $0 <   r <  \delta$ define $ T = \sup  \{ t \geq 0 :  \|\tilde{a}(t)\|  \leq r \}$.
	By variation of constants 
	 we obtain 
	\[
	\|\tilde{a}(t)\|  \leq \| \tilde{\phi}\|  + \int_0^t  \left\|h(z(s),\tilde{a}(s)) \right\|  ds 
	\qquad \qquad \forall t \in [0,T].
	\]
Using the estimate on $h$ from  Proposition \ref{prop:GeneralBlowUp} we obtain: 
		\[
	\|\tilde{a}(t)\|  \leq \| \tilde{\phi}\|  + \int_0^t \|z(s)^p \tilde{a}(s)\|  K  ds 
	\qquad \qquad \forall t \in [0,T].
	\]
	Applying Gr\"onwall's inequality and using the estimate $ \| \tilde{a}(s)\| \leq r$ and  Proposition \ref{prop:Integrability} we have the bound 
	\[
	\|\tilde{a}(t)\| 
	\leq 
	\| \tilde{\phi}\| \exp
	 \left\{ 		\int_0^t |z(s)|^p r  K  ds 
	 \right\}
	< \| \tilde{\phi}\|  \exp \{ C r K  |z_0|\}.
	\]
	By fixing   $ \rho_1 <   \exp \{ -		C r K \rho_0\}$, it follows that $ \| \tilde{\phi}\|  \exp \{ 		C r K  |z_0|\}$ is less than $ r$, whereby $\| \tilde{a}(t)\|  <  r$ for all $ t \in [0,T]$. Hence $T = + \infty$, and the solution $u$ to \eqref{eq:NLS}  with Fourier coefficients $ a(t) = \iota^0( z(t)) + z(t)^p \tilde{a}(t)$ exists for all time, and its Fourier coefficients satisfy $\| a(t) - \iota^0(z(t))\|  \leq r |z(t)|^p $, with $ z \to 0 $ as $ t \to \infty$.

\end{proof}

\begin{proof}[Proof of Theorem \ref{thm:GlobalExistenceIntro}]
 The result for solutions converging to zero in forward time follows immediately from Theorem \ref{thm:HomoclinicBlowupGeneral}. 
To obtain the rate at which solutions decay to zero, if  $z(t)$ solves $ \dot{z} = i z^p f(z)$ then by \eqref{eq:zetaBound2}  there exists some $K>0$ such that $ |z(t)| \leq K |t|^{-1/(p-1)}$, and from the proof of Theorem \ref{thm:GlobalExistenceIntro} we have  $\| a(t) - \iota^0(z(t))\|  \leq r |z(t)|^p $ as $ t \to \infty$. 
 The result for solutions converging to zero in backward time follows with parity, and is left to the reader.   
\end{proof}
%
%


\begin{proof}[Proof of Theorem \ref{thm:OpenHomoclinics}] 
	Recall our definition of $\cR(\rho_0) \subseteq \R_+ \times \bS^1 \equiv \C$ in \eqref{eq:ZeroRadialDerivative} of points whose radial derivative is zero under the dynamics of $ z' = i z^p f(z)$.  
	Letting $\iota^0 : \C \to \ell_{\nu,d}^1$ denote the section from $\C$ to the zero-Fourier coefficient of $\ell_{\nu,d}^1$, note that $\cB(\rho_0,\rho_1)$ -- the set of points we just proved converge in norm to $0$ in forward time -- forms an open neighborhood about $\iota^0(\cR(\rho_0))$. 
	As the NLS is time reversible, it follows with parity that for $ \rho_0',\rho_1' >0$ there exists an open set $\cB'(\rho_0',\rho_1')$ about $ \iota^0(\cR(\rho_0'))$ which converges in norm to $0$ in backward time. 
	As $\cR(\rho_0) \cap \cR(\rho_0') \neq \emptyset$, then  $U = \cB(\rho_0,\rho_1) \cap \cB'(\rho_0',\rho_1')$ is a nontrivial open set comprised of solutions which are homoclinic to $0$.
\end{proof}

\begin{proof}[Proof of Theorem \ref{thm:NoConservedQuantities}]
	As $F$ is analytic on $X$ and $i_X : C^{\omega}(\T^d , \C) \hookrightarrow X$ is a continuous embedding, the induced functional $ \tilde{F} = F \circ i_X$ on $C^\omega$ is analytic. By Theorem \ref{thm:OpenHomoclinics} there exists an open set $U \subseteq C^{\omega}$ of solutions to \eqref{eq:NLS} which are homoclinic to the $0$ equilibrium. It follows that on all of $U$, the function $\tilde{F}$ is constant and equal to $\tilde{F}(0)$. Since $\tilde{F}$ is analytic it
	must be constant on all of $C^\omega$. As $C^\omega (\T^d,\C) \subseteq X $ is dense, then $F$ must be constant on all of $X$.
	 
\end{proof}

\begin{proof}[Proof of Proposition  \ref{prop:NoSeparableSolutions}]
Suppose such a solution $u(t,x) = \psi(t) \phi(x)$ exists. 	If there is a point $t_0 \in \R$ for which $\psi(t_0) = 0$, then $ \psi(t)=0$ for all $ t \in \R$. So without loss of generality let us suppose that $ \psi(t) = e^{\theta(t)}$ for a complex valued function $\theta: \R\to \C$. 
Assuming that $ e^{\theta(t)} \phi(x)$ solves \eqref{eq:NLS_p}, we obtain 
	\[
	-i 	\theta'(t) e^{\theta(t)}\phi(x) = e^{\theta(t)} \triangle \phi(x) + e^{p\theta(t)} \phi^p(x).
	\]
	Rearranging, dividing by $e^{\theta(t)}$ and collecting terms, we see that 
	\[
	\triangle \phi(x)
	= \phi(x) \left(
	-i 	\theta'(t) 
	+ e^{(p-1)\theta(t)} \phi^{p-1}(x)
	\right).
	\]
	If we take the time derivative of this expression, then both sides of the equation above must be zero, thereby 
	\[
	0
	= 
	-i 	\theta''(t) 
	+  (p-1) \theta'(t) e^{(p-1)\theta(t)} \phi^{p-1}(x)
	.
	\]
	Taking a spatial derivative, we thus obtain  
\begin{align*}
		0
		&	=  (p-1)^2 \theta'(t) e^{(p-1)\theta(t)} \phi^{p-2}(x) 
		\partial_x \phi(x)\\ 
		&	=   \theta'(t)   \phi^{p-2}(x) 
		\partial_x \phi(x)
	.
\end{align*}
	Hence, either $ \theta'(t) \equiv 0$ or $ \partial_x \phi(x) \equiv 0$. That is to say either $ \theta(t)$ or $ \phi(x)$ is constant. 
\end{proof}


\section{Unstable manifold of nontrivial equilibria} \label{sec:UnstableManifold}

In this section, we fix the dimension of the domain to be $d=1$ and we combine the Parameterization Method (e.g. see \cite{MR1976079,MR1976080,MR2177465}) and rigorous numerics to parameterize one-dimensional subsets of unstable manifolds of steady states of the NLS equation
\begin{equation} \label{eq:NLS_1d}
 u_t = i (u_{xx} + u^2).
\end{equation}

For the sake of simplicity of the presentation, denote $\ell_{\nu}^1=\ell_{\nu,1}^1$. Also denote by $a*b$ the discrete convolution of $a,b \in \ell_{\nu}^1$. Proceeding formally for the moment, solutions of \eqref{eq:NLS_1d} may be expanded in Fourier series as
\begin{equation} \label{eq:Fourier_expansion}
u(t,x) \bydef \sum_{k \in \Z} a_k(t) e^{i k \omega x},
\end{equation}
and matching like terms, this leads to the infinite system of nonlinear ODEs
\begin{equation} \label{eq:CGL_ODEs}
\dot a_k(t) = g_k(a) \bydef i \left( -k^2\omega^2 a_k+\left( a^2 \right)_k \right),
\end{equation}
where $a^2 = a*a$. 

\begin{remark}[\bf Symmetry assumption]
Throughout the rest of this section we make the assumption that the Fourier coefficients in \eqref{eq:Fourier_expansion} satisfy the symmetry
\begin{equation} \label{eq:Fourier_symmetry}
a_{-k}(t) = a_{k}(t), \quad \text{for all } k \in \Z.
\end{equation}
\end{remark}

We are interested in solutions of \eqref{eq:CGL_ODEs} whose spectral representation have decaying coefficients. So, recalling Definition~\ref{def:ell_nu^1} and the symmetry assumption \eqref{eq:Fourier_symmetry}, let
\begin{equation} \label{eq:ell_nu_one}
\ell_\nu^1 \bydef \left\{ a = (a_k)_{k \ge 0} : |a|_{\nu} \bydef |a_0| + 2 \sum_{k=1}^\infty |a_k| \nu^k  < \infty \right\}.
\end{equation}
Note that restricting to assumption \eqref{eq:Fourier_symmetry}, $\ell_\nu^1 = \ell_{\nu,d}^1$ for $d=1$.
Given two sequences $a=(a_k)_{k \ge 0},b=(b_k)_{k \ge 0} \in \ell_\nu^1$ defined in \eqref{eq:ell_nu_one}, their discrete convolution $a*b=((a*b)_k )_{k \ge 0}$ is given component-wise by
\[
(a*b)_k = \sum_{k_1+k_2 = k \atop k_1,k_2 \in \Z} a_{|k_1|} b_{|k_2|}.
\]

Denote $g(a) = (g_k(a))_{k \in \Z}$, so that \eqref{eq:CGL_ODEs} can be more densely written as $\dot a = g(a)$, which is an ODE defined on the Banach space $\ell_\nu^1$.

In Theorem \ref{thm:Equilibria} we establish via computer-assisted proof the existence of a steady state $\ta \in \ell_\nu^1$ (that is $g(\ta)=0$) and of an eigenpair $(\tilde \lambda,\tb) \in \C \times \ell_\nu^1$ satisfying $Dg(\ta)\tb=\tilde \lambda \tb$ with ${\rm Re}(\tilde \lambda)>0$.
The details of the computation of $\ta$ and $(\tilde \lambda,\tb)$ can be found in Appendix~\ref{sec:eigenpairs}. More explicitly, for all $k \in \Z$, this implies that
\begin{align} 
\label{eq:f=0_steady_state}
-k^2\omega^2 \ta_k+\left( \ta^2 \right)_k & = 0
\\
\label{eq:g=0_eigenpair}
i \left( -k^2\omega^2 \tb_k + 2 \left(\ta*\tb \right)_k\right) - \tilde \lambda \tb_k &= 0.
\end{align}
Note that for $\ta$ and $\tb$, we assume that the symmetry condition \eqref{eq:Fourier_symmetry} holds.
The (computer-assisted) approach to obtain the steady state $\ta$ and the eigenpair $(\tilde \lambda,\tb)$ is presented in Section~\ref{sec:eigenpairs}. 

From now on, we follow closely the approach presented in \cite{MR3906120}, and present a rigorous computational approach to parameterize a subset of the unstable manifold of $\ta$ given by
\[
W^u(\ta) \bydef \left\{ a_0 ~ \big| ~ \exists \text{ a solution } a(t) \text{ of  \eqref{eq:CGL_ODEs} with } a(0) = a_0 
\text{ and } \lim_{t \to -\infty} a(t) = \ta \right\}.
\]
Denote $\mathbb{D} = \{ z \in \C : |z| \le 1\}$ the closed unit disk centered at $0$ in the complex plane. We look for a function $P:\mathbb{D} \to \ell_\nu^1$ such that
\begin{equation} \label{eq:invariance_equation}
g(P(\sigma)) = \tilde \lambda \sigma DP(\sigma), \quad \text{for all }\sigma \in \mathbb{D},
\end{equation}
 subject to the first order constraints
\begin{equation} \label{eq:first_order_constraints}
P(0)=\ta \quad \text{and} \quad DP(0) = \tb.
\end{equation}
We refer to equation \eqref{eq:invariance_equation} as the {\em invariance equation} for the parameterization method.

The following result justifies the role played by the invariance equation \eqref{eq:invariance_equation}. We omit the elementary proof.

\begin{lemma}\label{lem:parameterization}
Assume that $P:\mathbb{D} \to \ell_\nu^1$ solves \eqref{eq:invariance_equation} subject to the first order constraints of \eqref{eq:first_order_constraints}. Then for every $\sigma \in \mathbb{D}$ the function
\begin{equation} \label{eq:unstable_solution}
a(t) \bydef P(e^{\tilde \lambda t} \sigma)
\end{equation}
solves the differential equation on $\ell_\nu^1$ given by $\dot a = g(a)$ on $(-\infty,0)$. Moreover, since $a(0) = P(\sigma) \in image(P)$ and $\lim_{t \to -\infty} a(t) = P(0)=\ta$, then $P$ parameterizes a local unstable sub-manifold for $\ta$, that is 
\[
image(P) = P(\mathbb{D}) \subset W^u(\ta).
\]
\end{lemma}

The strategy is now clear: compute a function $P:\mathbb{D} \to \ell_\nu^1$ which solves \eqref{eq:invariance_equation} subject to the first order constraints of \eqref{eq:first_order_constraints}. 
We represent $P$ as a power series
\begin{equation} \label{eq:P_power_series}
P(\sigma) = \sum_{m \ge 0} p_m \sigma^m, \qquad p_m \in \ell_\nu^1.
\end{equation}
To represent the coefficients of the power series \eqref{eq:P_power_series},
we define the Banach space of infinite sequences of points in $\ell_\nu^1$ given by
\begin{equation} \label{eq:X_nu}
X^\nu \bydef  \left\{ p = \{p_m\}_{m \ge 0} : p_m \in \ell_\nu^1 \text{ and } 
\| p\|_\nu \bydef \sum_{m \ge 0} |p_m|_\nu < \infty \right\}.
\end{equation}
Note that any $p \in X^\nu$ may be written as $p = (p_{k,m})_{k,m \ge 0}$ and that the norm in $X^\nu$ may be written as 
\[
\| p\|_\nu = \sum_{m \ge 0} |p_m|_\nu
= \sum_{m \ge 0} \left( |p_{0,m}| + 2 \sum_{k=1}^\infty |p_{k,m}| \nu^k \right)
= \sum_{k,m \ge 0} |p_{k,m}| \omega_k,
\]
where we define the weights $\omega=(\omega_k)_{k\geq0}$ by $\omega_0=1$ and $\omega_k = 2 \nu^k$ for $k\geq 1$. 

\begin{definition}
Given $p,q \in X^\nu$, define their {\em Taylor-Fourier product} $*_{TF} :X^\nu \times X^\nu \to X^\nu$ by
\begin{equation}
(p *_{TF} q)_m \bydef \sum_{\ell =0}^m p_\ell * q_{m-\ell},
\end{equation} 
where $*:\ell_\nu^1 \times \ell_\nu^1 \to \ell_\nu^1$ is the standard discrete convolution. 
\end{definition}

We have the following result which makes $(X^\nu,*_{TF})$ a Banach algebra.

\begin{lemma} \label{lem:TF_Banach_algebra}
Let $p,q \in X^\nu$. Then
\begin{equation} 
\|p *_{TF} q \|_\nu \le \|p \|_\nu \|q \|_\nu.
\end{equation}
\end{lemma}

Proceeding formally for the moment, plugging the power series \eqref{eq:first_order_constraints} in the invariance equation \eqref{eq:invariance_equation} and imposing the first order constraints \eqref{eq:first_order_constraints} leads to look for a solution $p \in X^\nu$ of $f(p)=0$, where the map $f = (f_{k,m})_{k,m \ge 0}$ is given by
\begin{equation} \label{eq:f=0_manifold}
f_{k,m}(p) \bydef 
\begin{cases}
p_{k,0} - \ta_k, & m=0, \quad k \ge 0
\\
p_{k,1} - \tb_k, & m=1, \quad k \ge 0
\\
\left(\tilde \lambda m + i k^2 \omega^2 \right) p_{k,m} -  i (p*_{TF}p)_{k,m}, & m \ge 2, \quad k \ge 0.
\end{cases}
\end{equation}
Since $(X^\nu,*_{TF})$ is a Banach algebra (see Lemma~\ref{lem:TF_Banach_algebra}), 
then $f : X^\nu \to \tilde X^{\nu}$, where
\[
\tilde X^\nu \bydef  \left\{ p = \{p_m\}_{m \ge 0} : 
\sum_{m \ge 0} \left( |p_{0,m}| + 2 |p_{1,m}| \nu
+ 2 \sum_{k=2}^\infty |p_{k,m}| \frac{\nu^k}{|\tilde \lambda m + i k^2 \omega^2|}  \right)<\infty \right\}.
\]

The following result resumes what we have achieved so far.

\begin{lemma}
Suppose that $p = (p_{k,m})_{k,m \ge 0} \in X^\nu$ solves $f(p)=0$, and for each $m\ge0$, denote $p_m = (p_{k,m})_{k \ge 0}$. Then, for each $\sigma \in \mathbb{D}$,
\[
P(\sigma) \bydef \sum_{m \ge 0} p_m \sigma^m \in W^u(\ta).
\]
\end{lemma}

%

The remaining part of this section introduces the method to compute rigorously an infinite dimensional vector $p \in X^\nu$ such that $f(p)=0$. This is done using a Newton-Kantorovich type theorem, which we now introduce. Before that, denote by $B_r(y) \bydef \{ q \in X^\nu : \| q - y \|_{\nu} \le r\}$ the closed ball of radius $r>0$ centered at $y \in X^{\nu}$. 

\begin{theorem}[{\bf A Newton-Kantorovich type theorem}] \label{thm:radii_polynomials}
Let $X$ and $X'$ be Banach spaces, $A^{\dagger} \in B(X,X')$ and $A \in B(X',X)$ be bounded linear operators.
Consider a point $\bp \in X$ (typically a numerical approximation) and assume that $f \colon X \to X'$ is Fr\'echet differentiable at $\bp$. Moreover, assume that $A$ is injective and that $A f \colon X \to X$.
Let $Y_0$, $Z_0$ and $Z_1$ be nonnegative constants, and a function $Z_2:(0,\infty) \to (0,\infty)$ satisfying
\begin{align}
\label{eq:general_Y_0}
\| A f(\bp) \|_X &\le Y_0
\\
\label{eq:general_Z_0}
\| I - A A^{\dagger}\|_{B(X)} &\le Z_0
\\
\label{eq:general_Z_1}
\| A[A^{\dagger} - Df(\bp)] \|_{B(X)} &\le Z_1,
\\
\label{eq:general_Z_2}
\| A[Df(c) - Df(\bp)]\|_{B(X)} &\le Z_2(r) r, \quad \text{for all } q \in B_r(\bp),
\end{align}
where $\| \cdot \|_{B(X)}$ denotes the operator norm.  Define the radii polynomial by 
\begin{equation} \label{eq:general_radii_polynomial}
p(r) \bydef Z_2(r) r^2 - ( 1 - Z_1 - Z_0) r + Y_0.
\end{equation}
If there exists $r_0>0$ such that $p(r_0)<0$, then there exists a unique $\tp \in B_{r_0}(\bp)$ such that $f(\tp) = 0$.
\end{theorem}

The Banach spaces we choose to solve $f=0$ are $X=X^\nu$ and $X' = \tilde X^\nu$. An application of Theorem~\ref{thm:radii_polynomials} requires computing a numerical approximation $\tp$ of a zero of $f$ given in \eqref{eq:f=0_manifold}, which is done by introducing first a finite dimensional projection. 

Given $M \ge 2$ (the Taylor projection) and $K>0$ (the Fourier projection), we define the projections $\Pi^{(K,M)},\Pi^{(\infty)}:X^\nu \to X^\nu$ by
\[
\left( \Pi^{(K,M)}(p) \right)_{k,m} \bydef 
\begin{cases} p_{k,m},& k=0,\dots,K \text{ and } m = 0,\dots,M
\\
0, & \text{otherwise}
\end{cases}
\]
and
\[
\Pi^{(\infty)} \bydef I-\Pi^{(K,M)}.
\]
Denote by $\iota^{(K,M)}:\C^{(K+1)\times (M+1)} \to X^\nu$ the natural inclusion defined by taking a vector 
$p=(p_{k,m})_{k=0,\dots,K \atop m=0,\dots,M} \in \C^{(K+1)\times (M+1)}$ and mapping it to $\iota^{(K,M)}(p) \in X^\nu$ as
\[
\left( \iota^{(K,M)}(p) \right)_{k,m} \bydef 
\begin{cases} p_{k,m},& k=0,\dots,K \text{ and } m = 0,\dots,M
\\
0, & \text{otherwise}.
\end{cases}
\]
Define the finite dimensional projection mapping $f^{(K,M)}: \C^{(K+1)\times (M+1)} \to \C^{(K+1)\times (M+1)}$ by
\[
f^{(K,M)}(p) \bydef \Pi^{(K,M)} \left( f \left( \iota^{(K,M)}(p) \right) \right).
\]
Assume that using Newton's method, we computed $\bp \in \C^{(K+1)\times (M+1)} $ with $f^{(K,M)}(\bp) \approx 0$.
The next step required to apply Theorem~\ref{thm:radii_polynomials} is to define the operators $A,\dagA$. 
To simplify the presentation we use the notation $\bp$ to represent $\bp \in \C^{(K+1)\times (M+1)}$ and $\bp \in X^\nu$, the natural inclusion in $X^\nu$. 

Given $h \in X^\nu$, denote $h^{(K,M)} \bydef \Pi^{(K,M)} h$.
Define the operator $\dagA$ action-wise by
\begin{equation} \label{eq:op_dagA}
\left(\dagA h \right)_{k,m} = 
\begin{cases}
\left( Df^{(K,M)}(\bp) h^{(K,M)} \right)_{k,m}, & k=0,\dots,K \text{ and } m = 0 ,\dots,M
\\
h_{k,m}, & k >K \text{ and } m = 0 ,1
\\
\left( \tilde \lambda m + i k^2 \omega^2 \right) h_{k,m}, & \text{otherwise}.
\end{cases}
\end{equation}

If the projection dimensions $K,M$ are taken large enough, it is expected that the operator $\dagA$ defined in \eqref{eq:op_dagA} acts as an approximation for the Fr{\'e}chet derivative $Df(\bp)$.
 
Denote by $A^{(K,M)}$ a numerical inverse of the Jacobian matrix $Df^{(K,M)}(\bp)$, that is 
\[
\| I - A^{(K,M)} Df^{(K,M)}(\bp) \| \ll 1.
\]

Define the operator $A$ action-wise by
\begin{equation} \label{eq:op_A}
\left(A h \right)_{k,m} = 
\begin{cases}
\left( A^{(K,M)} h^{(K,M)} \right)_{k,m}, & k=0,\dots,K \text{ and } m = 0 ,\dots,M
\\
h_{k,m}, & k >K \text{ and } m = 0 ,1
\\
\left( \tilde  \lambda m + i k^2 \omega^2 \right)^{-1} h_{k,m}, & \text{otherwise}.
\end{cases}
\end{equation}

Denote
\[
\mu_{k,m} \bydef \tilde \lambda m + i k^2 \omega^2.
\]
Let us decompose the matrix $A^{(K,M)} \in M_{(K+1)(M+1)}(\C)$ block-wise as
\[
A^{(K,M)} =
\begin{pmatrix}
A^{(K,M)}_{0,0} & A^{(K,M)}_{0,1} & \cdots & A^{(K,M)}_{0,M}  \\
A^{(K,M)}_{1,0} & A^{(K,M)}_{1,1} & \cdots & A^{(K,M)}_{1,M}  \\
\vdots & \vdots & \ddots & \vdots \\
A^{(K,M)}_{M,0} & A^{(K,M)}_{M,1} & \cdots & A^{(K,M)}_{M,M}
\end{pmatrix}
\]
so that it acts on $h^{(K,M)} \bydef (h_0^{(K)},\dots,h_M^{(K)}) \in \C^{(K+1)(M+1)}$. Note that each block of $A^{(K,M)}$ is a matrix $A^{(K,M)}_{i_1,i_2} \in M_{K+1}(\C)$ acting on $h_{i_2}^{(K)} \in \C^{M+1}$.  
Thus we define $A$ as
\begin{equation} \label{eq:A}
A =
\begin{pmatrix}
A_{0,0} & A_{0,1} & \cdots & A_{0,M} & 0 & \cdots \\
A_{1,0} & A_{1,1} & \cdots & A_{1,M} & 0 & \cdots \\
\vdots & \vdots & \ddots & \vdots & 0 & \cdots \\
A_{M,0} & A_{M,1} & \cdots & A_{M,M} & 0 & \cdots \\
0 & 0 & \cdots & 0 & A_{M+1,M+1} & \cdots \\
\vdots & \vdots &  & \vdots & 0 & \ddots
\end{pmatrix},
\end{equation}
where each $A_{j,m} \in B(\ell_\nu^1)$ and where the action of each block of $A$ is finite (that is they act on $h^{(K,M)} = \Pi^{(K,M)} h$ only) except for the diagonal blocks $A_{j,j}$. 
More explicitly, for $m=0,1$,
\[
(A_{m,m} h_m)_k = 
\begin{cases}
\bigl( A_{m,m}^{(K,M)} h_m^{(K)} \bigr)_k & \quad\text{for }  k = 0,\dots,K, \\
(h_m)_k & \quad\text{for } k > K,
\end{cases}
\]
for $m=2,\dots,M$
\[
(A_{m,m} h_m)_k = 
\begin{cases}
\bigl( A_{m,m}^{(K,M)} h_m^{(K)} \bigr)_k & \quad\text{for }  k = 0,\dots,K, \\
\displaystyle \frac{1}{\mu_{k,m}(\tilde \lambda))} (h_m)_k & \quad\text{for } k > K,
\end{cases}
\]
and for $m>M$, 
\[
(A_{m,m} h_m)_k = 
\frac{1}{\mu_{k,m}(\tilde \lambda))} (h_m)_k \quad\text{for } k \ge 0.
\]
The finite non-diagonal blocks satisfy
\[
(A_{j,m} h_m)_k = 
\begin{cases}
\bigl( A_{j,m}^{(K,M)} h_m^{(K)} \bigr)_k & \quad\text{for }  k = 0,\dots,K, \\
0 & \quad\text{for } k > K,
\end{cases}
\]

Having introduced the map $f$ in \eqref{eq:f=0_manifold}, the Banach spaces $X=X^\nu$ and $X' = \tilde X^\nu$, the operators $\dagA$ and $A$ given in \eqref{eq:op_dagA} and \eqref{eq:op_A}, respectively, we are ready to compute the bounds $Y_0$, $Z_0$, $Z_1$ and $Z_2$ required to apply Theorem~\ref{thm:radii_polynomials}. Before that, let us introduce some useful elementary functional analytic tools.

\subsection{Elementary functional analytic results}


Recall the definition of the Banach space $\ell_\nu^1$ given in \eqref{eq:ell_nu_one}.

\begin{lemma}\label{l:dualbound}
The dual space $(\ell_\nu^1)^*$ is isometrically isomorphic to 
\[
\ell_{\nu^{-1}}^\infty = \left\{ c = (c_k)_{k \ge 0} : 
| c |_{\infty,\nu^{-1}} \bydef \max \left( |c_0|, \tfrac{1}{2} \sup_{k \geq 1} |c_k| \nu^{-k} \right) < \infty \right\}.
\]
For all $a \in \ell_\nu^1$ and $c \in \ell^\infty_{\nu^{-1}}$ we  have
\begin{equation}\label{e:dualbound}
\Bigl|\sum_{k \geq 0} c_k p_k \Bigr| \leq |c|_{\infty,\nu^{-1}} |a|_{\nu}.
\end{equation}
\end{lemma}

Using the bound~\eqref{e:dualbound}, we estimate the convolution
\[
  \sup_{|v|_{\nu} \leq 1} | (a \ast v)_k | =
   \sup_{|v|_{\nu} \leq 1} \left| \sum_{k' \in \mathbb{Z}} v_{|k'|} a_{|k-k'|}   \right| \leq 
   \max \left\{ |a_k| ,\sup_{k'\geq1} \frac{|a_{|k-k'|} + a_{|k+k'|}|}{2 \nu^{k'}} \right\}.
\]
Given $v = (v_k)_{k \ge 0} \in \ell_\nu^1$, define $\hv \in \ell_\nu^1$ as follows:
\[
\hv_k \bydef \begin{cases} 0 & \text{if } k \le K,\\
v_k & \text{if } k > K. \end{cases}
\]
A similar estimate as the one above leads to
\begin{corollary} \label{cor:psi_k}
\[
  \sup_{|v|_{\nu} \leq 1} | (a \ast \hv)_k |  \leq 
  \sup_{k'\geq M+1} \frac{|a_{|k-k'|} + a_{|k+k'|}|}{2 \nu^{k'}}  \bydef \Psi_k(a).
\]
Moreover, if $a \in \ell_\nu^1$ satisfies $a_k = 0$ for all $k \ge K+1$, then we obtain the computable bound
\[
\Psi_k(a) = \max_{\ell = K+1,\dots,k+K} \frac{|a_{|k-\ell|}|}{2 \nu^{\ell}},
\quad k=0,\dots,K.
\]
\end{corollary}


The final results of this short section concern the computation of norms of bounded linear operators defined on $\ell_\nu^1$ and on $X^\nu$, and are useful when computing the bounds $Z_0$ and $Z_2$.

\begin{lemma}\label{l:Blnu1norm}
Let $\Gamma \in B(\ell^1_{\nu})$, the space of bounded linear operators from $\ell^1_\nu$ to itself, acting as $(\Gamma a)_i =\sum_{j\geq 0} \Gamma_{i,j} a_j$ for $i=1,2$. Define the weights
$\omega=(\omega_k)_{k\geq0}$ by $\omega_0=1$ and $\omega_k = 2 \nu^k$ for $k\geq 1$. Then 
\[
   \| \Gamma \|_{B(\ell^1_\nu)} = \sup_{j \geq 0} \frac{1}{\omega_j}  \sum_{i\geq 0} | \Gamma_{i,j} | \omega_i .
\] 
\end{lemma}

The following consequence of Lemma~\ref{l:Blnu1norm} provides an explicit bound on norms of bounded linear operators on $\ell_\nu^1$ with a specific structure, namely as in the blocks of \eqref{eq:A}.

\begin{corollary}\label{cor:OperatorNorm}
Let $\Gamma^{(K)}$ be an $(K+1) \times (K+1)$ matrix, 
$\{\gamma_k\}_{k=K+1}^{\infty}$
be a sequence of numbers with 
\[
|\gamma_k| \leq |\gamma_{K+1}|, \qquad\text{for all } k \ge K+1,
\]
and 
$\Gamma \colon \ell_{\nu}^1 \to \ell_{\nu}^1$ be the linear operator defined by
\begin{equation} \label{eq:Gamma_blocks}
\Gamma a = 
\begin{pmatrix}
\Gamma^{(K)} &  & 0   \\
  & \gamma_{K+1}  &   \\
0  &   & \gamma_{K+2} & \\  
& & & \ddots
\end{pmatrix}
\begin{pmatrix}
 a^{(K)}  \\
  a_{K+1}   \\
  a_{K+2} \\
  \vdots 
\end{pmatrix}.
\end{equation}
Here $a^{(K)} = (a_0, \ldots, a_{K})^T \in \C^{K+1}$. Then $\Gamma \in B(\ell_{\nu}^1)$ and 
\begin{equation} \label{eq:norm_B_ell_nu_1}
\| \Gamma \|_{B(\ell_{\nu}^1)} =  \max (K, |\gamma_{K+1}|),
\end{equation}
where 
\[
K \bydef \max_{0 \leq j \leq K} \frac{1}{\omega_j}\sum_{i=0}^{K} |\Gamma_{i,j}| \omega_i.
\]
\end{corollary}

Using the previous result, we can computing an upper bound for the operator norm of $A$ defined in \eqref{eq:A}.

\begin{lemma} \label{lem:normA}
Let %
\begin{equation} \label{eq:normA}
\alpha_A \bydef \max \left( 
\max_{m=0,\dots,M} \sum_{j=0}^M \|A_{j,m} \|_{B(\ell_\nu^1)}
~,~
\frac{1}{\text{Re}( \tilde \lambda)(M+1)}
\right).
\end{equation}
Then $\| A \|_{B(X^\nu)} \le \alpha_A$.
\end{lemma}
\begin{proof}
Let $h = (h_m)_{m\ge 0} \in X^\nu$ with $\| h\|_{\nu} \le 1$. Then,
\[
\| A h \|_{\nu}  = \sum_{j\ge 0} |(Ah)_j|_{\nu}  = \sum_{j\ge 0} |\sum_{m \ge 0}A_{j,m}h_m|_{\nu}  \le \sum_{j\ge 0} \sum_{m \ge 0} \|A_{j,m}\|_{B(\ell_\nu^1)} |h_m|_{\nu}  \le  \sum_{m \ge 0} \left( \sum_{j\ge 0} \|A_{j,m}\|_{B(\ell_\nu^1)} \right) |h_m|_{\nu}.
\]
For each $m\ge 0$, denote 
\[
\alpha_m \bydef \sum_{j\ge 0} \|A_{j,m}\|_{B(\ell_\nu^1)}. 
\]
Recalling \eqref{eq:A}, for $m=0,\dots,M$, note that $\alpha_m = \sum_{j = 0}^M \|A_{j,m}\|_{B(\ell_\nu^1)}$, where each $\|A_{j,m}\|_{B(\ell_\nu^1)}$ can be computed using formula \eqref{eq:norm_B_ell_nu_1}. Finally, for $m>M$, 
$\alpha_m = \sum_{j = 0}^M \|A_{j,m}\|_{B(\ell_\nu^1)}
= \|A_{m,m}\|_{B(\ell_\nu^1)}$, where
\[
\|A_{m,m}\|_{B(\ell_\nu^1)} = \sup_{|b|_{\nu} \le 1} | A_{m,m} b|_{\nu}  = \sup_{|b|_{\nu} \le 1}  \sum_{k\ge 0} \left| \frac{1}{\mu_{k,m}(\tilde \lambda))} b_k \right| \omega_k \le \frac{1}{\text{Re}(\tilde \lambda) (M+1)} \sup_{|b|_{\nu} \le 1}  \sum_{k\ge 0} \left|  b_k \right| \omega_k  \le \frac{1}{\text{Re}(\tilde \lambda) (M+1)}
\]
since for $m>M$ and $k \ge 0$,
$ |\mu_{k,m}(\tilde \lambda)| \ge \text{Re}(\tilde \lambda) m  \ge \text{Re}(\tilde \lambda) (M+1)$. \qedhere
\end{proof}

\subsection{The \boldmath $Y_0$ \unboldmath bound}

Recall that we want to compute $Y_0$ satisfying \eqref{eq:general_Y_0}, that is
\[
\| A f(\bp) \|_X \le Y_0.
\]
Recall that we obtained (via a previous computer-assisted proof) $\tilde \lambda \in \C$, $\ta \in \ell_\nu^1$ and $\tb \in \ell_\nu^1$ such that \eqref{eq:f=0_steady_state} and \eqref{eq:g=0_eigenpair} hold for all $k \in \Z$. Note that the error bounds for 
$(\tilde \lambda,\ta,\tb) \in B_{r_0}(\bar \lambda,\ba,\bb)$ for $\bar \lambda,\ba,\bb$ numerical approximations. More explicitly, letting 
\[
\delta=(\delta_\lambda,\delta_a,\delta_b) \bydef (\tilde \lambda - \bar \lambda,\ta-\ba,\tb-\bb) \in \C \times \ell_\nu^1 \times \ell_\nu^1,
\]
we get that
\[
\max \left( |\delta_\lambda|, |\delta_a|_{\nu} , |\delta_a|_{\nu} \right)  = 
\max \left( |\tilde \lambda - \bar \lambda|, | \ta-\ba|_{\nu} , |\tb-\bb|_{\nu} \right) \le r_0.
\]

Recall \eqref{eq:f=0_manifold} and notice that
\[
f_{k,m}(\bp) \bydef 
\begin{cases}
\bp_{k,0} - \ta_k, & m=0, \quad k \ge 0
\\
\bp_{k,1} - \tb_k, & m=1, \quad k \ge 0
\\
\left( \tilde \lambda m + i k^2 \omega^2 \right) \bp_{k,m} - i (\bp*_{TF}\bp)_{k,m}, & m \ge 2, \quad k \ge 0.
\end{cases}
\]
and hence $f(\bp) = \bar f(\bp) + f^{\delta}$, where
\[
\bar f_{k,m}(\bp) \bydef 
\begin{cases}
\bp_{k,0} - \ba_k, & m=0, \quad k \ge 0
\\
\bp_{k,1} - \bb_k, & m=1, \quad k \ge 0
\\
\left( \tilde \lambda m + i k^2 \omega^2 \right) \bp_{k,m} - i (\bp*_{TF}\bp)_{k,m}, & m \ge 2, \quad k \ge 0.
\end{cases}
\]
and 
\[
f^{\delta}_{k,m} \bydef 
\begin{cases}
-(\delta_a)_k, & m=0, \quad k \ge 0
\\
-(\delta_b)_k, & m=1, \quad k \ge 0
\\
0, & m \ge 2, \quad k \ge 0.
\end{cases}
\]
Denote 
\[
\bar f(\bp) = \left( \bar f_m(\bp) \right)_{m\ge 0}
\quad \text{and} \quad
f^{\delta}(\bp) =  \left( f_m^{\delta}(\bp) \right)_{m\ge 0}.
\]
Since $\bp_{m} = 0$ for all $m>M$ and $(\bp*_{TF}\bp)_{m} = 0$ for all $m > 2M$, then
\begin{align}
\nonumber
\| A \bar f(\bp) \|_{\nu} &= \sum_{j=0}^{2M} |(A \bar f(\bp))_j|_{\nu} 
= \sum_{j=0}^{M} |(A \bar f(\bp))_j|_{\nu} + \sum_{j=M+1}^{2M} |A_{j,j} \bar f_j(\bp)|_{\nu} 
\\
\nonumber
& \le \sum_{j=0}^{M} \left| 
\sum_{m=0}^{M} A_{j,m} \bar f_m \right|_{\nu} + \sum_{m=M+1}^{2M} |A_{m,m} (\bp*_{TF}\bp)_m|_{\nu} 
\\
\nonumber
& \le
\sum_{j=0}^{M} \left( 
\sum_{k \ge 0}
\left|  \sum_{m=0}^{M} (A_{j,m} \bar f_m)_k \right| \omega_k \right) + \sum_{m=M+1}^{2M} 
\sum_{k \ge 0} |(A_{m,m} (\bp*_{TF}\bp)_m)_k| \omega_k 
\\
\nonumber
& = 
\sum_{j=0}^{M} \left( 
\sum_{k = 0}^{2N}
\left|  \sum_{m=0}^{M} (A_{j,m} \bar f_m)_k \right| \omega_k \right) + \sum_{m=M+1}^{2M} 
\sum_{k = 0}^{2N} | \frac{1}{\mu_{k,m}(\tilde \lambda))} (\bp*_{TF}\bp)_{k,m} | \omega_k 
\\
\nonumber
& = 
\sum_{j=0}^{M} \left( 
\sum_{k = 0}^{K}
\left|  \sum_{m=0}^{M} (A^{(K,M)}_{j,m} \bar f^{(K)}_m)_k \right| \omega_k \right) 
+
\sum_{j=0}^{M} \left( 
\sum_{k = K+1}^{2K}
\left|  (A_{j,j} \bar f_j)_k \right| \omega_k \right)
\\
\nonumber
& \quad 
+ \sum_{m=M+1}^{2M} 
\sum_{k = 0}^{2K} | \frac{1}{\mu_{k,m}(\tilde \lambda))} (\bp*_{TF}\bp)_{k,m} | \omega_k 
\\
\nonumber
& = 
\sum_{j=0}^{M} \left( 
\sum_{k = 0}^{K}
\left|  \sum_{m=0}^{M} \left( A^{(K,M)}_{j,m} \bar f^{(K,M)} \right)_{k,m} \right| \omega_k \right) 
+ 
\sum_{k = K+1}^{2K}
\left| (\bar f_0)_k \right| \omega_k 
+ 
\sum_{k = K+1}^{2K}
\left| (\bar f_1)_k \right| \omega_k 
\\
&   + 
\sum_{m=2}^{M}
\sum_{k = K+1}^{2K}
\left| \frac{1}{\mu_{k,m}(\tilde \lambda))} (\bar f_m)_k \right| \omega_k 
+ \hspace{-.2cm}
\sum_{m=M+1}^{2M} 
\sum_{k = 0}^{2K} | \frac{1}{\mu_{k,m}(\tilde \lambda))} (\bp*_{TF}\bp)_{k,m} | \omega_k \bydef Y_0^{(1)},
\label{eq:Y0_1}
\end{align}
which is a finite computation that can be rigorously obtained using that $|\tilde \lambda - \bar \lambda| \le r_0$ and using interval arithmetic. Now,
\begin{align}
\nonumber
\| A f^\delta \|_{\nu} & =  \sum_{j \ge 0} |(A f^\delta)_j|_{\nu} 
= \sum_{j \ge 0} \left| \sum_{m \ge 0} A_{j,m} f_m^\delta \right|_{\nu}
= \sum_{j \ge 0} \left| A_{j,0} f_0^\delta + A_{j,1} f_1^\delta \right|_{\nu}
= \sum_{j = 0}^M \left| A_{j,0} f_0^\delta + A_{j,1} f_1^\delta \right|_{\nu}
\\
\nonumber
& \le \sum_{j = 0}^M \left( \|A_{j,0}\|_{B(\ell_\nu^1)} |f_0^\delta|_{\nu} 
+ \|A_{j,1}\|_{B(\ell_\nu^1)} |f_1^\delta|_{\nu} \right)
 = \sum_{j = 0}^M \left( \|A_{j,0}\|_{B(\ell_\nu^1)} |\delta_a|_{\nu} 
+ \|A_{j,1}\|_{B(\ell_\nu^1)} |\delta_b|_{\nu} \right)
\\
\label{eq:Y0_2}
& \le \left( \sum_{j = 0}^M \left( \|A_{j,0}\|_{B(\ell_\nu^1)}  
+ \|A_{j,1}\|_{B(\ell_\nu^1)} \right) \right) r_0 \bydef Y_0^{(2)}.
\end{align}
Combining \eqref{eq:Y0_1} and \eqref{eq:Y0_2}, we finally set
\begin{equation} \label{eq:Y0_explicit}
Y_0 \bydef Y_0^{(1)} + Y_0^{(2)}.
\end{equation}

\subsection{The \boldmath $Z_0$ \unboldmath bound}

Denote
\[
B \bydef I - A \dagA.
\]
Then, using a similar approach as in the proof of Lemma~\ref{lem:normA}, letting
\begin{equation} \label{eq:Z0_explicit}
Z_0 \bydef 
\max_{m=0,\dots,M} \sum_{j=0}^M \|B_{j,m} \|_{B(\ell_\nu^1)}
\end{equation}
we get that $\|I - A \dagA\|_{B(X^\nu)}=\|B\|_{B(X^\nu)} \le Z_0$.

\subsection{The \boldmath $Z_1$ \unboldmath bound} \label{sec:Z1}


Let $h \in X^{\nu}$ and denote
\[
z = z(h) = \left( Df(\bp)-\dagA \right) h.
\]

Recalling \eqref{eq:f=0_manifold}, as
\[
f_{k,m}(p) = 
\begin{cases}
p_{k,0} - a_k, & m=0, \quad k \ge 0
\\
p_{k,1} - b_k, & m=1, \quad k \ge 0
\\
\left( \tilde \lambda m + i k^2 \omega^2 \right) p_{k,m} - i (p*_{TF}p)_{k,m}, & m \ge 2, \quad k \ge 0,
\end{cases}
\]
then we get that 
\[
z_{k,m} = 
\begin{cases}
0, & m=0, \quad k \ge 0
\\
0, & m=1, \quad k \ge 0
\\
- 2 i (\bp*_{TF} h^{(\infty)})_{k,m}, & m = 2,\dots,M, \quad k = 0,\dots,K
\\
- 2 i (\bp*_{TF} h)_{k,m}, & \text{otherwise}.
\end{cases}
\]
where 
\[
h^{(\infty)} = \left( h_0^{(\infty)},h_1^{(\infty)},\dots,h_M^{(\infty)},h_{M+1},h_{M+2},\dots \right),
\]
where for $m=0,\dots,M$
\[
\left( h^{(\infty)}_{m} \right)_k
\bydef
\begin{cases}
0, & k=0,\dots,K
\\
h_{k,m}, & k > K.
\end{cases}
\]
Now, let us compute component-wise upper bounds for $z_{k,m}$, with $m=0,\dots,M$ and $k = 0,\dots,K$. In this case, using Corollary~\ref{cor:psi_k}, we get that
\begin{align}
\nonumber
|z_{k,m}| 
& \le 2 \left| (\bp*_{TF} h^{(\infty)})_{k,m} \right| 
\\
\nonumber
& \le 2  \sum_{\ell=0}^m  \left| \left( \bp_{\ell} * h_{m-\ell}^{(\infty)} \right)_k \right| 
\\
& \le \hat z_{k,m} \bydef 2  \sum_{\ell=0}^m  \Psi_k (\bp_\ell).
\label{eq:hat z_{k,m}}
\end{align}

Denote $\tilde \lambda = \tilde \lambda_r + i \tilde \lambda_i$ so that
\[
\mu_{k,m}(\tilde \lambda) = \tilde \lambda_r m + i \left(\tilde \lambda_i m + k^2 \omega^2 \right).
\]
Assume that 
\begin{equation} \label{eq:mu_tail_bound_large_K_assumption}
\min_{ m \in \{0,\dots,M\} } \left\{ \tilde \lambda_i m + (K+1)^2 \omega^2 \right\} > 0,
\end{equation}
and in this case define 
\begin{equation} \label{eq:mu_tail_bound_large_K}
\mu^*(\tilde \lambda) \bydef \min_{ m \in \{0,\dots,M\} } \sqrt{ (\tilde \lambda_r m)^2 + \left(\tilde \lambda_i m + (K+1)^2 \omega^2 \right)^2 }.
\end{equation}
Assuming that \eqref{eq:mu_tail_bound_large_K_assumption} holds and since $\tilde \lambda_r > 0$ (it is an unstable eigenvalue), then one can readily verify that
\begin{equation} \label{eq:eig_tail-bound}
|\mu_{k,m}(\tilde \lambda)| \ge 
\begin{cases}
\tilde \lambda_r (M+1), & m>M
\\
\mu^*(\tilde \lambda), & k>K \text{ and } m \in \{0,\dots,M\}.
\end{cases}
\end{equation}
Using inequality \eqref{eq:hat z_{k,m}} and \eqref{eq:eig_tail-bound}, we obtain that
\begin{align}
\nonumber
\| A z \|_\nu & = \sum_{j=0}^\infty | (A z)_j |_\nu  
\\
\nonumber
&= \sum_{j=0}^{M} |(A z)_j|_{\nu} + \sum_{m>M} |A_{m,m} z_m|_{\nu} 
\\
\nonumber
&= \sum_{j=0}^{M} 
\left| \sum_{m=0}^M A_{j,m} z_m \right|_{\nu} 
+ \sum_{m>M} \left| \sum_{k \ge 0} \frac{1}{\mu_{k,m}(\tilde \lambda))} z_{k,m} \right| \omega_k 
\\
\nonumber
&\le \sum_{j=0}^{M} \sum_{k\ge0} \left| \sum_{m=0}^M  (A_{j,m} z_m)_k \right| \omega_k 
+ \frac{1}{\tilde \lambda_r(M+1)} \sum_{m>M} \left| \sum_{k \ge 0} z_{k,m} \right| \omega_k 
\\
\nonumber
&\le \sum_{j=0}^{M} \sum_{k=0}^K  \sum_{m=0}^M  (|A^{(K,M)}_{j,m}| \hat z_m)_k  \omega_k 
+ \sum_{k>K}  \sum_{m=0}^M  |(A_{m,m} z_m)_k|  \omega_k 
 + \frac{2}{\tilde \lambda_r(M+1)} \|  \bp*_{TF} h \|_\nu
\\
\nonumber
&\le \sum_{j=0}^{M} \sum_{k=0}^K  \sum_{m=0}^M  (|A^{(K,M)}_{j,m}| \hat z_m)_k  \omega_k 
+ \sum_{k>K}  \sum_{m=0}^M \frac{1}{|\mu_{k,m}(\tilde \lambda)|} |z_{k,m}|  \omega_k 
 + \frac{2}{\tilde \lambda_r(M+1)} \|  \bp*_{TF} h \|_\nu
 \\
 \nonumber
&\le \sum_{j=0}^{M} \sum_{k=0}^K  \sum_{m=0}^M  (|A^{(K,M)}_{j,m}| \hat z_m)_k  \omega_k 
+ \frac{1}{\mu^*(\tilde \lambda)} \|  \bp*_{TF} h \|_\nu
 + \frac{2}{\tilde \lambda_r(M+1)} \|  \bp*_{TF} h \|_\nu
 \\
&\le \sum_{j,m=0}^{M} \sum_{k=0}^K   (|A^{(K,M)}_{j,m}| \hat z_m)_k  \omega_k 
+ \frac{2\|  \bp \|_\nu}{\mu^*(\tilde \lambda)} 
 + \frac{2 \|  \bp \|_\nu}{\tilde \lambda_r(M+1)}   \bydef Z_1.
 \label{eq:Z1_explicit}
 \end{align}

\subsection{The \boldmath $Z_2$ \unboldmath bound} \label{sec:Z2}

Let $h \in X^{\nu}$ and $c \in B_{r}(\bp)$ and denote
\[
z = z(c,h) = \left( Df(c) - Df(\bp) \right) h.
\]
Recalling \eqref{eq:f=0_manifold}, as
\[
f_{k,m}(p) = 
\begin{cases}
p_{k,0} - a_k, & m=0, \quad k \ge 0
\\
p_{k,1} - b_k, & m=1, \quad k \ge 0
\\
\left( \lambda m + i k^2 \omega^2 \right) p_{k,m} - i (p*_{TF}p)_{k,m}, & m \ge 2, \quad k \ge 0,
\end{cases}
\]
then we get that 
\[
z_{k,m} = 
\begin{cases}
0, & m=0, \quad k \ge 0
\\
0, & m=1, \quad k \ge 0
\\
 - 2 i ((c-\bp)*_{TF}h)_{k,m}, & m \ge 2, \quad k \ge 0.
\end{cases}
\]
Hence,
\[
\|Az\|_{\nu} \le \|A\|_{B(X^{\nu})} \|z\|_{\nu} 
\le \|A\|_{B(X^{\nu})} \left( |2 i| \| c-\bp \|_{\nu} \| h\|_{\nu} \right) \le
2  \|A\|_{B(X^{\nu})} r.
\]
Therefore, we can set 
\begin{equation} \label{eq:Z2_explicit}
Z_2 \bydef 2  \|A\|_{B(X^{\nu})}.
\end{equation}

\subsection{Rigorous computations of unstable sets}

Consider a steady state $\tilde u \in \{u_1^i,u_1^{ii}\}$, where $u_1^i$ and $u_1^{ii}$ are presented in Theorem~\ref{thm:Equilibria} and portrayed in Figure~\ref{fig:equilibria}. Denote by $\ta$ the Fourier coefficients of the steady state $\tilde u$. Fix a Taylor projection $M \ge 2$ and a Fourier projection $K>0$, and using Newton's method compute the corresponding approximate Fourier-Taylor coefficients $\bp \in \C^{(K+1)\times (M+1)}$ of the unstable manifold attached to $\ta$. 

\begin{remark} \label{rem:flexibility_parameters_eigenpairs}
Given a steady state $\ta$ and an eigenpair $(\tilde \lambda,\tb)$ solutions of \eqref{eq:f=0_steady_state} and \eqref{eq:g=0_eigenpair}, respectively, by linearity, $(\tilde \lambda, \alpha e^{i \theta} \tb)$ is also a solution of \eqref{eq:g=0_eigenpair}, for any $\alpha \in \R$ and $\theta \in [0,2\pi)$. The flexibility in the choice of the  angle $\theta$ leads to a one-dimensional parameterized family of (strong) unstable manifolds. In the sequel, we will play with these parameters to generate the manifolds.
\end{remark}

Having obtained explicit and computable formulas for the bounds $Y_0$, $Z_0$, $Z_1$ and $Z_2$ given in \eqref{eq:Y0_explicit}, \eqref{eq:Z0_explicit}, \eqref{eq:Z1_explicit} and \eqref{eq:Z2_explicit}, respectively, and recalling \eqref{eq:general_radii_polynomial}, define the radii polynomial by 
\[
p(r) \bydef Z_2 r^2 - ( 1 - Z_1 - Z_0) r + Y_0.
\]
We wrote a MATLAB program (available at \cite{bib:codes}) which rigorously computes (i.e. controlling the floating point errors) the bounds $Y_0$, $Z_0$, $Z_1$ and $Z_2$. 
The program uses the interval arithmetic library INTLAB available at \cite{Ru99a}.
Fixing $\nu =1$ and using the MATLAB program, verify the existence of $r_p>0$ such that $p(r_p)<0$. By Theorem~\ref{thm:radii_polynomials}, there exists a unique $\tp \in B_{r_p}(\bp)$ such that $f(\tp) = 0$. By construction, the resulting function $\tilde P:\mathbb{D} \to \ell_\nu^1$ given by 
\[
\tilde P(\sigma) \bydef \sum_{m \ge 0} \tp_m \sigma^m
\]
parameterizes the unstable manifold {\em tangent} to $\tb$. Moreover, for each $\sigma \in \mathbb{D}$, $\tilde P(\sigma) \in W^u(\ta)$. Denoting by 
\[
\bar P(\sigma) \bydef \sum_{m \ge 0}^M \bp_m \sigma^m
\]
the finite dimensional approximation, we get the rigorous $C^0$ error bound
\[
\sup_{\sigma \in \mathbb{D}} | \tilde P(\sigma) - \bar P(\sigma)  |_\nu 
= \sup_{\sigma \in \mathbb{D}} \left| \sum_{m \ge 0} \tp_m \sigma^m - \sum_{m \ge 0} \bp_m \sigma^m  \right|_\nu \le \sup_{\sigma \in \mathbb{D}} \sum_{m \ge 0} |\tp_m - \bp_m|_\nu |\sigma|^m \le \| \tp - \bp \|_\nu \le r_p.
\]
We applied this approach to obtain the parameterizations of the local unstable manifolds attached to $u_1^i$, $(u_1^i)^*$, $u_1^{ii}$ and $(u_1^{ii})^*$. For each of the equilibria, we use the flexibility of the angle $\theta$ and the scaling $\alpha$ as described in Remark~\ref{rem:flexibility_parameters_eigenpairs}. We portray these parameterizations in Figure~\ref{fig:unstable_manifolds}. 

\begin{figure}[h!]
	\centering
	\includegraphics[width = .48 \textwidth]{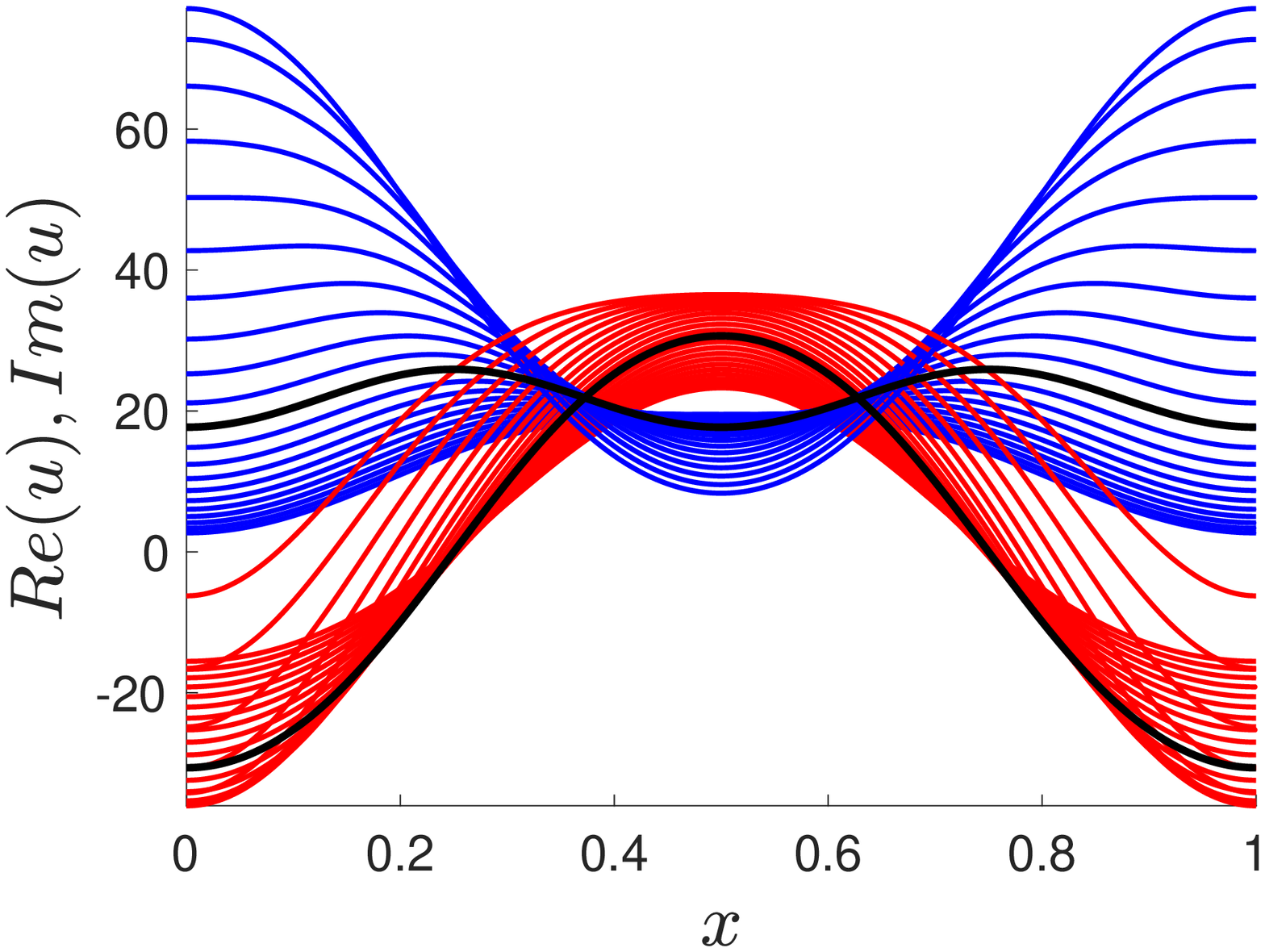}
	\includegraphics[width = .48 \textwidth]{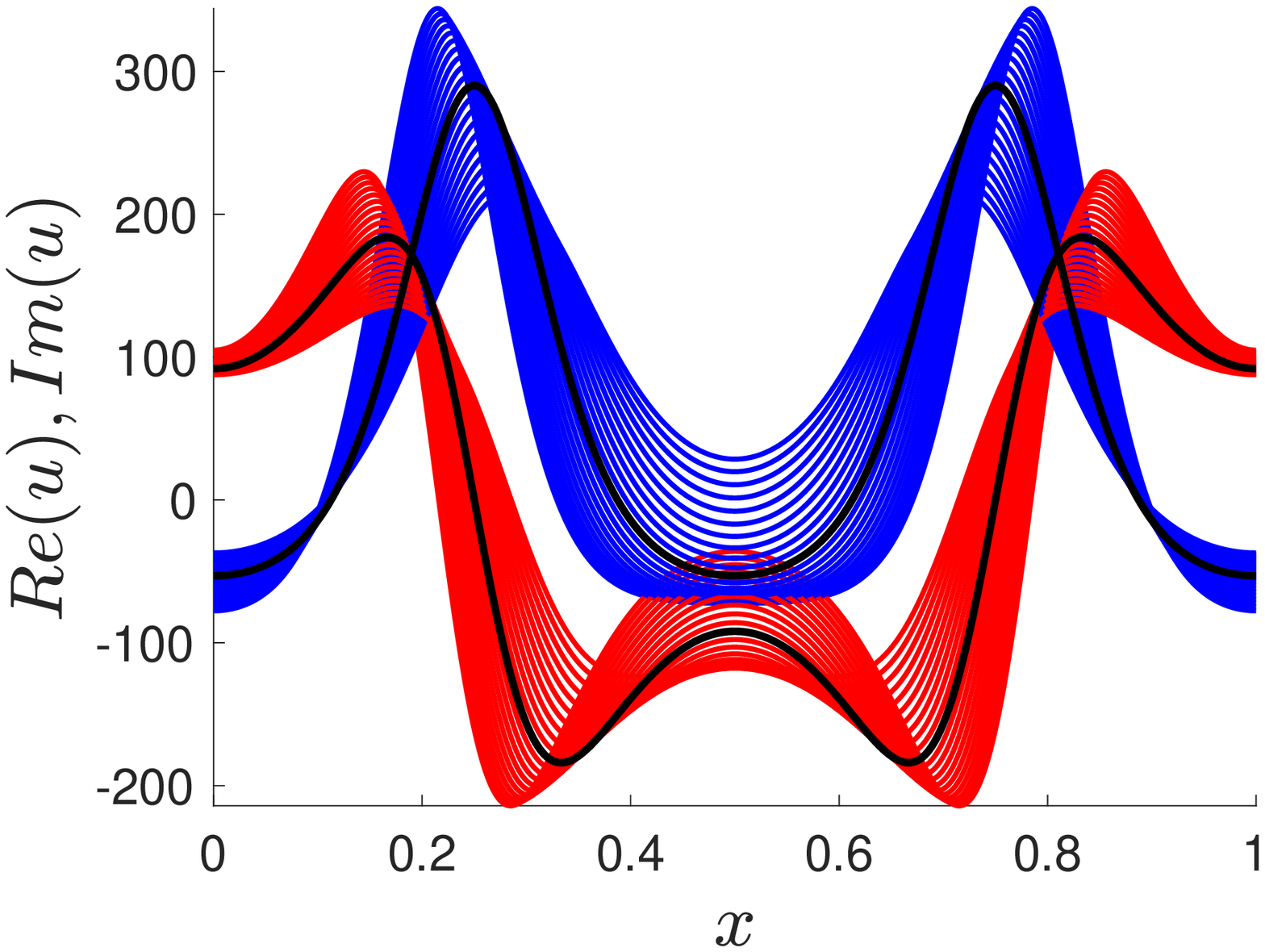}
	\includegraphics[width = .48 \textwidth]{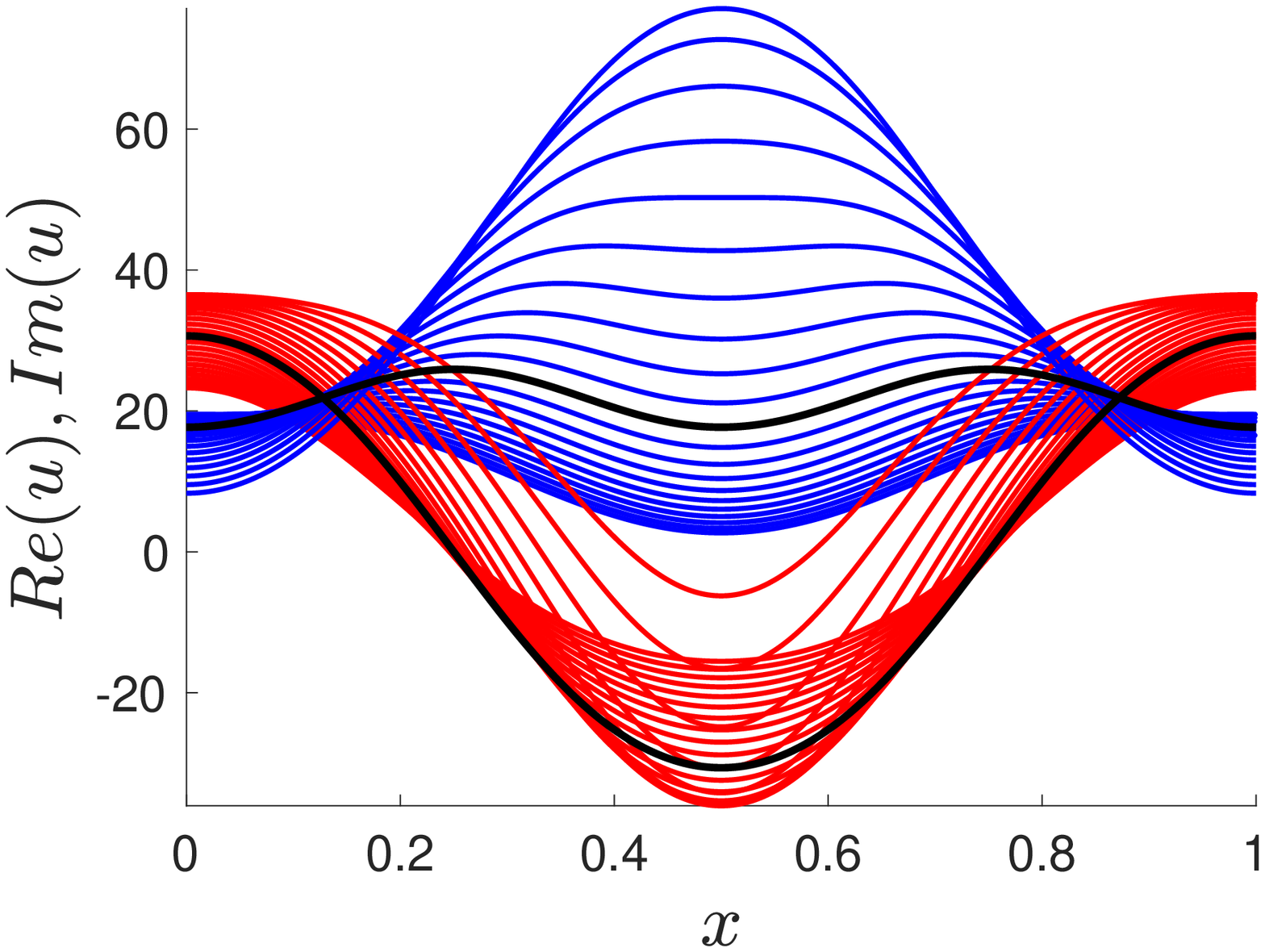}
	\includegraphics[width = .48 \textwidth]{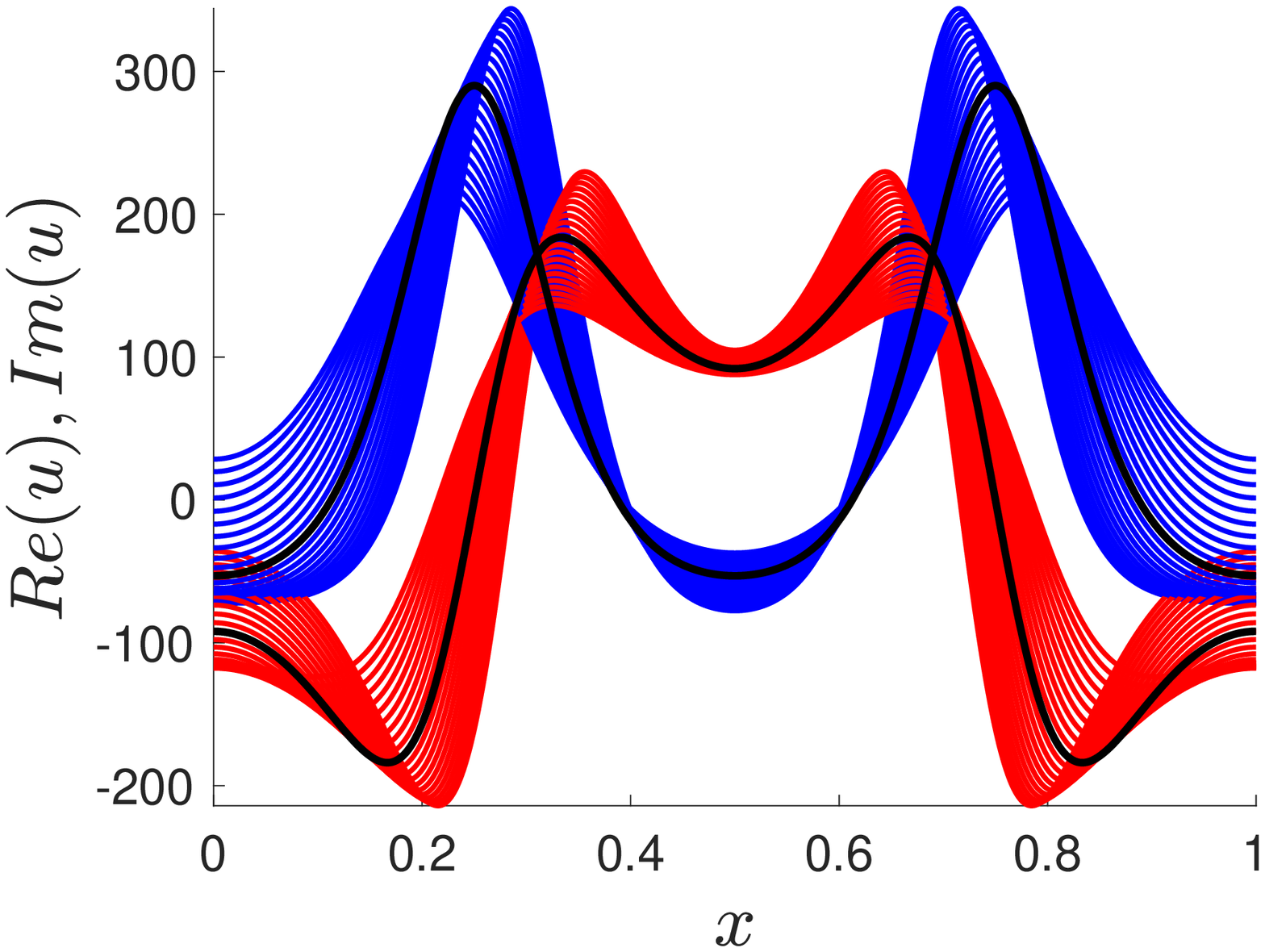}
	\caption{Profiles of the unstable manifolds of the four equilibria: $u^i_1(x)$ (top left), $u^{ii}_1(x)$ (top right), $(u^i_1(x))^*$ (bottom left), $(u^{ii}_1(x))^*$ (bottom right) the existence of which are established via CAPs. The Fourier projections are taken to be $K=27$ for $u^i_1$ and $(u^i_1(x))^*$, and $K=60$ for $u^{ii}_1$ and $(u^{ii}_1)^*$. The Taylor projections are taken to be $M=150$ for $u^i_1$ and $(u^i_1(x))^*$, and $M=60$ for $u^{ii}_1$ and $(u^{ii}_1)^*$. For the scaling of the eigenvectors $\tb$ we fixed to be $\|\tb\|_2 = 20$ for $u^i_1$ and $(u^i_1(x))^*$, and $\|\tb\|_2 = 75$ for $u^{ii}_1$ and $(u^{ii}_1)^*$.}\label{fig:unstable_manifolds}
\end{figure}

\begin{remark}[\bf Analyticity of the solutions]\label{rem:analyticity}
It is important to realize that all bounds $Y_0$, $Z_0$, $Z_1$ and $Z_2$ given in \eqref{eq:Y0_explicit}, \eqref{eq:Z0_explicit}, \eqref{eq:Z1_explicit} and \eqref{eq:Z2_explicit}, respectively, depend continuously on the parameter $\nu$. Hence, the coefficients of the radii polynomial $p(r) = p(r,\nu)$ given in \eqref{eq:general_radii_polynomial} depend continuously on $\nu \ge 1$. In the computer-assisted proofs, we fixed $\nu = 1$ and verified the existence of $r_p>0$ such that $p(r_p,\nu)<0$. By continuity, there exists $\tilde \nu>1$ such that $p(r_p,\tilde \nu)<0$, and hence, the unique solution of $f=0$ satisfies $\tp \in X^{\tilde \nu}$. This in turn implies that for each $\sigma \in \mathbb{D}$, $\tilde P(\sigma) \in \ell_{\tilde \nu}^1$ and therefore the corresponding points on the unstable manifold are analytic in space. 
\end{remark}

%
%
%
%

\section{Rigorous integrator for nonlinear Schr\"odinger equation}
\label{sec:Integrator}

As in the previous section, let $d = 1$ for the dimension of space.
In this section we provide a rigorous numerical integrator for the nonlinear Schr\"odinger equation \eqref{eq:NLS_Quad} forward in time.
This integrator will propagate the endpoint of unstable manifold that is rigorously computed by the parameterization method introduced in Section \ref{sec:UnstableManifold}.
In \cite{takayasu2019rigorous}, the authors with H.~Okamoto have introduced a rigorous integrator for a complex valued nonlinear heat equation
\[
u_t = \expig \left(u_{xx} + u^2\right),\quad \theta\in \left(-\frac{\pi}{2},\frac{\pi}{2}\right).
\]
The case of $\theta = \pi/2$ was not treated in the previous work. Here we extend the integrator in the case of $\theta = \pi/2$, that is nonlinear Schr\"odinger equation \eqref{eq:NLS_Quad}. We also show some modifications to improve the accuracy of rigorous integration of the solutions.

We take the initial data of \eqref{eq:NLS_Quad} by the Fourier series
\[
u(x,0)=\sum_{k \in \Z}\varphi_k e^{ik\omega x},\quad \varphi=(\varphi)_{k\in\mathbb{Z}}\in \ell_{\nu,1}^1.
\]
We consider the Cauchy problem of \eqref{eq:NLS_Quad} under the above initial condition.
Expanding the unknown function $u(t,x)$ by the Fourier series, as defined in \eqref{eq:Fourier_expansion}, we have the infinite system of nonlinear ODEs \eqref{eq:CGL_ODEs} with the initial data $a(0)=\varphi$.
For a fixed time $h>0$, let $J\bydef[0,h]$ be a time step.
We numerically validate the local existence of the solution in the following Banach space:
\[
X\bydef C(J;\ell_{\nu,1}^1),\quad \|a\|_X\bydef\sup_{t\in J}\|a(t)\|,\quad\|a(t)\|=\sum_{k \in \Z}|a_k(t)|\nu^{|k|}~\mbox{for a fixed}~t\in J.
\]
More precisely, letting 
$\bar{a}(t) \bydef\left(\ldots, 0,0, \bar{a}_{-K}(t), \ldots, \bar{a}_{K}(t), 0,0, \ldots\right)$
be an approximation of $a(t)$ with the Fourier projection $K>0$,
we will rigorously include the Fourier coefficients in the neighborhood of numerical solution defined by
\[
B_{J}(\bar{a}, \varrho) \bydef\left\{a \in X:\|a-\bar{a}\|_{X} \leq \varrho,~a(0)=\varphi\right\}.
\]

Let us define the \emph{Laplacian} operator $L$ acting on a bi-infinite sequence $b=(b_k)_{k\in\mathbb{Z}}$ as
\[
Lb\bydef\left(-k^2\omega^2b_k\right)_{k\in\mathbb{Z}}.
\]
and let $D(L)\subset\ell^1_{\nu,1}$ denote the domain of the operator $L$.
Defining an operator for $a\in C^1(J;D(L))$ as
\begin{align}\label{eq:def_F}
(F(a))(t)\bydef\dot{a}(t)-i\left(La(t)+a^2\left(t\right)\right),
\end{align}
we consider the nonlinear ODEs \eqref{eq:CGL_ODEs} as the zero-finding problem $F(a)=0$.
Hence, if $F(a)=0$ with the initial condition $a(0)=\varphi$ holds, such $a(t)$ solves the Cauchy problem \eqref{eq:CGL_ODEs}.
Let us define an operator $T:X\to X$ as
\begin{align}\label{eq:simp_Newton_op}
(T(a))(t)\bydef U(t,0)\varphi+i\int_0^tU(t,s)\left(a^2(s)-2\bar{a}(s)*a(s)\right)ds,\quad t\in J,
\end{align}
where $\left\{U(t,s)\right\}_{0\le s\le t\le h}$ is the evolution operator on the Banach space $\ell^1_{\nu,1}$ defined by a solution map of the following Cauchy problem:
\begin{equation}\label{eq:linearized_problem}
\dot{b}_k(t)+ik^2\omega^2b_k(t)-2i\left(\bar{a}\left(t\right)*b(t)\right)_k=0\quad(k\in\mathbb{Z})
\end{equation}
with any initial data $b(s)=\psi\in\ell^1_\nu$ ($0\le s\le t$). The evolution operator provides the solution of \eqref{eq:linearized_problem} via the relation $b(t)=U(t,s)\psi$.
To validate the existence of the evolution operator, in \cite[Theorem 3.2]{takayasu2019rigorous}, a hypothesis is shown for providing a uniform bound of the evolution operator over the simplex $\cS_h\bydef \{(t,s):0\le s\le t\le h\}$.
More precisely, a computable constant $\bm{W_h}>0$ satisfying
\begin{equation}\label{eq:W_h_constant}
\|b\|_{X}=\sup_{(t,s)\in\cS_h}\|U(t,s)\psi\|\le \bm{W_h}\|\psi\|,\quad\forall\psi\in\ell^1_{\nu,1}
\end{equation}
is obtained by decomposing the solution $b$ by the finite mode $b^{(K')}=\left(b_{k}\right)_{|k|\le K'}$ and the tail  $b^{(\infty)}=\left(b_{k}\right)_{|k|>K'}$ for $K'\in\mathbb{N}$ satisfying $K>K'$.

The local existence of the solution can be validated by the following theorem:

\begin{theorem}[\!\!{\cite[Theorem 4.1]{takayasu2019rigorous}}]\label{thm:local_inclusion}
	Given the approximate solution $\ba$ of \eqref{eq:NLS_Quad} and the initial sequence $\varphi$, assume that $\|\varphi-\ba(0)\|\le\varepsilon$ holds for $\varepsilon\ge 0$.
	Assume also that $\ba\in C^1(J;D(L))$ and any $a\in B_J\left(\ba,\varrho\right)$ satisfies
	$\left\| T(a)-\ba\right\|_X\le f_{\varepsilon}\left(\varrho\right)$,
	where 
	$f_{\varepsilon}\left(\varrho\right)\bydef \bm{W_h}\left[\varepsilon+h\left(2\varrho^2+\delta\right)\right]$.
	Here, $\bm{W_h}>0$ and $\delta> 0$ satisfy
	$\sup_{(t,s)\in\mathcal{S}_h}\left\|U(t,s)\right\|_{B (\ell^1_{\nu,1})}\le \bm{W_h}$ and
	$\left\|F(\ba)\right\|_{X}\le\delta$, respectively.
	If
	$f_{\varepsilon}\left(\varrho\right)\le\varrho$
	holds, then the Fourier coefficients $\ta$ of the solution of \eqref{eq:NLS_Quad} are rigorously included in $B_J\left(\ba,\varrho\right)$ and are unique in $B_J\left(\ba,\varrho\right)$.
\end{theorem}

The proof is based on Banach's fixed-point theorem for the operator $T$ defined in \eqref{eq:simp_Newton_op}.

\begin{remark}
	The case of $\theta=\pi/2$ is not treated in \cite[Theorem 4.1]{takayasu2019rigorous} but this theorem can be extended naturally to that case. The most different part is the constant $\bm{W_h}$ always becomes greater than $1$. This is due to the lack of smoothing effect of the evolution operator (cf., e.g.,\cite{pazy1983semigroups}), which was previously generated by the analytic semigroup. This also implies that the error estimate $\varepsilon$ in the past time step is monotonically increasing and makes it difficult to rigorously integrate over a long period of time. Moreover, the smoothness of the solution depends only on that of initial data. In other words, smoothness of the initial data propagates along the time evolution.
\end{remark}

To prove the existence of heteroclinic orbits from the nonhomogeneous equilibria to the zero equilibrium, rigorous integrator using Theorem \ref{thm:local_inclusion} cannot succeed to connect the end point of unstable manifold to the stable region obtained by validating the hypothesis of Theorem \ref{thm:HomoclinicBlowup}. We therefore introduce some modifications to improve the accuracy of rigorous integrator. 

\subsection{Two improved validating theorems}
The main idea is that we separate the hypothesis of Theorem \ref{thm:local_inclusion} into the zero-mode and other modes. To show more precisely, let us define a finite dimensional (Fourier) projection $\varPi^{(0)}:\ell^1_{\nu,1} \to \ell^1_{\nu,1}$ for a given vector $\phi =\left(\phi_{k}\right)_{k \in \Z } \in \ell^1_{\nu,1}$ as
\[
\left( \varPi^{(0)} \phi \right)_k = \begin{cases} \phi_0, & k=0 \\ 0, & k\neq 0, \end{cases}
\]
for $k \in \Z$. Given $\phi \in \ell^1_{\nu,1}$, we denote $\phi^{(0)} \bydef \varPi^{(0)} \phi \in \ell^1_{\nu,1}$ and $\phi^{(\infty)} \bydef \left( \mathrm{Id} - \varPi^{(0)} \right) \phi \in \ell^1_{\nu,1}$, where $\mathrm{Id}$ is the identity operator on $\ell^1_{\nu,1}$. Thus, $\phi$ is represented by $\phi=\phi^{(0)}+\phi^{(\infty)}$. From this fact, we split the neighborhood $B_J\left(\ba,\varrho\right)$ into the following two parts:
\begin{align}
&B_{J}^{(0)}(\ba,\varrho_0)\bydef\left\{a^{(0)}\in X: \left\|a^{(0)}-\ba^{(0)}\right\|_X\le \varrho_0,~a^{(0)}(0)=\varphi^{(0)}\right\},\label{eq:B0}\\
&B_{J}^{(\infty)}(\ba,\varrho_\infty)\bydef\left\{a^{(\infty)}\in X  : \left\|a^{(\infty)}-\ba^{(\infty)}\right\|_X\le \varrho_\infty,~a^{(\infty)}(0)=\varphi^{(\infty)}\right\}.\label{eq:Binfty}
\end{align}
Using the neighborhood $B_{J}^{(0)}(\ba,\varrho_0)\times B_{J}^{(\infty)}(\ba,\varrho_\infty)$ instead of $B_J\left(\ba,\varrho\right)$, we will have a sharper estimate for the local existence result.

Our first goal here is to show how we numerically prove the existence of the evolution operator $U(t,s)$ arising from \eqref{eq:linearized_problem}, given a step size $h>0$, using the above splitting formulation of $\ell^1_{\nu,1}$. This is equivalent to compute a constant $\bm{W_h}$ satisfying \eqref{eq:W_h_constant}. We therefore give an analogue of the discussion in \cite[Section 3]{takayasu2019rigorous}, that is, we separate \eqref{eq:linearized_problem} by considering yet another Cauchy problem with respect to the sequence $c(t)=(c_k(t))_{k\in \Z }$
\begin{align}
\label{eq:linearizedeq_finite}
&\dot{c}_0(t)-2i\ba_0\left(t\right)c_0(t) =0\\[1mm]
\label{eq:linearizedeq_infinite}
&\dot{c}_k(t)+ik^2\omega^2c_k(t)-2i\left(\ba\left(t\right)*c^{(\infty)}(t)\right)_k =0,\quad(k\neq 0).
\end{align}
with initial data $c_0(s)=\psi_0$ and $c_k(s)=\psi_k$ ($k\neq 0$).
This decoupled formulation, while not being equivalent to \eqref{eq:linearized_problem}, is used to control the evolution operator associated to \eqref{eq:linearized_problem}. Denote by $C^{(0)}(t,s)$ and $C^{(\infty)}(t,s)$ the evolution operators of the finite dimensional equation \eqref{eq:linearizedeq_finite} and the infinite dimensional equation \eqref{eq:linearizedeq_infinite}, respectively. We extend the action of the operator $C^{(0)}(t,s)$ (resp. $C^{(\infty)}(t,s)$) on $\ell^1_{\nu,1}$ by introducing the operator $\bar U^{(0)}(t,s)$ (resp. $\bar U^{(\infty)}(t,s)$) as follows. Given $\phi \in \ell^1_{\nu,1}$, define $\bar U^{(0)}(t,s):\ell^1_{\nu,1}\to \ell^1_{\nu,1}$ and $\bar U^{(\infty)}(t,s):\ell^1_{\nu,1} \to \ell^1_{\nu,1}$ element-wisely as
\begin{align}
\label{eq:bU_m_definition}
\left( \bar U^{(0)}(t,s) \phi \right)_k &\bydef \begin{cases} C^{(0)}(t,s) \phi_0, & k = 0\\ 0, & k\neq 0 \end{cases}
\\
\label{eq:bU_infty_definition}
\left( \bar U^{(\infty)}(t,s) \phi \right)_k &\bydef \begin{cases} 0 , & k=0 \\ \left( C^{(\infty)}(t,s)  \phi^{(\infty)} \right)_k, & k\neq 0.   \end{cases}
\end{align}
We also define the dual space of $\ell^1_{\nu,1}$, say $\ell^\infty_{\nu^{-1}}$, as
\[
\ell^\infty_{\nu^{-1}}\bydef\left\{c=(c_k)_{k\in\Z}:\|c\|_{\infty,\nu^{-1}}\bydef \sup_{k\in \Z}|c_k|\nu^{-|k|}<\infty\right\}.
\]
and
$X_{\nu^{-1}}\bydef C(J;\ell^\infty_{\nu^{-1}})$ with the norm $\|a\|_{X_{\nu^{-1}}}\bydef\sup_{t\in J}\|a(t)\|_{\infty,\nu^{-1}}$.

\begin{theorem}\label{thm:sol_map}
	Let $(t,s)\in \mathcal{S}_h$ and $\ba$ an approximate solution of \eqref{eq:CGL_ODEs}.
	Assume that there exists a constant $W_0>0$ such that
	\begin{equation} \label{eq:bound_W_0}
	\sup_{(t,s)\in \mathcal{S}_h}\left\| \bU^{(0)}(t,s)\right\|_{B(\ell^1_{\nu,1})}\le W_0.
	\end{equation}
	Assume that $C^{(\infty)}(t,s)$ exists and that $\bU^{(\infty)}(t,s)$ defined in \eqref{eq:bU_infty_definition} satisfies
	\begin{equation} \label{eq:assumption_existence_U_infty}
	\left\| \bU^{(\infty)}(t,s)\right\|_{B (\ell^1_{\nu,1})}\le W^{(\infty)}(t,s)\bydef e^{2\int_{s}^{t}\|\ba(\tau)\| d\tau}.
	\end{equation}
	Define the constants $W_{\infty}\ge0$, $\bar{W}_{\infty}\ge0$,  $W_{\infty}^{\sup}>0$ by
	\begin{align}
	\label{eq:def_W_infty}
	W_\infty &\bydef \frac{e^{2h\|\ba\|_{X}}-1}{2\|\ba\|_{X}}\\[1mm]
	\label{eq:def_barW_infty}
	\bar{W}_{\infty} & \bydef
	\frac{W_\infty-h }{2\|\ba\|_{X}}\\[1mm]
	\label{eq:def_W_infty_sup}
	W_{\infty}^{\sup} & \bydef 	e^{2h\|\ba\|_X},
	\end{align}
	respectively.
	Define $\ba^{(\infty)}(t)\bydef (\mathrm{Id}-\varPi^{(0)})\ba(t) \in \ell^1_{\nu,1}$ for each $t\in J$.
	We also define $U^{(0)}(t,s)\bydef \varPi^{(0)}U(t,s):\ell^1_{\nu,1}\to\ell^1_{\nu,1}$ and $U^{(\infty)}(t,s)\bydef \left( {\rm Id} - \varPi^{(0)} \right) U(t,s):\ell^1_{\nu,1}\to\ell^1_{\nu,1}$ for $(t,s)\in\cS_h$.
	If 
	\begin{equation} \label{eq:kappa_condition}
	\kappa\bydef 1-4W_0\bar{W}_\infty\|\ba^{(\infty)} \|_X\|\ba^{(\infty)}\|_{X_{\nu^{-1}}}>0,
	\end{equation}
	then the evolution operator $U(t,s)$ exists and the following estimate holds 
	\begin{align}\label{eq:U_estimate}
	\sup_{(t,s)\in\cS_h}\begin{pmatrix}
	\left\|U^{(0)}(t,s)\phi\right\| \\[2mm]
	\left\|U^{(\infty)}(t,s)\phi\right\| 
	\end{pmatrix}
	\le\bm{U_h}\begin{pmatrix}
	\left\|\phi^{(0)}\right\| \\[2mm]
	\left\|\phi^{(\infty)}\right\| 
	\end{pmatrix}
	\end{align}
	for any $\phi=\phi^{(0)}+\phi^{(\infty)}\in\ell^1_{\nu,1}$, where
	\begin{equation} \label{eq:definition_of_U_h}
	\bm{U_h} \bydef\begin{pmatrix}
	W_0\kappa^{-1} & 2W_0W_\infty\|\ba^{(\infty)}\|_{X_{\nu^{-1}}}\kappa^{-1}\\
	2W_0W_\infty\|\ba^{(\infty)}\|_{X}\kappa^{-1} & W_{\infty}^{\sup}\kappa^{-1}
	\end{pmatrix}.
	\end{equation}
	Here, the inequality for a vector means each inequality holds element-wisely.
\end{theorem}

The proof is given later in Section \ref{sec:proof_thm_sol_map}.
After the proof of Theorem~\ref{thm:sol_map}, we introduce in Section~\ref{sec:W0_bounds} a rigorous computational method based on interval arithmetic to obtain the bound $W_0$ satisfying \eqref{eq:bound_W_0} associated with the finite dimensional evolution operator $C^{(0)}(t,s)$, which gives sharper estimate than our previous work \cite{takayasu2019rigorous}.


\begin{remark}
	Theorem \ref{thm:sol_map} gives a hypothesis of existence of the evolution operator $U(t,s)$ of the linearized problem \eqref{eq:linearized_problem}, which is numerically checked by interval arithmetic. Furthermore, the bound $\bm{W_h}$ satisfying \eqref{eq:W_h_constant} is given by $\bm{W_h}=\|\bm{U_h}\|_1$, where $\|\cdot\|_1$ denotes the 1-norm of matrices.
\end{remark}

Second goal is to introduce a modified version of Theorem \ref{thm:local_inclusion} for numerically proving the local existence of solution in  $B_{J}^{(0)}(\ba,\varrho_0)\times B_{J}^{(\infty)}(\ba,\varrho_\infty)$ as the following theorem:

\begin{theorem}\label{thm:modified_thm}
	Given an initial sequence $\varphi\in\ell^1_{\nu,1}$,  assume that the initial error is split into the zero-mode and the other modes such that
	\[
	\left|\varphi_0-\ba_0(0)\right|\le\varepsilon_0,\quad\|\varphi^{(\infty)}-\ba^{(\infty)}(0)\|\le\varepsilon_\infty.
	\]
	Assume also that $\ba\in C^1(J;D(L))$ and any $a\in B_{J}^{(0)}(\ba,\varrho_0)\times B_{J}^{(\infty)}(\ba,\varrho_\infty)$ for some $\varrho_0$, $\varrho_\infty>0$ satisfy
	\begin{align*}
	\left\|\varPi^{(0)}\left(T(a)-\ba\right)\right\|_X\le f_{\varepsilon}^{(0)}\left(\varrho_0,\varrho_\infty\right),\quad
	\left\|(\mathrm{Id}-\varPi^{(0)})\left(T(a)-\ba\right)\right\|_X\le f_{\varepsilon}^{(\infty)}\left(\varrho_0,\varrho_\infty\right),
	\end{align*}
	where $f_{\varepsilon}^{(0)}$ and $f_{\varepsilon}^{(\infty)}$ are defined by
	\begin{align*}
	\begin{pmatrix}
	f_{\varepsilon}^{(0)}\left(\varrho_0,\varrho_\infty\right)\\
	f_{\varepsilon}^{(\infty)}\left(\varrho_0,\varrho_\infty\right)
	\end{pmatrix}
	\bydef \bm{U_h}\begin{pmatrix}
	\varepsilon_0+h\left(2\left(\varrho_0^2+\varrho_\infty^2\right)+\delta_0\right)\\
	\varepsilon_\infty+h\left(2\left(\varrho_0+\varrho_\infty\right)^2+\delta_\infty\right)
	\end{pmatrix}.
	\end{align*}
	Here $\bm{U_h}$ is defined in \eqref{eq:definition_of_U_h} and $\delta_{j}$ satisfies $\left\|F^{(j)}(\ba)\right\|_X\le\delta_j$ $(j=0,\infty)$.
	If
	\[
	f_{\varepsilon}^{(0)}\left(\varrho_0,\varrho_\infty\right)\le\varrho_0,\quad f_{\varepsilon}^{(\infty)}\left(\varrho_0,\varrho_\infty\right)\le\varrho_\infty
	\]
	hold, then the Fourier coefficients $(\ta^{(0)},\ta^{(\infty)})$ of the solution of \eqref{eq:NLS} are rigorously included in $B_{J}^{(0)}(\ba,\varrho_0)\times B_{J}^{(\infty)}(\ba,\varrho_\infty)$ and are unique in $B_{J}^{(0)}(\ba,\varrho_0)\times B_{J}^{(\infty)}(\ba,\varrho_\infty)$.
\end{theorem}

We will give the proof in Section \ref{sec:proof_modified_thm}.

\begin{remark}
	The initial sequence is now chosen from $\ell^1_{\nu,1}$. The solution is then in the class of $C^1(J;\ell^1_{\nu,1})$ thanks to the bootstrapping arguments with respect to time variables. Furthermore, if one wants to have a classical solution of the nonlinear Schr\"odinger equation \eqref{eq:NLS_Quad}, the initial sequence should be chosen from $\ell^1_{\nu,1}$ with a strictly positive weight $\nu>1$. This is due to the $C^0$-property of the evolution operator (cf.~\cite{pazy1983semigroups}). However, as mentioned in Remark \ref{rem:analyticity}, the continuity with respect to the parameter $\nu$ yields that the solution is analytic in space.
\end{remark}

\subsubsection{Proof of Theorem \ref{thm:sol_map}}\label{sec:proof_thm_sol_map}

The proof of Theorem~\ref{thm:sol_map} uses the following elementary result:

\begin{lemma} \label{lem:ineqW_infty}
	Consider the constants $W_{\infty}\ge0$, $\bar{W}_{\infty}\ge0$ and  $W_{\infty}^{\sup}>0$ as defined in \eqref{eq:def_W_infty}, \eqref{eq:def_barW_infty} and \eqref{eq:def_W_infty_sup}, respectively. 
	Then $W^{(\infty)}$, defined in \eqref{eq:assumption_existence_U_infty}, obeys the following inequalities:

	\begin{align}
	\label{eq:ineqW_infty_sup}
	\sup_{(t,s)\in \mathcal{S}_h}	W^{(\infty)}(t,s)  & \le W_{\infty}^{\sup}
	\\
	\label{eq:ineqW_infty}
	\sup _{(t,s)\in \mathcal{S}_h} \int_{s}^{t} W^{(\infty)}(\tau, s) d \tau,\quad\sup _{(t,s)\in \mathcal{S}_h} \int_{s}^{t} W^{(\infty)}(t, \tau) d \tau &\le W_\infty
	\\
	\label{eq:ineqbarW_infty}
	\sup _{(t,s)\in \mathcal{S}_h}\int_{s}^{t} \int_{s}^{\tau} W^{(\infty)}(\tau, \sigma) d \sigma d \tau,\quad \sup _{(t,s)\in \mathcal{S}_h}\int_{s}^{t} W^{(\infty)}(t, \tau)(\tau -s) d \tau & \le  \bar{W}_{\infty}.
	\end{align}
\end{lemma}

\begin{proof}
	First, note that from \eqref{eq:assumption_existence_U_infty}
	\begin{align*}
	\sup_{(t,s)\in \mathcal{S}_h}	W^{(\infty)}(t,s) 
	= \sup_{(t,s)\in \mathcal{S}_h}e^{2\int_{s}^{t}\|\ba(\tau)\|d \tau}
	\le e^{2h\|\ba\|_{X}} = W_{\infty}^{\sup}.
	\end{align*}
	Second, note that from \eqref{eq:assumption_existence_U_infty}
	\begin{align*}
	\sup _{(t,s)\in \mathcal{S}_h} \int_{s}^{t} W^{(\infty)}(\tau, s) d \tau
	&=\sup _{(t,s)\in \mathcal{S}_h}\left(\int_{s}^{t} e^{2\int_{s}^{\tau}\|\ba(\sigma)\| d\sigma}d\tau\right)\\
	&\le\sup _{(t,s)\in \mathcal{S}_h}\left(\int_{s}^{t} e^{2\|\ba\|_X(\tau-s)}d\tau\right)\\
	&=\sup _{(t,s)\in \mathcal{S}_h}\left(\frac{e^{2 \|\ba\|_{X}(t-s)}-1}{2\|\ba\|_{X}}\right)\\
	&\le \frac{e^{2h\|\ba\|_{X}}-1}{2\|\ba\|_{X}} = W_\infty
	\end{align*}
	and that
	\begin{align*}
	\sup _{(t,s)\in \mathcal{S}_h} \int_{s}^{t} W^{(\infty)}(t, \tau) d \tau
	&=\sup _{(t,s)\in \mathcal{S}_h}\left(\int_{s}^{t} e^{2\int_{\tau}^{t}\|\ba(\sigma)\| d\sigma}d\tau\right)\\
	&\le\sup _{(t,s)\in \mathcal{S}_h}\left(\int_{s}^{t} e^{2\|\ba\|_X(t-\tau)}d\tau\right)\\
	&=\sup _{(t,s)\in \mathcal{S}_h}\left(\frac{e^{2 \|\ba\|_{X}(t-s)}-1}{2\|\ba\|_{X}}\right)\\
	&\le \frac{e^{2 h\|\ba\|_{X}}-1}{2\|\ba\|_{X}}=W_\infty.
	\end{align*}
	Third, 
	\begin{align*}
	\int_{s}^{t} \int_{s}^{\tau} W^{(\infty)}(\tau, \sigma) d \sigma d \tau
	&\le \int_{s}^{t} \frac{e^{2 \|\ba\|_{X}(\tau-s)}-1}{2\|\ba\|_{X}}d \tau\\
	&=\frac{1}{2\|\ba\|_{X}}\left[\frac{e^{2 \|\ba\|_{X}(\tau-s)}}{2\|\ba\|_{X}}-\tau\right]_{\tau=s}^{\tau=t}\\
	&=\frac{1}{2\|\ba\|_{X}}\left[\frac{e^{2 \|\ba\|_{X}(t-s)}-1}{2\|\ba\|_{X}}-(t-s)\right]
	\end{align*}
	and it follows that
	\[
	\sup _{(t,s)\in \mathcal{S}_h}\int_{s}^{t} \int_{s}^{\tau} W^{(\infty)}(\tau, \sigma) d \sigma d \tau\le \frac{1}{2\|\ba\|_{X}}\left(\frac{e^{2h \|\ba\|_{X}}-1}{2\|\ba\|_{X}}-h\right) = \bar{W}_{\infty}.
	\]
	Finally,
	\begin{align*}
	\int_{s}^{t} W^{(\infty)}(t,\tau) (\tau -s) d \tau
	&\le \int_{s}^{t} e^{2 \|\ba\|_{X}(t-\tau)} (\tau -s)d \tau\\
	&=\frac{-1}{2\|\ba\|_{X}}\left[\frac{e^{2 \|\ba\|_{X}(t-\tau)}}{2\|\ba\|_{X}}+e^{2 \|\ba\|_{X}(t-\tau)}(\tau -s)\right]_{\tau=s}^{\tau=t}\\
	&=\frac{1}{2\|\ba\|_{X}}\left[\frac{e^{2 \|\ba\|_{X}(t-s)}-1}{2\|\ba\|_{X}}-(t-s)\right]
	\end{align*}
	and hence, it follows that
	\[
	\sup _{(t,s)\in \mathcal{S}_h}\int_{s}^{t} W^{(\infty)}(t,\tau)(\tau -s) d \tau\le \frac{1}{2\|\ba\|_{X}}\left(\frac{e^{2h \|\ba\|_{X}}-1}{2\|\ba\|_{X}}-h\right) = \bar{W}_{\infty}.\qedhere
	\]
\end{proof}

\begin{proof}[Proof of Theorem~\ref{thm:sol_map}]
	First note that for $k=0$, the system of differential equations \eqref{eq:linearized_problem} is described by the following non-homogeneous equation:
	\begin{equation} \label{eq:non-homogeneous-finite-system}
	\dot{b}_0(t)-2i\ba_0\left(t\right)b_0(t) =2i\left(\ba\left(t\right)*b^{(\infty)}(t) \right)_0
	\end{equation}
	with the initial value $b_0(s)=\phi_0$.	
	Consider the homogeneous equation \eqref{eq:linearizedeq_finite} and denote by $C^{(0)}(t,s)$ the solution of the variational problem (also called integrating factor).
	Let $\bU^{(0)}(t,s)$ be the extension of the action of $C^{(0)}(t,s)$ on $\ell^1_{\nu,1}$.
	By variation of constants \eqref{eq:non-homogeneous-finite-system} is transformed into the integral equation using the evolution operator $\bU^{(0)}(t,s):\ell^1_{\nu,1} \to \ell^1_{\nu,1}$	
	\begin{equation} \label{eq:b^{m}_integral_equation}
	b^{(0)}(t)=\bU^{(0)}(t, s) \phi^{(0)} + 2 i\int_{s}^{t} \bU^{(0)}(t, \tau)\,\varPi^{(0)}\left(\ba(\tau) * b^{(\infty)}(\tau)\right) d \tau.
	\end{equation}
	Here, for the case of $k=0$,
	\[
	\left(\ba * b^{(\infty)}\right)_{0} = \sum_{\substack{k_{1}+k_{2}=0\\\left|k_{1}\right| \leq K, k_{2}\neq 0}} \ba_{k_{1}} b_{k_{2}} = \sum_{m\neq 0,~|m|\le K} \ba_{m} b_{-m}
	\]
	holds. Taking the absolute value, we have
	\begin{equation}\label{eq:ba*binf_0}
	\left\|\varPi^{(0)}\left(\ba * b^{(\infty)}\right)\right\|=\left|\sum_{m\neq 0,~|m|\le K} \ba_{m} b_{-m}\right|\le \max_{m\neq 0}\frac{|\ba_m|}{\nu^{|m|}}\sum_{m\neq 0,~|m|\le K}|b_m|\nu^{|m|}\le \|\ba^{(\infty)}\|_{\infty,\nu^{-1}}\|b^{(\infty)}\|.
	\end{equation}
	Combining \eqref{eq:bound_W_0} and \eqref{eq:b^{m}_integral_equation}, and using the inequality \eqref{eq:ba*binf_0}, it follows that
	\begin{equation}\label{eq:zero-mode}
	\|b^{(0)}(t)\|  \leq W_0\|\phi^{(0)}\| +2 W_0\int_{s}^{t}\| \ba^{(\infty)}(\tau)\|_{\infty,\nu^{-1}} \|b^{(\infty)}(\tau)\|  d \tau.
	\end{equation}

	Next, for the case of $k\neq 0$, we rewrite the system of differential equations \eqref{eq:linearized_problem} as
	\begin{equation} \label{eq:non-homogeneous-infinite-system}
	\dot{b}_k(t)+ik^2\omega^2b_k(t)-2i\left(\ba\left(t\right)*b^{(\infty)} \right)_k= 2i\left(\ba\left(t\right)*b^{(0)} \right)_k\quad(k\neq 0)
	\end{equation}
	with the initial sequence $b_k(s)=\phi_k$ for $k\neq 0$. 
	Define $\bU^{(\infty)}(t,s)$ as in \eqref{eq:bU_infty_definition}. 
	Using that operator, the system \eqref{eq:non-homogeneous-infinite-system} is described as
	\begin{equation} \label{eq:b^{infty}_integral_equation}
	b^{(\infty)}(t)=\bU^{(\infty)}(t, s) \phi^{(\infty)} + 2 i \int_{s}^{t} \bU^{(\infty)}(t, \tau) \left({\rm Id} -\varPi^{(0)}\right)\left(\ba(\tau)*b^{(0)}(\tau)\right) d \tau.
	\end{equation}
	For $k\neq 0$, we have
	\begin{equation}\label{eq:ba*b0_inf}
	\left\|\left({\rm Id} -\varPi^{(0)}\right)\left(\ba * b^{(0)}\right)\right\| = \sum_{k\neq 0}|\ba_{k} b_{0}|\nu^{|k|}\le \|\ba^{(\infty)}\||b_0| = \|\ba^{(\infty)}\|\|b^{(0)}\|.
	\end{equation}
	Combining \eqref{eq:assumption_existence_U_infty}, \eqref{eq:b^{infty}_integral_equation} and using the inequality \eqref{eq:ba*b0_inf}, it follows that
	\begin{equation}\label{eq:tail-mode}
	\|b^{(\infty)}(t)\| \le W^{(\infty)}(t,s)\|\phi^{(\infty)}\| +2\int_{s}^{t} W^{(\infty)}(t, \tau)\|\ba^{(\infty)}(\tau)\| \|b^{(0)}(\tau)\|  d \tau.
	\end{equation}
	Plugging \eqref{eq:tail-mode} into \eqref{eq:zero-mode}, and using the inequalities \eqref{eq:ineqW_infty} and \eqref{eq:ineqbarW_infty} from Lemma~\ref{lem:ineqW_infty},
	we have
	\begin{align}\nonumber
	\|b^{(0)}(t)\| 
	&\le W_0\|\phi^{(0)}\| +2 W_0\int_{s}^{t}\|\ba^{(\infty)}(\tau)\|_{\infty,\nu^{-1}} \Bigg\{
	W^{(\infty)}(\tau, s)\|\phi^{(\infty)}\| \\\nonumber
	&\hphantom{=}\quad +2 \int_{s}^{\tau} W^{(\infty)}(\tau, \sigma)\|\ba^{(\infty)}(\sigma)\| \|b^{(0)}(\sigma)\| d\sigma
	\Bigg\}d \tau\\\nonumber
	&=W_0\|\phi^{(0)}\| +\left(2 W_0\int_{s}^{t}\|\ba^{(\infty)}(\tau)\|_{\infty,\nu^{-1}}  W^{(\infty)}(\tau, s) d \tau\right)\|\phi^{(\infty)}\| \\\nonumber
	&\hphantom{=}\quad +4 W_0\int_{s}^{t}\|\ba^{(\infty)}(\tau)\|_{\infty,\nu^{-1}}\left(\int_{s}^{\tau}W^{(\infty)}(\tau, \sigma) \|\ba^{(\infty)}(\sigma)\|  \|b^{(0)}(\sigma)\|  d \sigma\right)d \tau\\
	&\le W_0\|\phi^{(0)}\| +2 W_0W_\infty\|\ba^{(\infty)}\|_{X_{\nu^{-1}}}\|\phi^{(\infty)}\| 
	+4 W_0\bar{W}_{\infty}\|\ba^{(\infty)}\|_{X}\|\ba^{(\infty)}\|_{X_{\nu^{-1}}}
	\|b^{(0)}\|_{X}.
	\label{eq:zero-mode_estimate}
	\end{align}
	
	By assumption \eqref{eq:kappa_condition}, $\kappa = 1-4W_0\bar{W}_\infty\|\ba^{(\infty)} \|_X\|\ba^{(\infty)}\|_{X_{\nu^{-1}}}>0$ and using \eqref{eq:zero-mode_estimate} yields that
	\begin{equation}\label{eq:b_m_norm}
	\|b^{(0)}\|_X\le \frac{W_0\|\phi^{(0)}\| +2 W_0W_\infty\|\ba^{(\infty)}\|_{X_{\nu^{-1}}} \|\phi^{(\infty)}\| }{\kappa},
	\end{equation}
	which guarantees the existence of the solution of the zero mode of \eqref{eq:linearized_problem}.
	Conversely, plugging \eqref{eq:zero-mode} into \eqref{eq:tail-mode}, and using the inequalities \eqref{eq:ineqW_infty} and \eqref{eq:ineqbarW_infty} from Lemma~\ref{lem:ineqW_infty}, we have
	\begin{align}\label{eq:tail-mode_estimate}\nonumber
	\|b^{(\infty)}(t)\|
	&\le W^{(\infty)}(t,s)\|\phi^{(\infty)}\| +2\int_{s}^{t} W^{(\infty)}(t, \tau)\|\ba^{(\infty)}(\tau)\| \Bigg\{ W_0\|\phi^{(0)}\|\\\nonumber
	&\hphantom{\le}\quad +2 W_0\int_{s}^{\tau}\| \ba^{(\infty)}(\sigma)\|_{\infty,\nu^{-1}} \|b^{(\infty)}(\sigma)\|  d \sigma \Bigg\}  d \tau\\\nonumber
	&= W^{(\infty)}(t,s)\|\phi^{(\infty)}\| +\left(2W_0\int_{s}^{t} W^{(\infty)}(t, \tau) \|\ba^{(\infty)}(\tau)\|  d\tau \right)\|\phi^{(0)}\|\\\nonumber
	&\hphantom{=}\quad + 4W_0\int_{s}^{t} W^{(\infty)}(t, \tau)\|\ba^{(\infty)}(\tau)\|\left(\int_{s}^{\tau}\| \ba^{(\infty)}(\sigma)\|_{\infty,\nu^{-1}} \|b^{(\infty)}(\sigma)\|  d \sigma\right) d\tau\\
	& \le W^{(\infty)}(t,s)\|\phi^{(\infty)}\| +2W_0W_\infty\|\ba^{(\infty)}\|_X \|\phi^{(0)}\|+4 W_0\bar{W}_{\infty}\|\ba^{(\infty)}\|_{X}\|\ba^{(\infty)}\|_{X_{\nu^{-1}}}
	\|b^{(\infty)}\|_{X}.
	\end{align}
	By assumption \eqref{eq:kappa_condition}, $\kappa = 1-4W_0\bar{W}_\infty\|\ba^{(\infty)} \|_X\|\ba^{(\infty)}\|_{X_{\nu^{-1}}}>0$ and using the inequality \eqref{eq:ineqW_infty_sup} in Lemma~\ref{lem:ineqW_infty}, the tail \eqref{eq:tail-mode_estimate} follows that
	\begin{align}
	\|b^{(\infty)}\|_{X} &\le \frac{W_{\infty}^{\sup}\|\phi^{(\infty)}\| + 2W_0W_\infty\|\ba^{(\infty)} \|_X\|\phi^{(0)}\|}{\kappa}.\label{eq:b_inf_norm}
	\end{align}
	Then it is shown that there exists the evolution operator $U(t,s)$ of \eqref{eq:linearized_problem}.
	Finally, for any initial data $\phi=\phi^{(0)}+\phi^{(\infty)}\in\ell^1_{\nu,1}$, \eqref{eq:b_m_norm} and \eqref{eq:b_inf_norm} yield
	\begin{align*}
	\sup_{(t,s)\in\cS_h}\begin{pmatrix}
	\left\|U^{(0)}(t,s)\phi\right\| \\[2mm]
	\left\|U^{(\infty)}(t,s)\phi\right\| 
	\end{pmatrix}
	&=\begin{pmatrix}
	\|b^{(0)}\|_{X}\\[2mm]\|b^{(\infty)}\|_{X}
	\end{pmatrix}\\
	&\le\begin{pmatrix}
	W_0\kappa^{-1} & 2W_0W_\infty\|\ba^{(\infty)}\|_{X_{\nu^{-1}}}\kappa^{-1}\\
	2W_0W_\infty\|\ba^{(\infty)} \|_X\kappa^{-1} & W_{\infty}^{\sup}\kappa^{-1}
	\end{pmatrix}\begin{pmatrix}
	\|\phi^{(0)}\| \\\|\phi^{(\infty)}\| 
	\end{pmatrix}\\
	&= \bm{U_h}\begin{pmatrix}
	\|\phi^{(0)}\| \\\|\phi^{(\infty)}\| 
	\end{pmatrix}. \qedhere
	\end{align*}
\end{proof}

\subsubsection{Proof of Theorem \ref{thm:modified_thm}}\label{sec:proof_modified_thm}

In this part, we will show the proof of Theorem \ref{thm:modified_thm}, which is essentially the same as \cite[Theorem 4.1]{takayasu2019rigorous}. The difference is the  choice of the neighborhood $B_{J}^{(0)}(\ba,\varrho_0)\times B_{J}^{(\infty)}(\ba,\varrho_\infty)$ defined in \eqref{eq:B0} and \eqref{eq:Binfty}.

\begin{proof}[Proof of Theorem \ref{thm:modified_thm}]
	We will prove the operator $T$ defined in \eqref{eq:simp_Newton_op} becomes a contraction mapping on $B_{J}^{(0)}(\ba,\varrho_0)\times B_{J}^{(\infty)}(\ba,\varrho_\infty)$. Firstly, recall $U^{(0)}(t,s)= \varPi^{(0)}U(t,s):\ell^1_{\nu,1}\to\ell^1_{\nu,1}$ and $U^{(\infty)}(t,s)= \left( {\rm Id} - \varPi^{(0)} \right) U(t,s):\ell^1_{\nu,1}\to\ell^1_{\nu,1}$ for each $(t,s)\in \cS_h$ defined in Theorem \ref{thm:sol_map}.
	For any $a\in B_{J}^{(0)}(\ba,\varrho_0)\times B_{J}^{(\infty)}(\ba,\varrho_\infty)$, we have from \eqref{eq:def_F}, \eqref{eq:simp_Newton_op}, and properties of the evolution operator (cf., e.g., \cite{pazy1983semigroups})
	\begin{align*}
	&T(a)-\ba\\
	&= U(t,0)\varphi+i\int_0^tU(t,s)\left(a^2(s)-2\bar{a}(s)*a(s)\right)ds - \ba\\
	&=U(t,0)\varphi+i\int_0^tU(t,s)\left(a^2(s)-2\bar{a}(s)*a(s)\right)ds - U(t,t)\ba\\
	&=U(t,0)(\varphi-\ba(0))+i\int_0^tU(t,s)\left(a^2(s)-2\bar{a}(s)*a(s)\right)ds - \int_0^t\frac{d}{ds}\left(U(t,s)\ba(s)\right)ds\\
	&=U(t,0)(\varphi-\ba(0))+i\int_0^tU(t,s)\left(a^2(s)-2\bar{a}(s)*a(s)\right)ds - \int_0^t\left(\frac{d}{ds}U(t,s)\ba(s)+U(t,s)\frac{d}{ds}\ba(s)\right)ds\\
	&=U(t,0)(\varphi-\ba(0))+\int_0^tU(t,s)\left(ia^2(s)-2i\bar{a}(s)*a(s)+iL\ba(s)+2i\ba(s)*\ba(s)-\frac{d}{ds}\ba(s)\right)ds\\
	&=U(t,0)(\varphi-\ba(0))+\int_0^tU(t,s)\left[ia^2(s)-i\ba^2(s)-2i\bar{a}(s)*(a(s)-\ba(s))-(F(\ba))(s)\right]ds\\
	&=U(t,0)(\varphi-\ba(0))+\int_0^tU(t,s)\left[i(a(s)-\ba(s))*(a(s)-\ba(s))-(F(\ba))(s)\right]ds.
	\end{align*}
	Let $\chi\bydef a-\ba$. It follows from \eqref{eq:U_estimate} that
	\begin{align}\label{eq:Ta-ta_norm_divided}
	\begin{pmatrix}
	\left\|\varPi^{(0)}(T(a)-\ba)\right\|_X\\[2mm]
	\left\|\left( {\rm Id} - \varPi^{(0)} \right)(T(a)-\ba)\right\|_X
	\end{pmatrix}&\le\sup_{t\in J}
	\begin{pmatrix}
	\left\| U^{(0)}(t,0)\chi(0)\right\| +\int_0^t\left\|U^{(0)}(t,s)\xi(s)\right\| ds\\[2mm]
	\left\| U^{(\infty)}(t,0)\chi(0)\right\| +\int_0^t\left\|U^{(\infty)}(t,s)\xi(s)\right\| ds
	\end{pmatrix}\nonumber\\
	&\le \bm{U_h}\begin{pmatrix}
	\|\chi^{(0)}(0)\| \\[2mm]
	\|\chi^{(\infty)}(0)\| 
	\end{pmatrix} + \bm{U_h}\begin{pmatrix}
	h\sup_{s\in J}\|\xi^{(0)}(s)\| \\[2mm]
	h\sup_{s\in J}\|\xi^{(\infty)}(s)\| 
	\end{pmatrix},
	\end{align}
	where $\xi(s)\bydef i\chi^2(s)-\left(F(\ba)\right)(s)$.
	From the assumption of this theorem,
	$\|\chi^{(0)}(0)\|=\left|\varphi_0-\ba_0(0)\right| \le\varepsilon_0$ and $\|\chi^{(\infty)}(0)\|=\|\varphi^{(\infty)}-\ba^{(\infty)}(0)\|  \le\varepsilon_\infty$
	hold. Furthermore, from the definition of $\xi(s)$, we have using the property of Banach algebra for $\ell^1_{\nu,1}$
	\begin{align*}
	\|\xi^{(0)}(s)\| &\le 2\left|(\chi(s)*\chi(s))_0\right|+\left\|\left(F^{(0)}(\ba)\right)(s)\right\| \le 2\left(\|\chi^{(0)}(s)\| ^2+\|\chi^{(\infty)}(s)\| ^2\right)+\delta_0\\
	\|\xi^{(\infty)}(s)\| &\le 2\sum_{k\neq 0}\left|\left(\chi(s)*\chi(s)\right)_k\right|\nu^{|k|}+\left\|\left(F^{(\infty)}(\tilde{a})\right)(s)\right\| \le 2\|\chi(s)\| ^2+\delta_\infty.
	\end{align*}
	Since $a\in B_{J}^{(0)}(\ba,\varrho_0)\times B_{J}^{(\infty)}(\ba,\varrho_\infty)$, $\|\chi^{(0)}\|_{X}\le \varrho_0$ and $\|\chi^{(\infty)}\|_{X}\le \varrho_\infty$ hold. It follows that $\|\chi\|_{X}\le\|\chi^{(0)}\|_{X}+\|\chi^{(\infty)}\|_{X}\le \varrho_0+\varrho_\infty$. Using \eqref{eq:U_estimate}, \eqref{eq:Ta-ta_norm_divided} is bounded by
	\begin{align}\label{eq:onto_divided}
	\begin{pmatrix}
	\left\|\varPi^{(0)}(T(a)-\ba)\right\|_X\\[2mm]
	\left\|\left( {\rm Id} - \varPi^{(0)} \right)(T(a)-\ba)\right\|_X
	\end{pmatrix}
	\le\bm{U_h}\begin{pmatrix}
	\varepsilon_0+h\left(2\left(\varrho_0^2+\varrho_\infty^2\right)+\delta_0\right)\\[2mm]
	\varepsilon_\infty+h\left(2\left(\varrho_0+\varrho_\infty\right)^2+\delta_\infty\right)
	\end{pmatrix}
	=\begin{pmatrix}
	f_{\varepsilon}^{(0)}\left(\varrho_0,\varrho_\infty\right)\\[2mm]
	f_{\varepsilon}^{(\infty)}\left(\varrho_0,\varrho_\infty\right)
	\end{pmatrix}.
	\end{align}
	From the assumption $f_{\varepsilon}^{(0)}\left(\varrho_0,\varrho_\infty\right)\le\varrho_0$ and $f_{\varepsilon}^{(\infty)}\left(\varrho_0,\varrho_\infty\right)\le\varrho_\infty$, \eqref{eq:onto_divided} yields that $T(a)\in B_{J}^{(0)}(\ba,\varrho_0)\times B_{J}^{(\infty)}(\ba,\varrho_\infty)$ holds for any $a\in B_{J}^{(0)}(\ba,\varrho_0)\times B_{J}^{(\infty)}(\ba,\varrho_\infty)$.

	Secondly, we will show the contraction property of $T$.
	For $a_1$, $a_2\in B_{J}^{(0)}(\ba,\varrho_0)\times B_{J}^{(\infty)}(\ba,\varrho_\infty)$, we define the distance of $B_{J}^{(0)}(\ba,\varrho_0)\times B_{J}^{(\infty)}(\ba,\varrho_\infty)$ as
	\begin{align}\label{eq:def_distance}
	\mathbf{d}(a_1,a_2)\bydef \max\left\{\frac{\left\|\varPi^{(0)}(a_1-a_2)\right\|_X}{\varrho_0},\frac{\left\|\left( {\rm Id} - \varPi^{(0)} \right)(a_1-a_2)\right\|_X}{\varrho_\infty}\right\},
	\end{align}
	where $\varrho_0$ and $\varrho_\infty$ are strictly positive.
	This is associated with the following norm of vectors:
	\begin{align}\label{eq:def_sm_norm}
	\|x\|_{w}\bydef\max_{j=1,2}\frac{|x_j|}{w_j},\quad x=(x_1,x_2)^T
	\end{align}
	with the strictly positive scaling vector $w = (w_1,w_2)^T>0$.
	The analogous discussion above yields from \eqref{eq:simp_Newton_op} and \eqref{eq:U_estimate}
	\begin{align}\label{eq:Ta-Tb_norm_divided}
	\begin{pmatrix}
	\left\|\varPi^{(0)}(T(a_1)-T(a_2))\right\|_X \\[2mm]
	\left\|\left( {\rm Id} - \varPi^{(0)} \right)(T(a_1)-T(a_2))\right\|_X
	\end{pmatrix}\le\sup_{t\in J}
	\begin{pmatrix}
	\int_0^t\left\|U^{(0)}(t,s)\tilde{\xi}(s)\right\| ds\\[2mm]
	\int_0^t\left\|U^{(\infty)}(t,s)\tilde{\xi}(s)\right\| ds
	\end{pmatrix}\le \bm{U_h}\begin{pmatrix}
	h\sup_{s\in J}\|\tilde{\xi}^{(0)}(s)\| \\[2mm]
	h\sup_{s\in J}\|\tilde{\xi}^{(\infty)}(s)\| 
	\end{pmatrix},
	\end{align}
	where 
	\begin{align*}
	\tilde{\xi}(s)&\bydef2i\left(\int_0^1\left[\eta (a_1-\ba)+(1-\eta)(a_2-\ba)\right]d\eta*\zeta\right)
	\end{align*}
	and $\zeta\bydef a_1-a_2$.
	Since $B_{J}^{(0)}(\ba,\varrho_0)\times B_{J}^{(\infty)}(\ba,\varrho_\infty)$ is convex set, $\eta (a_1-\ba)+(1-\eta)(a_2-\ba)\in B_{J}^{(0)}(0,\varrho_0)\times B_{J}^{(\infty)}(0,\varrho_\infty)$ holds for any $\eta\in(0,1)$.
	From the definition of $\tilde{\xi}$, it follows that
	\begin{align*}
	\|\tilde{\xi}^{(0)}(s)\| &\le 2\left(\varrho_0\|\zeta^{(0)}(s)\| +\varrho_\infty\|\zeta^{(\infty)}(s)\| \right)\\
	\|\tilde{\xi}^{(\infty)}(s)\| &\le 2\left(\varrho_0+\varrho_\infty\right)\left(\|\zeta^{(0)}(s)\| +\|\zeta^{(\infty)}(s)\| \right).
	\end{align*}
	From \eqref{eq:Ta-Tb_norm_divided}, we have
	\begin{align}\label{eq:Ta-Tb_norm}
	\begin{pmatrix}
	\left\|\varPi^{(0)}(T(a_1)-T(a_2))\right\|_X \\[2mm]
	\left\|\left( {\rm Id} - \varPi^{(0)} \right)(T(a_1)-T(a_2))\right\|_X
	\end{pmatrix}
	\le \bm{M_\varrho}
	\begin{pmatrix}
	\|\zeta^{(0)}\|_{X}\\[2mm]
	\|\zeta^{(\infty)}\|_{X}
	\end{pmatrix},
	\end{align}
	where $\bm{M_\varrho}$ is a strictly positive $2$ by $2$ matrix defined by
	\[
	\bm{M_\varrho}\bydef2h\bm{U_h}\begin{pmatrix}\varrho_0&\varrho_\infty\\\varrho_0+\varrho_\infty&\varrho_0+\varrho_\infty
	\end{pmatrix}.
	\]
	Using \eqref{eq:def_distance} and \eqref{eq:Ta-Tb_norm}, the distance between $T(a_1)$ and $T(a_2)$ is bounded by
	\begin{align}\label{eq:contraction_divided}
	\mathbf{d}\left(T(a_1),T(a_2)\right)\le\|\mathbf{M}_\varrho\|\mathbf{d}(a_1,a_2),
	\end{align}
	where
	$\left\|\bm{M_\varrho}\right\|\bydef\max_{\|x\|_{w}=1}\|\bm{M_\varrho} x\|_{w}=\|\bm{M_\varrho}w\|_{w}$ with $w=(\varrho_0,\varrho_\infty)^T$, which is the matrix norm arising from the scaling maximum norm defined by \eqref{eq:def_sm_norm}. On the other hand, from \eqref{eq:onto_divided}, we remark that $f_{\varepsilon}^{(i)}$ ($i=0,\infty$) are represented by
	\begin{align*}
	\begin{pmatrix}
	f_{\varepsilon}^{(0)}\left(\varrho_0,\varrho_\infty\right)\\[2mm]
	f_{\varepsilon}^{(\infty)}\left(\varrho_0,\varrho_\infty\right)
	\end{pmatrix}=\bm{M_\varrho}\begin{pmatrix}
	\varrho_0\\\varrho_\infty
	\end{pmatrix}+\gamma,\quad
	\gamma\bydef\bm{U_h}\begin{pmatrix}
	\varepsilon_0+h\delta_0\\
	\varepsilon_\infty+h\delta_\infty
	\end{pmatrix}>0.
	\end{align*}
	It follows from the assumption of this theorem
	\[
	\bm{M_\varrho}\begin{pmatrix}
	\varrho_0\\\varrho_\infty
	\end{pmatrix}<\bm{M_\varrho}\begin{pmatrix}
	\varrho_0\\\varrho_\infty
	\end{pmatrix}+\gamma\le\begin{pmatrix}
	\varrho_0\\\varrho_\infty
	\end{pmatrix}.
	\]
	Hence, $\|\bm{M_\varrho}\|=\|\bm{M_\varrho} w\|_{w}<\|w\|_w=1$ holds.
	Finally, from \eqref{eq:contraction_divided}, the operator $T$ defined in \eqref{eq:simp_Newton_op} becomes the contraction mapping on $B_{J}^{(0)}(\ba,\varrho_0)\times B_{J}^{(\infty)}(\ba,\varrho_\infty)$.
\end{proof}

\subsection{\boldmath$W_0$ bounds via interval inclusion of range of \boldmath$\bar{U}^{(0)}(t,s)$\unboldmath}\label{sec:W0_bounds}

The remaining tasks of this section consist of obtaining the bound $W_0$ satisfying \eqref{eq:bound_W_0}, which is associated with $\bU^{(0)}(t,s)$ defined in \eqref{eq:bU_m_definition}. From the definition, we have
\begin{align}\label{eq:supbU0}
\sup_{(t,s)\in \mathcal{S}_h}\left\| \bU^{(0)}(t,s)\right\|_{B(\ell^1_{\nu,1})}=\sup_{(t,s)\in \mathcal{S}_h}\left| C^{(0)}(t,s)\right|=\sup_{(t,s)\in \mathcal{S}_h}\left| \Phi(t)\Psi(s)\right|,
\end{align}
where $C^{(0)}(t,s)$ is a fundamental solution of \eqref{eq:linearizedeq_finite} and $\Phi(t)\in \C$ is the principal fundamental solution, which solves
\begin{align}\label{eq:Phi(t)}
\frac{d}{dt}\Phi(t)-2i\ba_0(t)\Phi(t)=0,\quad \Phi(0)=1,\quad t\in J=[0,h].
\end{align}
In addition, $\Psi(s)=\Phi(s)^{-1}$ is the solution of the adjoint problem of \eqref{eq:Phi(t)}
\[
\frac{d}{ds}{\Psi}(s)+2i\Psi(s)\ba_0(s)=0,\quad \Phi(0)=1,\quad s\in J.
\]
\begin{definition}
	In the rest of this section, we define the approximate solution $\ba(t)=(\ba_{k})_{|k|\le K}$ as
	\[
		\ba_{k}(t) \bydef \sum_{n=0}^{N-1}\ba_{n,k}T_n(t),
	\]
	where the Chebyshev polynomials $T_n: J \to \R$ $(n \ge 0)$ are orthogonal polynomials defined by $T_0(t)\bydef 1$, $T_1(t)\bydef\xi(t)$ and $T_{n+1}(t)\bydef2\xi(t) T_n(t) - T_{n-1}(t)$ for $n \ge 1$ with the rescaling $\xi(t) \bydef 2t/h -1$.
\end{definition}

Using the approach of \cite{MR3148084} and \cite[Section 3.2]{takayasu2019rigorous}, we rigorously compute the Chebyshev series expansion of the fundamental solutions $\Phi(t)$ and $\Psi(s)$, which are denoted by
\[
\Phi(t)=\sum_{n=0}^\infty c^{(\phi)}_n T_n(t),\quad \Psi(s)=\sum_{n=0}^\infty c^{(\psi)}_n T_n(s).
\]
The fundamental solution is represented by
\begin{align}\label{eq:PhiPsi}
C^{(0)}(t,s)=\Phi(t)\Psi(s)=\left(\sum_{n = 0}^\infty c^{(\phi)}_nT_n(t)\right)\left(\sum_{n = 0}^\infty c^{(\psi)}_nT_n(s)\right)
\end{align}

To obtain the bound $W_0$, we take  supremum of $|\Phi(t)\Psi(s)|$ over the simplex $(t,s)\in\cS_h$ using interval arithmetic. Firstly,  let $N\in\N$ be a truncate number of Chebyshev polynomials to approximate $\Phi(t)$ and $\Psi(s)$. The approach of \cite{MR3148084} provides the truncated coefficients $\bar{c}^{(\phi)} \bydef (\bar{c}^{(\phi)}_n)_{n< N}$ and $\bar{c}^{(\psi)} \bydef (\bar{c}^{(\psi)}_n)_{n< N}$ with rigorous error bounds $r_\Phi\ge 0$ (resp. $r_\Psi\ge 0$) such that
\[
\left|c^{(\phi)}-\bar{c}^{(\phi)}\right|_\nu \le r_\Phi\quad\left(\mbox{resp.}\quad	\left|c^{(\psi)}-\bar{c}^{(\psi)}\right|_\nu \le r_\Psi\right),
\]
where $c^{(\phi)}\bydef(c^{(\phi)}_n)_{n\ge 0}$ and $c^{(\psi)}\bydef(c^{(\psi)}_n)_{n\ge 0}$ are the Chebyshev coefficients of $\Phi(t)$ and $\Psi(t)$ respectively. Moreover, $|\cdot|_\nu$ is defined in \eqref{eq:ell_nu_one}.
From the property of Chebyshev polynomials $\sup_{t\in J}|T_n(t)|\le 1$, it follows for a fixed $t\in J$ that
\begin{align}
\Phi(t) &= \sum_{n = 0}^\infty c^{(\phi)}_nT_n(t)\\
&=  \sum_{n = 0}^{N-1} \bar{c}^{(\phi)}_nT_n(t) +  \sum_{n = 0}^{N-1} (c^{(\phi)}_n-\bar{c}^{(\phi)}_n)T_n(t) +  \sum_{n \ge N} c^{(\phi)}_nT_n(t)\\
&\in \sum_{n = 0}^{N-1} \bar{c}^{(\phi)}_nT_n(t) +  \left|c^{(\phi)}-\bar{c}^{(\phi)}\right|_\nu\cdot\langle 0,1 \rangle\\
&\in \sum_{n = 0}^{N-1} \bar{c}^{(\phi)}_nT_n(t) + \langle 0,r_{\Phi}\rangle,\label{eq:Phi_interval}
\end{align}
where $\langle c,r\rangle$ denote a complex interval centered at $c\in\C$ with the radius $r>0$.
Similarly, rigorous inclusion of $\Psi(s)$ using interval arithmetic is given by
\begin{align}\label{eq:Psi_interval}
\Psi(s) &\in \sum_{n = 0}^{N-1} \bar{c}^{(\psi)}_nT_n(s) + \langle 0,r_{\Psi}\rangle,\quad s\in J.
\end{align}

Secondly, for $P\in\N$, we divide the time step $J=[0,h]$ into $NP+1$ intervals\footnote{In the actual implementation, we fix $P=64$. It is possible to take more partitions, but in such a case the computation time will be slower.} by the formula $\bm{I}_j\bydef[t_j,t_{j+1}]$ ($j = 0,1,\dots,NP$), where $t_j\bydef(1+\cos\theta_j)h/2$ and $\theta_j = \pi(1-j/(NP+1))$. Let $T_n(\bm{I}_j)$ be interval extension such that
\begin{align}
T_n(\bm{I}_j)\supseteq \left\{T_n(t)=\cos\left(n\arccos\left(\xi(t)\right)\right):t\in\bm{I}_j\right\},\quad j=0,1,\dots,NP.
\end{align}

Thirdly, from \eqref{eq:supbU0}, \eqref{eq:PhiPsi}, \eqref{eq:Phi_interval}, and \eqref{eq:Psi_interval} we have the bound $W_0$ using interval arithmetic
\begin{align}
\sup_{(t,s)\in \mathcal{S}_h}\left| \Phi(t)\Psi(s)\right|&\le \sup_{(t,s)\in \mathcal{S}_h}\mathop{\textrm{mag}}\left(\left( \sum_{n = 0}^{N-1} \bar{c}^{(\phi)}_nT_n(t) + \langle 0,r_{\Phi}\rangle\right)\left(\sum_{n = 0}^{N-1} \bar{c}^{(\psi)}_nT_n(s) + \langle 0,r_{\Psi}\rangle\right)\right)\\
&\le \max_{0\le l\le j\le NP}\mathop{\textrm{mag}}\left(\left( \sum_{n = 0}^{N-1} \bar{c}^{(\phi)}_nT_n(\bm{I}_j) + \langle 0,r_{\Phi}\rangle\right)\left(\sum_{n = 0}^{N-1} \bar{c}^{(\psi)}_nT_n(\bm{I}_l) + \langle 0,r_{\Psi}\rangle\right)\right)\bydef W_0,
\end{align}
where $\mathop{\textrm{mag}}(\bm{I})$ denotes the magnitude of an interval $\bm{I}$ defined by $\mathop{\textrm{mag}}(\bm{I})\bydef\max\{|t|:t\in\bm{I}\}>0$.

\subsection{Time stepping scheme over multiple time intervals}
To sum up this section, we show a time stepping scheme of rigorous integrator to extend the local inclusion of solution of \eqref{eq:NLS_Quad} over multiple time intervals. Let $0=t_0<t_1<\dots$ be grid points of the time variable. We call $J_i \bydef [t_{i-1}, t_i]$ the $i$\,th time step. In addition, let us define $t_i \bydef ih_i$ ($i = 1, 2,\dots$) with the stepsize $h_i$ of $J_i$, which is adaptively changed. 

Firstly, we assume that the solution $a(t)$ of \eqref{eq:CGL_ODEs} is rigorously included in $B_{J_1}^{(0)}\left(\ba^{J_1}, \varrho_{0}\right) \times B_{J_1}^{(\infty)}\left(\ba^{J_1}, \varrho_{\infty}\right)$ defined in \eqref{eq:B0} and \eqref{eq:Binfty}. Secondly, to consider the next time step $J_2$, we set the time step $h_2$ and the approximate solution $\ba^{J_2}$ to satisfy the sufficient condition \eqref{eq:kappa_condition} of Theorem \ref{thm:sol_map} and obtain the matrix $\bm{U_h}$ defined by \eqref{eq:definition_of_U_h}.
Thirdly, the initial sequence is updated by a sequence at the endpoint of $J_1$, i.e., $\varphi = a(t_1)$. Replacing $J=J_2$, we apply Theorem \ref{thm:modified_thm} for the initial-boundary value problem on $J_2$. The error estimate of the zero-mode between the initial sequence and its approximate solution is bounded by 
\begin{align}
|\varphi_0-\ba^{J_2}_0(t_1)|\le |\varphi_0-\ba^{J_1}_0(t_1)|+|\ba^{J_1}_0(t_1)-\ba^{J_2}_0(t_1)|
\le \varrho_{0} +|\ba^{J_1}_0(t_1)-\ba^{J_2}_0(t_1)| = \varepsilon_0.
\end{align}
We note that $|\ba^{J_1}_0(t_1)-\ba^{J_2}_0(t_1)|$ is a tiny numerical error. For example, if we adopt Chebyshev polynomials to approximate the time variables, such a error becomes almost $10^{-14}$ but not zero. Similarly, the error estimate of the other modes is given by
\begin{align}
\sum_{k\neq 0}|\varphi_k-\ba^{J_2}_k(t_1)|\nu^{|k|}&\le \sum_{k\neq 0}|\varphi_k-\ba^{J_1}_k(t_1)|\nu^{|k|}+\sum_{k\neq 0}|\ba^{J_1}_k(t_1)-\ba^{J_2}_k(t_1)|\nu^{|k|}\\
&\le \varrho_{\infty} +\sum_{k\neq 0}|\ba^{J_1}_k(t_1)-\ba^{J_2}_k(t_1)|\nu^{|k|} = \varepsilon_\infty.
\end{align}
Fourthly, we validate the sufficient condition of Theorem \ref{thm:modified_thm} and obtain the next $\varrho_{0}$ and $\varrho_{\infty}$ of rigorous inclusion $B_{J_2}^{(0)}\left(\ba^{J_2}, \varrho_{0}\right) \times B_{J_2}^{(\infty)}\left(\ba^{J_2}, \varrho_{\infty}\right)$. Finally, continuing to the second part, we recursively repeat this process several times. Using this time stepping scheme we can extend the local inclusion of solution until either the sufficient condition of Theorem \ref{thm:local_inclusion} or that of Theorem \ref{thm:modified_thm} is no longer satisfied.

\begin{remark}
	When we numerically compute an approximate solution using Chebyshev polynomials for the time variables, the defect bound $\left\|F^{(j)}(\ba)\right\|_X\le\delta_j$ $(j=0,\infty)$ is given by the same way introduced in our previous paper. Let us refer to \cite[Section 4.2]{takayasu2019rigorous} for computing the bound $\delta_j$.
\end{remark}

\section{Proof of Theorem \ref{thm:Heteroclinics}}
\label{sec:FinalTheorem}

As established by Theorem \ref{thm:Equilibria}, there exist two spectrally unstable equilibria  $u_1^i$ and $u_1^{ii}$ to \eqref{eq:NLS_Quad} which generate infinite families of spectrally  unstable equilibria. 
Let $\tilde{u}$ denote either of the equilibria  $u_1^i$ or $u_1^{ii}$, and let $ \tilde{u}^*$ denote its complex conjugate. 
To prove Theorem \ref{thm:Heteroclinics}, we construct two heteroclinic orbits $ u_{a}(t)$  and $ u_{a^*}(t)$  such that 
\begin{align*}
\lim_{t \to -\infty} u_{a}(t) &= \tilde{u} &
\lim_{t \to +\infty} u_{a}(t) &= 0
\\ 
\lim_{t \to -\infty} u_{a^*}(t) &= \tilde{u}^*
&
\lim_{t \to +\infty} u_{a^*}(t) &=  0.
\end{align*}
One may check that 	if  $ u(t,x)$ is a solution to $u_t=i(u_{xx} +u^2)$, 
then $w(t,x)\bydef u(-t,x)^*$ is also a solution. 
Hence, $u_{b}(t) \bydef (u_{a^*}(-t))^*$  and $u_{b^*}(t) \bydef (u_{a}(-t))^*$ are connecting orbits satisfying 
\begin{align*}
\lim_{t \to -\infty} u_{b}(t) &= 0 &
\lim_{t \to +\infty} u_{b}(t) &= \tilde{u}
\\ 
\lim_{t \to -\infty} u_{b^*}(t) &= 0
&
\lim_{t \to +\infty} u_{b^*}(t) &=  \tilde{u}^*.
\end{align*}
By the rescaling \eqref{eq:RescalingSolutions}, there exist similar heteroclinic solutions if $\tilde{u}_n$ was some rescaling of $ \tilde{u}$.  

Thus, it suffices to 
prove  the existence of heteroclinic solutions $ u_{a}(t)$  and $ u_{a^*}(t)$, whose difference is the starting equilibrium $\tilde{u}$ or $ \tilde{u}^*$. 
We establish the existence of each heteroclinic constructively by a computer-assisted proof consisting of the following three steps:
\begin{description}
	\item[Step 1.] Set an equilibrium $\tilde{u}$ of \eqref{eq:NLS_Quad}. Validating the Fourier coefficients of the equilibrium $\ta$ satisfying \eqref{eq:f=0_steady_state} and an eigenpair $(\tilde{\lambda},\tb)$ satisfying \eqref{eq:g=0_eigenpair} with respect to the unstable direction attached to the equilibrium (this is done using the approach presented in Appendix~\ref{sec:eigenpairs}), we rigorously construct a part of unstable manifold $P(\sigma)$ using the Parameterization Method introduced in Section \ref{sec:UnstableManifold}.
	\item[Step 2.] From the endpoint of the unstable manifold denoted by $P(1)$ or $P(-1)$, our rigorous integrator provided in Section \ref{sec:Integrator} propagates the rigorous inclusion forward in time by using the time stepping scheme. At the end of each time step, we check whether the solution of \eqref{eq:NLS_Quad} is in the stable region by checking the hypothesis of Theorem \ref{thm:HomoclinicBlowup}.
	\item[Step 3.]  If the hypothesis of Theorem \ref{thm:HomoclinicBlowup} is obtained at a certain time (after several time stepping processes), then whole orbit connects from the equilibrium $\ta$ to the zero function, which completes the proof.
\end{description}
We refer to a conceptual picture (Figure \ref{fig:proof}) to summarize the entire proof.
\begin{figure}[ht]
	\centering
	\includegraphics[width = .9 \textwidth]{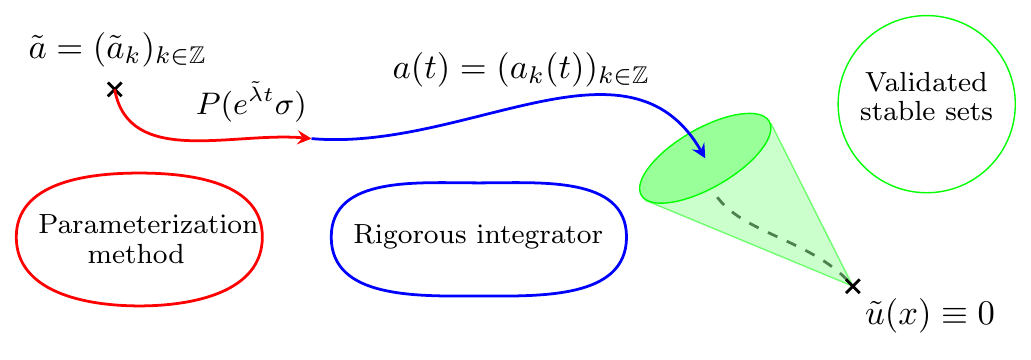}
	\caption{The entire proof consists of three steps: Parameterization method in Section \ref{sec:UnstableManifold}, rigorous integrator in Section \ref{sec:Integrator}, and validation of global existence in Section \ref{sec:CenterManifold}.}\label{fig:proof}
\end{figure}


All computations are carried out on Microsoft Windows 10 Pro, Intel(R) Core(TM) i9-10900K CPU@3.70 GHz, and MATLAB 2020b with INTLAB - INTerval LABoratory \cite{Ru99a} version 11 and Chebfun - numerical computing with functions \cite{MR2767023} version 5.7.0. All codes used to produce the proof in this section are freely available from \cite{bib:codes}.

%
%


We show the proof of a heteroclinic connection from $u^{i}_1(x)$ given by Theorem \ref{thm:Equilibria} to the zero equilibrium, whose time evolution is displayed in Figure \ref{fig:CO_from_1}.

\begin{figure}[htbp]
	\centering
	\includegraphics[width = .92 \textwidth]{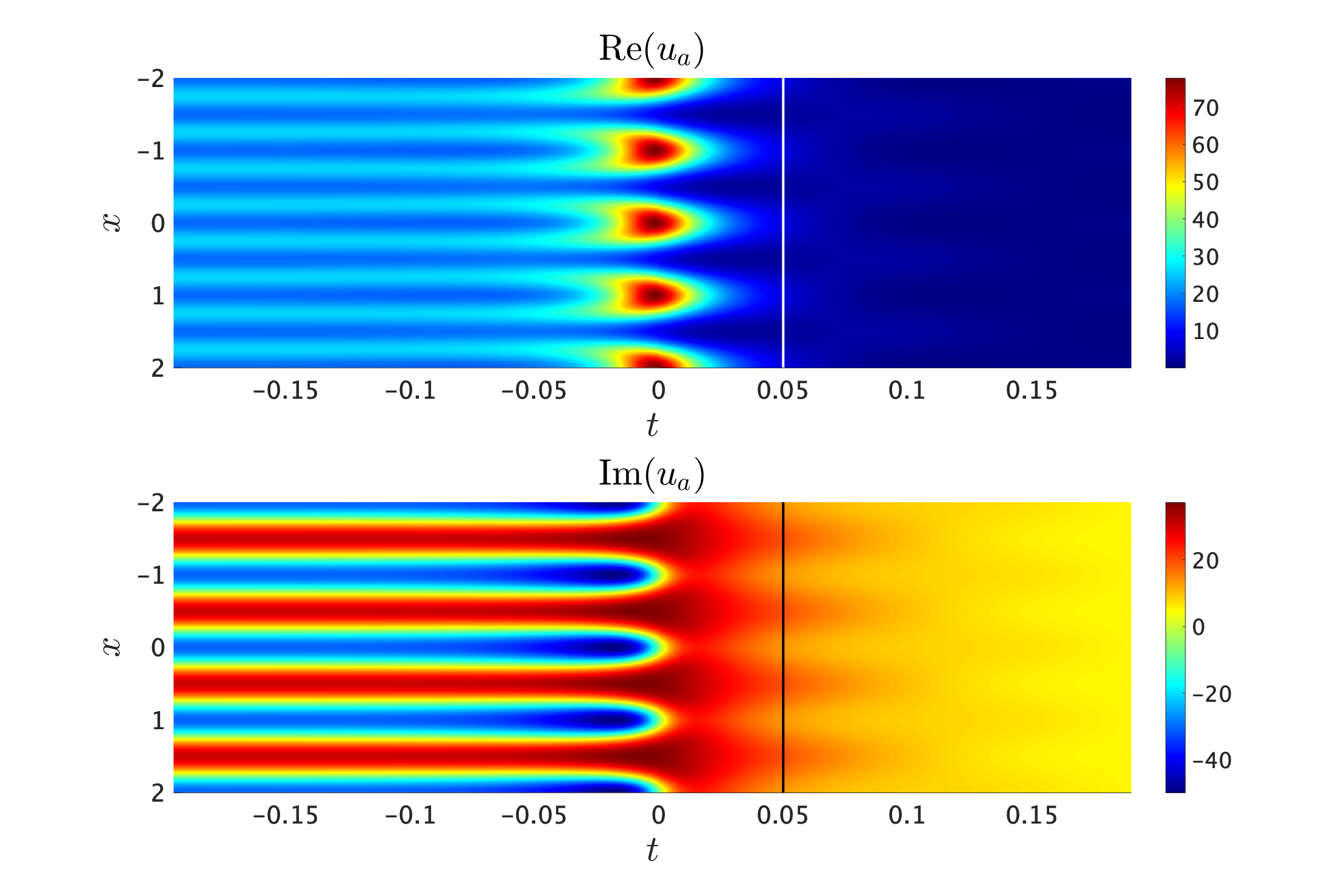}\\[-4pt]
	($a$) The heteroclinic solution $u_a$: connection from $u^i_1(x)$ to $0$.
	\includegraphics[width = .92 \textwidth]{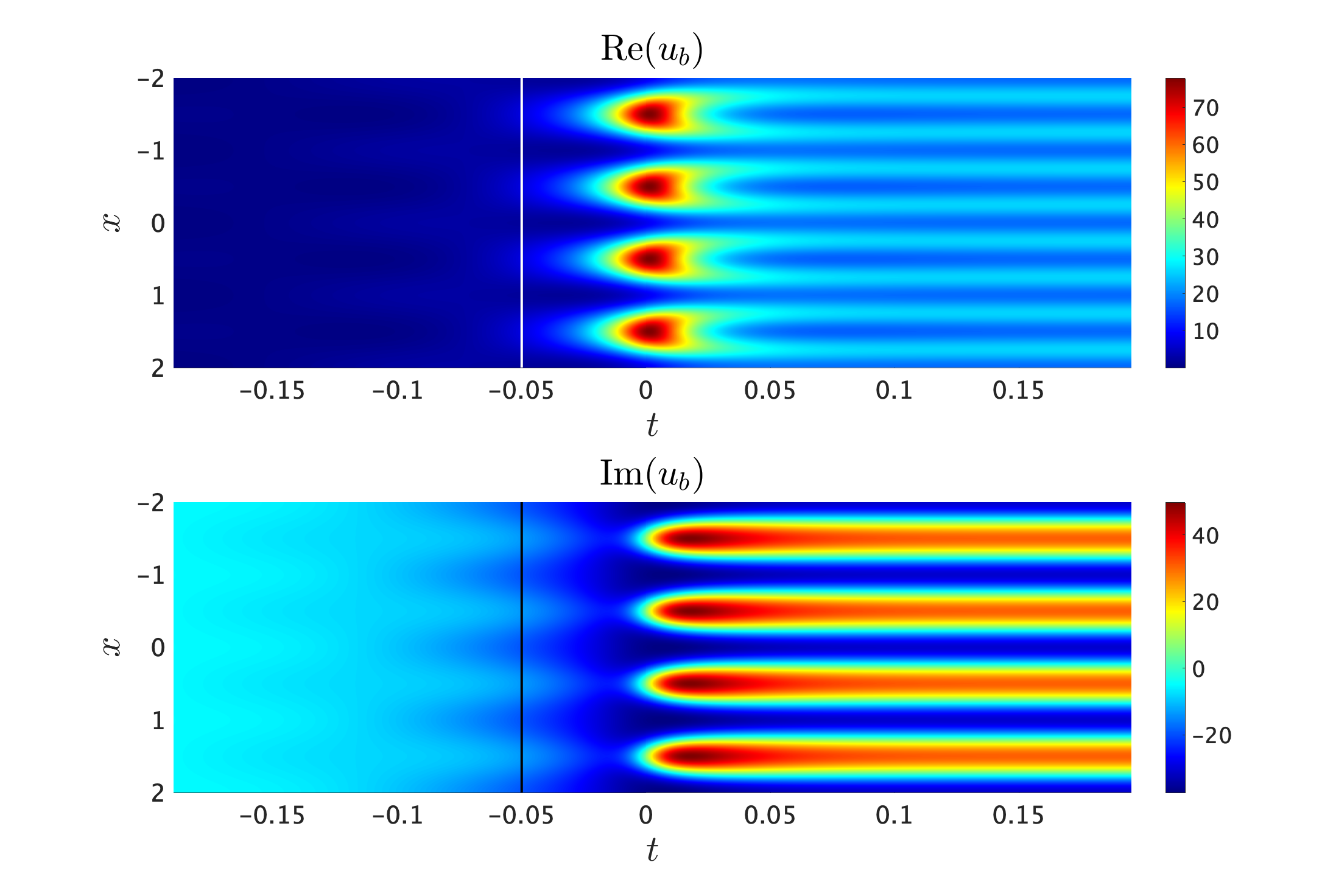}\\[-4pt]
	($b$) The heteroclinic solution $u_b$: connection from 0 to $u^i_1(x)$.
	\caption{Heteroclinic orbits of \eqref{eq:NLS_Quad} between $u^i_1(x)$ and $0$. ($a$)~We parameterized the unstable manifold attached to $u^i_1(x)$ for negative $t$ via rigorous numerics. From $t=0$ to the white/black line $t_{\mathrm{end}}=0.05$, rigorous integrator propagates the whole orbit including $P(1)$ to the stable region. At this line the proof is completed. ($b$) We plot the result of the computer-assisted proof of connection from $(u^i_1(x))^*$ to $0$ by the form $(u_{a^*}(-t))^*$, which corresponds to the connection from 0 to $u^i_1(x)$.}\label{fig:CO_from_1}
\end{figure}

\paragraph{Step 1.} Let us consider the equilibrium $u^{i}_1(x)$ (see Figure \ref{fig:equilibria}) of \eqref{eq:NLS_Quad} using Fourier series. The approach for the zero finding eigenvalue problem presented in Section \ref{sec:eigenpairs} provides the rigorous steady state $\ta$ and the eigenpair $(\tilde{\lambda},\tb)$ such that
\[
	|\ta -\ba|_\nu\le r_0,\quad |\tilde{\lambda}-\bar{\lambda}|\le r_0,\quad |\tb -\bb|_\nu\le r_0
\]
with $r_0=1.61\times 10^{-12}$, where $|\cdot|_\nu$ is defined in \eqref{eq:ell_nu_one} ($\nu=1$ in the rest of proofs). 
The execution time was about 0.15 sec. From these data we parameterize a subset of the unstable manifold of $\ta$ introduced in Section \ref{sec:UnstableManifold}. The parameterization method yields a solution $\tp=(\tp_{k,m})_{k,m\ge 0}\in X^{\nu}$ of $f(\tp)=0$ defined in \eqref{eq:f=0_manifold}, which denotes Taylor-Fourier coefficients of the solution $P(\sigma)$ in \eqref{eq:P_power_series}. The result of rigorous error bound is
\[
	\|\tp-\bp\|_{\nu} \le r_p,\quad r_p=5.91\times 10^{-10},
\]
where  $\|\cdot\|_\nu$ is defined in \eqref{eq:X_nu}.
From Lemma \ref{lem:parameterization} we have the end point of the unstable manifold satisfying $a(0) = P(1)$.
It follows that, at the end point of the unstable manifold, we have a rigorous inclusion such that
\[
	\|a(0)-\bar{\varphi}\|\le r_p,\quad\bar{\varphi}=\left(\sum_{m=0}^{M}\bp_{|k|,m}\right)_{|k|\le K}
\]
where $\|\cdot\|$ is given in Definition \ref{def:ell_nu^1} and $(K,M)=(27,150)$. In total the execution time for the parameterization method was about  89.4 sec.

\paragraph{Step 2.} 
Next, rigorous integrator provided in Section \ref{sec:Integrator} starts from the rigorous inclusion of the endpoint of unstable manifold.  Set the initial sequence $\varphi=P(1)$ and the step size $h = 2.5\times 10^{-3}$ equidistantly. We have the following initial error estimate, which is used in Theorem \ref{thm:modified_thm}:
\begin{align}\label{eq:varepsilon_0}
	|\varphi_0-\ba_0(0)|\le |\varphi_0-\bar{\varphi}_0|+|\bar{\varphi}_0-\ba_0(0)| \le r_p +|\bar{\varphi}_0-\ba_0(0)| = \varepsilon_0
\end{align}
\begin{align}\label{eq:varepsilon_inf}
	\|\varphi^{(\infty)}-\ba^{(\infty)}(0)\|\le \|\varphi^{(\infty)}-\bar{\varphi}^{(\infty)}\|+\|\bar{\varphi}^{(\infty)}-\ba^{(\infty)}(0)\| \le r_p+\|\bar{\varphi}^{(\infty)}-\ba^{(\infty)}(0)\|  =\varepsilon_\infty.
\end{align}
After 20 time stepping (at $t_{\mathrm{end}}=0.05$) our integrator yields the rigorous inclusion of solution trajectory with $\varrho_0 = 1.43\times 10^{-8}$ and $\varrho_\infty = 1.97\times 10^{-7}$. Then the hypothesis of Theorem \ref{thm:HomoclinicBlowup} holds. The number of Chebyshev polynomials for the time variable is $13$ and the computational time for rigorous integration was almost 19.5 sec.
\paragraph{Step 3.}
While Theorem \ref{prop:PowerGlobalExistence} may suffice to verify the solution converges to zero, we instead directly check the hypothesis of Theorem \ref{thm:HomoclinicBlowup} which offers a sharper result. Let us define $ z_0 \in \C$, and $\phi, \bar{\phi} \in \ell_{\nu}^1$ by    
\begin{align}
z_0 &\bydef 
\ba_0(t_{\mathrm{end}}) ,
&
\bar{\phi}  &\bydef  \bar{a}^{(\infty)}(t_{\mathrm{end}}),
&
(	\phi)_k &\bydef 
\begin{cases}
a_0(t_{\mathrm{end}}) - z_0  & \mbox{if } k =0 \\
a_k(t_{\mathrm{end}})  & \mbox{else }.
\end{cases}
\end{align}
Recall that $ \iota_0: \C \hookrightarrow  \ell_\nu^1$ denotes the inclusion into the 0\textsuperscript{th} Fourier mode.
Hence $a(t_{\mathrm{end}})  = \iota_0 \circ z_0 + \phi$ 
  and we note that $\| \phi\| \leq \| \bar{\phi}\| + \varrho_{0} + \varrho_{\infty}$. 
 We may then define $ \rho_0 \bydef | z_0| $ and $\rho_1 \bydef (\| \bar{\phi}\| + \varrho_{0} + \varrho_{\infty}) / \rho_0^2$. 
 Hence $ \| \phi \| \leq \rho_1 | z_0|^2$, and if $Im(z_0) \geq 0$ then $ \phi + \iota_0 \circ z_0\in \cB(\rho_0,\rho_1)$. 
By Theorem \ref{thm:HomoclinicBlowup} and Remark \ref{rem:StableSetP2}, if there exists some $ r >0$ such that 
 $	\rho_1  \exp \left\{  \tfrac{\pi}{2} \rho_0 r \right\} < r$  then $ \lim_{t\to +\infty} a(t) = 0$.  

For this example, this is achieved with values  $\rho_0=17.14$,  $\rho_1=0.014$ and  $r=0.036$.  
The execution time of Step 3 is less than 0.1 sec. Finally the proof for the heteroclinic solution from $u^{i}_1(x)$ to $ u \equiv 0$ is completed. \qed
\newline

In a similar manner, we construct a computer-assisted proof of a heteroclinic orbit from $ (u_1^i(x))^*$ to zero. 
We are also able to prove the existence of other heteroclinic orbits from $u_i^1$ to zero by starting with other points on its unstable manifold. 
In Figure \ref{fig:CO_from_minus_1} is displayed another connecting orbit with the end point of unstable manifold by the form $P(-1)$ in \eqref{eq:P_power_series}. 
This proof is achieved in the same manner, albeit with different  computational parameters. The most notable difference with this second computer-assisted proof being the number of time-steps the rigorous integrator had to take -- 20 time-steps for the former and 2523 for the latter --  before we could verify the solution was in the stable set of the zero-equilibrium. 
Indeed, as one can see from Figure \ref{fig:CO_from_minus_1}, a longer time is required for the non-zero Fourier modes to decay in comparison to Figure \ref{fig:CO_from_1}.

As the final proof, we show the existence of another connecting orbit, which is from $u^{ii}_1(x)$ to the zero equilibrium.
We display the time evolution of the solution in Figure \ref{fig:CO_pt2_from_1}\,($a$). We note that the amplitude of this equilibrium is 10 times larger than $u^{i}_1(x)$, and the connection from such a large equilibrium to zero function is quite nontrivial. 
The proof of this connecting orbit is much difficult than previous two proofs as we had to contend with significantly larger validated error bounds. 
In a similar fashion we are able to produce a computer-assisted proof of a heteroclinic orbit from $(u_2^i(x))^*$ to the zero solution in Figure \ref{fig:CO_pt2_from_1}\,($b$). 
 A summary of the major computational parameters used in our computer-assisted proofs are given in Appendix \ref{sec:ComputationalTable}.

\begin{figure}[htbp]
	\centering
	\includegraphics[width = .92 \textwidth]{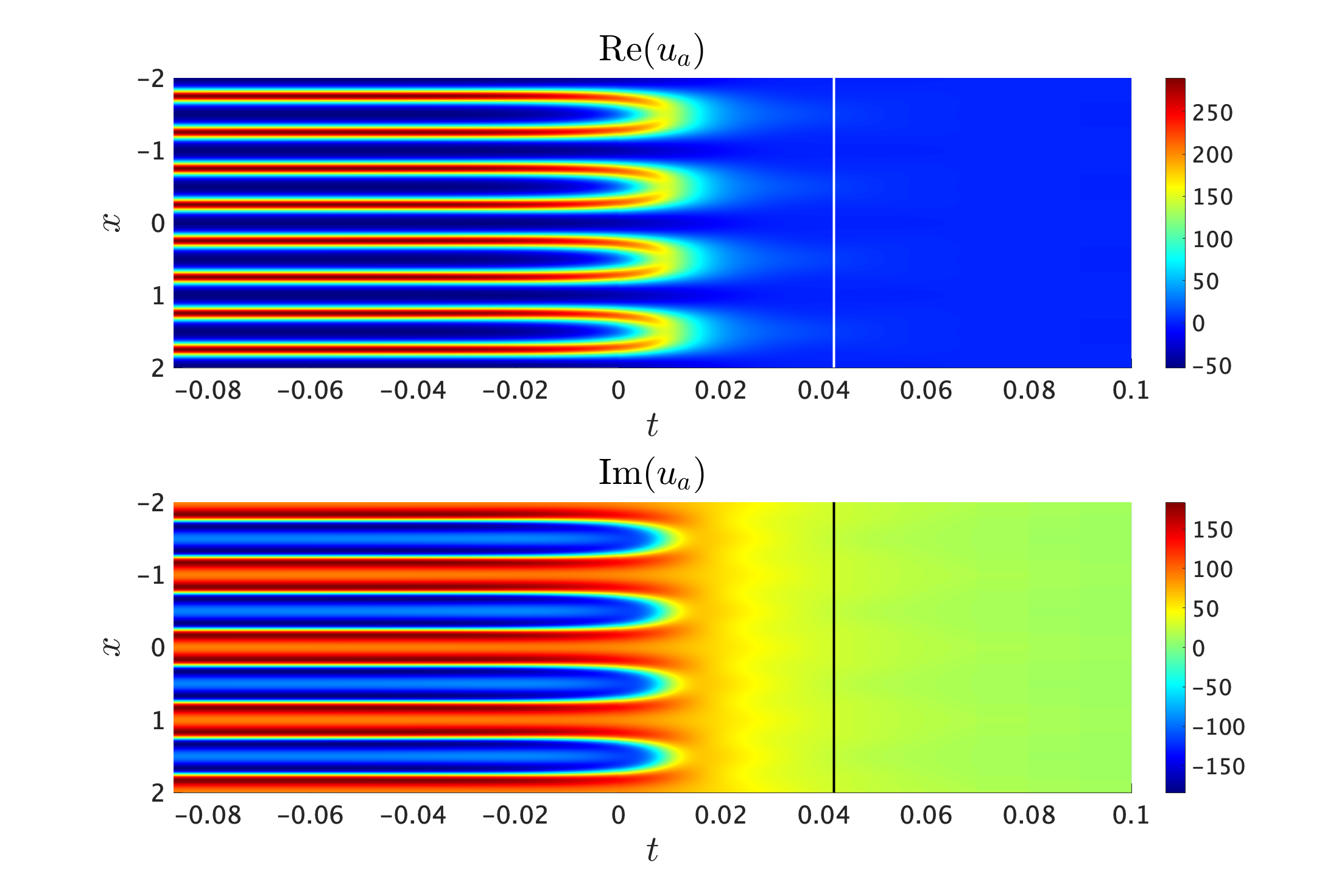}\\[-4pt]
	($a$) The heteroclinic solution $u_a$: connection from $u^i_1(x)$ to $0$.
	\includegraphics[width = .92 \textwidth]{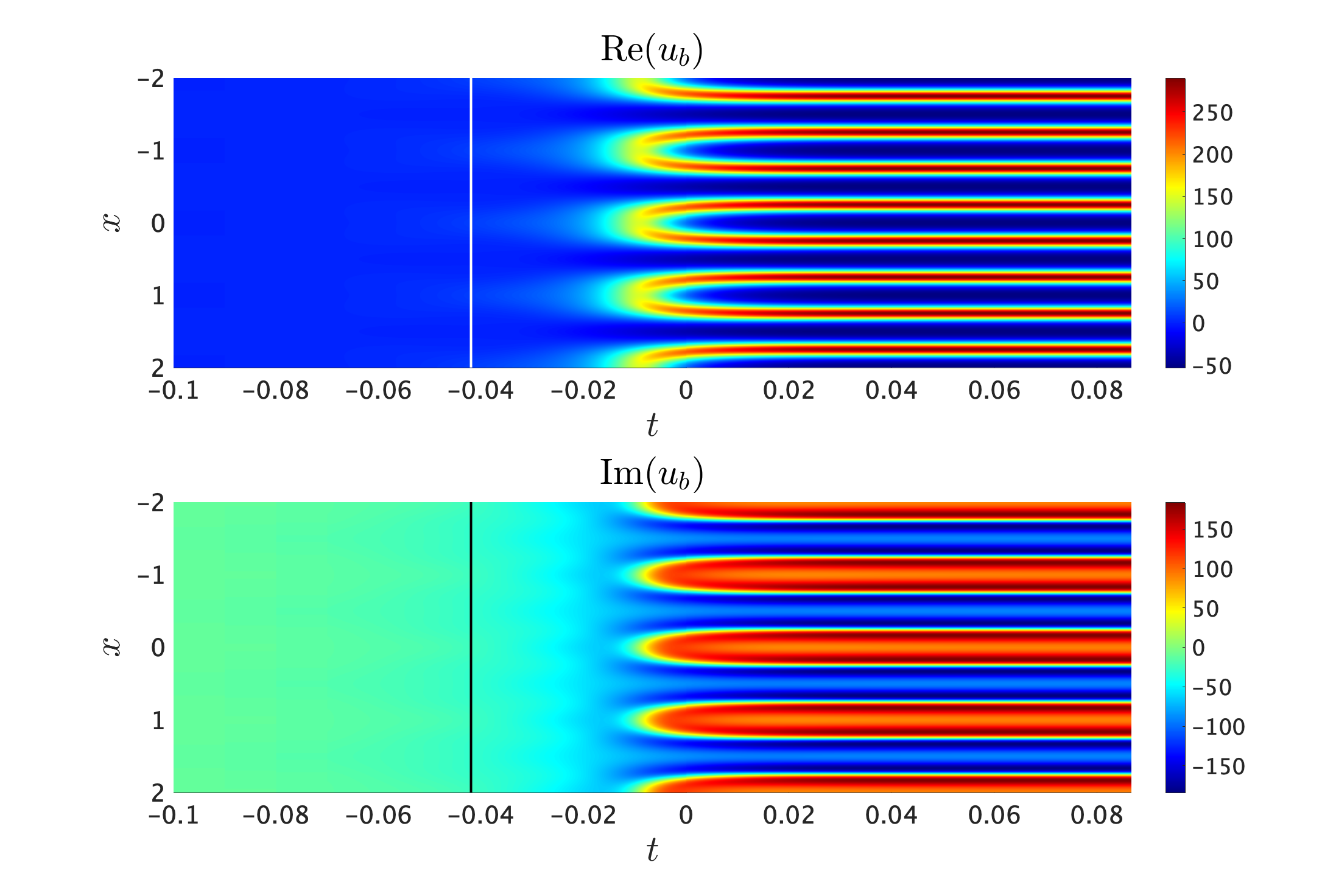}\\[-4pt]
	($b$) The heteroclinic solution $u_b$: connection from 0 to $u^i_1(x)$.
	\caption{Heteroclinic orbits of \eqref{eq:NLS_Quad} between $u^{ii}_1(x)$ to $0$. Proofs of these orbits are achieved in the same manner. Note that the amplitude of this orbit is much bigger than the previous ones. Such connection is quite nontrivial to prove. }\label{fig:CO_pt2_from_1}
\end{figure}

%
%
%
%
%



\section{Open Questions} 
\label{sec:OpenQuestions}

Through rigorous mathematical analysis, we have established the existence of rich dynamical behavior in nonconservative nonlinear Schr\"odinger equations. 
However all the phenomena we suspect to occur exceeds that which we can prove rigorously. 
Still there are nascent patterns we have noticed and merit further investigation, yet we have not the space in a single paper to fully expound.  
Below we have compiled a list of open questions concerning the equation \eqref{eq:NLS_Quad}. 
Similar questions, of course, may be asked for other equations of the type \eqref{eq:NLS}.

\begin{enumerate}
	\item As was conjectured in \cite{Cho2016}, do solutions to \eqref{eq:NLS_Quad} with  real initial data exist for all time?  We note that Theorem \ref{prop:PowerGlobalExistence} treats the case of arbitrarily large, yet linearly close to constant, real initial data. 
	
	\item In our numerical simulations, as with those reported in \cite{Cho2016}, it appears that real initial data limits in forward and backward time to the zero equilibrium. 
	Could this be proven? 
	Alternatively, there could exist real initial data which limits to a non-trivial equilibrium or some other invariant set. 	Could the existence of such a solution be proven?

		\item A differential equation is generally considered to be integrable if it has a maximal set of conserved quantities, and if one may write down solutions explicitly, for example by quadrature or the inverse scattering method. 
	In Theorem \ref{thm:OpenHomoclinics} we showed that  the equation \eqref{eq:NLS} has an open set of homoclinic orbits limiting to zero, and by Theorem \ref{thm:NoConservedQuantities} this forces the non-existence of any real analytic conserved quantities. 
	This property is also shared by the dynamics $\dot{z} = i z^2$ in the spatially homogeneous subsystem of \eqref{eq:NLS_Quad}, see Figure \ref{fig:HomogeneousDynamics}. 
	Nevertheless, as one may see from  \eqref{eq:PpowerSolution}, we are able to give explicit solutions to this ordinary differential equation, and moreover it enjoys a singular conserved quantity  $ V(z) = \frac{1}{z} + \frac{1}{z^*} $.  
	The reader is referred  to \cite{braddell2019invitation} for related work on singular symplectic structures. 
	
	From our attempts, we have been unable to explicitly construct any analogous type of singular conserved quantity preserved by the PDE \eqref{eq:NLS_Quad}. 
	Do there exist any such singular or meromorphic conserved quantities? 
	
	\item The manner in which we found the equilibria in Theorem \ref{thm:Equilibria} was essentially by randomly choosing some sequence of Fourier coefficients, and then applying Newton's method. 
	Is there a systematic classification of all the equilibria to \eqref{eq:NLS_Quad}?

	\item Numerical evidence suggests that the linearization about the non-trivial equilibria have an infinite number of eigenvalues on the negative imaginary axis, and only finitely many eigenvalues have non-zero real part.  Can this be proven? 
	
	\item Numerical evidence suggests that under the rescaling $\tilde{u} \mapsto n^2 \tilde{u}(nx)$, the linearization about the equilibria has an increasing number of  eigenvalues with non-zero real part, and an increasing number of eigenvalues have a positive imaginary part. Can this be explained and proven? 
	
	\item Do there exist periodic solutions to \eqref{eq:NLS_Quad}? Current work is underway by the first author studying periodic solutions near the zero equilibria. The existence of (presumably) imaginary eigenvalues to the linearization about the nontrivial equilibria suggest there exists periodic solutions nearby.

	\item If the linearization about the nontrivial equilibria has two imaginary eigenvalues which are  mutually irrational, that would suggest the existence of invariant tori. Can one prove this and/or the existence of quasi-periodic orbits?

	\item Suppose there exist periodic or quasi-periodic orbits nearby the non-trivial equilibria. Can the heteroclinic orbits from Theorem \ref{thm:Heteroclinics} be perturbed to construct  heteroclinic orbits between the zero equilibria and these (quasi)periodic orbits?

	\item Theorem \ref{thm:Heteroclinics} establishes the existence of a heteroclinic loop between the zero equilibria  and each of the non-trivial equilibria in Theorem \ref{thm:Equilibria}. 
	Do homoclinic orbits from the zero equilibrium exist as a small perturbation away from this heteroclinic loop? Moreover, if there exist periodic orbits or quasi-periodic orbits near the non-trivial equilibria, can one construct a heteroclinic loop between these orbits and the zero equilibria? Can these heteroclinic loops be perturbed to construct a homoclinic orbit? 

	\item The equilibria $u^i$ has  a one complex dimensional strong unstable manifold $P( \sigma)$, and we have proven for some specific values of $ \sigma_0 = r_0 e^{i \theta_0}$ that the trajectory $ P(\sigma_0)$ limits to the zero equilibrium in forward time. 
	Numerical evidence suggests that there is a unique angle $\theta_{B} \in [0,2 \pi]$,  for which the solution with initial data  $ P( r_0 e^{i \theta_{B} })$  will blow up in finite time. This blowup profile is similar  to that in Figure 	\ref{fig:CO_from_1}. 
	For all other angles $ \theta \neq \theta_B\mod 2\pi$, the solution with initial data  $ P( r_0 e^{i \theta })$  appears to converge to $0$. This also seems to be the case at the other equilibrium $u^{ii}$.   Can this be proven?

	\item It could be possible that there exists a connecting orbit between the non-trivial equilibria, or between other invariant sets not including the zero equilibrium. 
	Could the existence or  non-existence of such a trajectory be proven? 
	
	\item It could be possible that there exists an unstable,  chaotic invariant set in this equation.  
	Could the existence or  non-existence of such a set be proven?

\end{enumerate}


\section*{Acknowledgements} 
The authors would like to thank  A. Delshams, E. Miranda, J.D. Mireles James, T. Wanner,  and C.E. Wayne for informative discussions.  
The second author was supported by an NSERC Discovery Grant and an NSERC Accelerator Supplement.
The third author was supported by JSPS KAKENHI Grant Numbers JP18K13453, JP20H01820.

\appendix
\section{Appendix}

\subsection{Proof of Theorem \ref{thm:Equilibria}: The zero finding eigenvalue problem} \label{sec:eigenpairs}

To prove the theorem, it suffices to prove the existence of distinct, fundamental  equilibria, and their associated eigenvalue-eigenvector pair. 
The rest follows by rescaling. 
Looking for steady states to the NLS equation leads to 
\begin{equation}\label{eq:f=0_v0}
i \left( -k^2\omega^2a_k+\left(a^2 \right)_k\right) = 0,
\end{equation}
with $\omega \bydef 2 \pi$. An eigenvalue-eigenvector couple $(\lambda,b)$ associated to a steady-state $a$ satisfies 
\begin{equation}\label{eq:g=0_v0}
i  \left( -k^2\omega^2 b_k + 2 \left(a*b \right)_k\right) - \lambda b_k = 0.
\end{equation}

Note that the factor $i $ can be removed when solving \eqref{eq:f=0_v0} since $i  \ne 0$. Denoting $\mu_k \bydef -k^2\omega^2$, this leads to 
\begin{equation}\label{eq:f=0}
(f_1(a))_k \bydef \mu_k a_k+\left(a^2 \right)_k  = 0.
\end{equation}
and
\begin{equation}\label{eq:g=0}
(f_2(\lambda,a,b))_k \bydef  i  \left(\mu_k b_k + 2 \left(a*b \right)_k \right) - \lambda b_k = 0.
\end{equation}

Now, since eigenvectors come in continuous families (any rescaling also gives an eigenvector), we isolate them by imposing a linear phase condition of the form $\eta(b) = \ell(b)-1=0 \in \C$, which typically fixes a component of $b$.

We impose the conditions $a_{-k}=a_k \in \C$ for the steady states, which correspond to Neumann-0 boundary conditions. This implies that we impose the same condition on the eigenvectors $b$, that is $b_{-k}=b_k \in \C$. This symmetry condition on $a,b$ implies that $(f_1(a))_{-k}=(f_1(a))_{-k}$ and $(f_2(\lambda,a,b))_{-k}=(f_2(\lambda,a,b))_{k}$. Hence, we only need to solve $(f_1)_k=0$ and $(f_2)_k=0$ for $k \ge 0$. Hence, let $a \bydef (a_k)_{k\ge0}$, $b \bydef (b_k)_{k\ge0}$, $f_1(a) \bydef ((f_1(a))_k)_{k\ge0}$ and $f_2(\lambda,a,b) \bydef ((f_2(\lambda,a,b))_k)_{k\ge0}$. Using this notation, denote $x=(\lambda,a,b)$ and define the nonlinear operator
\begin{equation} \label{eq:F=0_eigs}
F(x) \bydef
\begin{pmatrix}
\eta(b)
\\
f_1(a)
\\
f_2(\lambda,a,b)
\end{pmatrix}.
\end{equation}

Given a weight $\nu \ge 1$, recall the Banach space $\ell_\nu^1$ given in \eqref{eq:ell_nu_one}, and define $X \bydef \C \times \ell_\nu^1 \times \ell_\nu^1$,
%
%
with the induced norm, given $x=(\lambda,a,b) \in X$,
\begin{equation} \label{eq:normX}
\| x \|_X \bydef \max\{ |\lambda|, |a|_\nu,|b|_{\nu} \}.
\end{equation}
The problem of simultaneously looking for a steady state $a$ and an eigenpair $(\lambda,b)$ therefore reduces to finding $x \in X$ such that $F(x)=0$, where the map $F$ is defined component-wise in \eqref{eq:F=0_eigs}. Solving the problem $F=0$ in $X$ is done using computer-assisted proofs via the Newton-Kantorovich type Theorem~\ref{thm:radii_polynomials}.

%
%
%

Proving the existence of a solution of $F=0$ using Theorem~\ref{thm:radii_polynomials} is often called the {\em radii polynomial approach}. This approach begins by computing an approximate solution $\bx$ of $F=0$. This first requires considering a finite dimensional projection. Fixing a Fourier projection size $m \in \N$, denote a finite dimensional projection of $x \in X$ by $x^{(m)} = \big( \lambda,(a_k)_{k=0}^{m-1} ,(b_k)_{k=0}^{m-1} \big) \in \C^{2m+1}$. The finite dimensional projection of $F$ is then given by $F^{(m)}=(\eta-1,f_1^{(m)},f_2^{(m)}):\C^{2m+1} \to \C^{2m+1}$ defined by 
\begin{equation} \label{eq:projection_F}
F^{(m)}(\lambda,a^{(m)},b^{(m)}) \bydef 
\begin{bmatrix}
\eta(b^{(m)})
\\
\left( (f_1(a^{(m)}))_k \right)_{0\leq k<m}
\\
\left( (f_2(\lambda,a^{(m)},b^{(m)}))_k \right)_{0\leq k<m}
\end{bmatrix}.
\end{equation}
%
Assume that a solution $\bx^{(m)}$ such that $F^{(m)}(\bx^{(m)}) \approx 0$ has been computed (e.g.~using Newton's method).
Denote $\ba = \left( \ba_0 ,\dots,\ba_{m-1},0,0,0,\dots \right)$ the vector which consists of embedding $\ba^{(m)} \in \C^m$ in the infinite dimensional space $\ell_\nu^1$ by {\em padding} the tail by infinitely many zeroes. We treat $\bb$ similarly. Denote $\bx = (\bar \lambda, \ba,\bb)$, and for the sake of simplicity of the presentation, we use the same notation $\bx$ to denote $\bx \in X$ and $\bx^{(m)} \in \C^{2m+1}$. Denote by $DF^{(m)}(\bx)$ the Jacobian of $F^{(m)}$ at $\bx$, and let us write it as
\[
DF^{(m)}(\bx)=
\begin{pmatrix}
0 & 0 & D_b \eta \\
0 & D_{a} f_1^{(m)}(\ba) & 0 \\
D_{\lambda} f_2^{(m)}(\bar \lambda, \ba,\bb) & D_{a} f_2^{(m)}(\bar \lambda, \ba,\bb) & D_{b} f_2^{(m)}(\bar \lambda, \ba,\bb)
\end{pmatrix} \in M_{2m+1}(\C).
\]
From now on, we denote $a_1=a$, $a_2=b$, and $x = (\lambda,a_1,a_2) \in X$.

The next step is to construct the linear operator $A^\dag$ (an approximate derivative of the derivative $DF(\bx)$), and the linear operator $A$ (an approximate inverse of $DF(\bx)$). Let
\begin{equation} \label{eq:dagA_eigenpair}
A^\dagger=
\begin{pmatrix}
A_{0,0}^\dagger &  A_{0,1}^\dagger & A_{0,2}^\dagger\\
A_{1,0}^\dagger & A_{1,1}^\dagger & A_{1,2}^\dagger\\
A_{2,0}^\dagger & A_{2,1}^\dagger & A_{2,2}^\dagger
\end{pmatrix} ,
\end{equation}
whose action on an element $h=(h_0,h_1,h_2) \in X$ is defined by $(A^\dagger h)_i = A_{i,0}^\dagger h_0 + A_{i,1}^\dagger h_1 + A_{i,2}^\dagger h_2$, for $i=0,1,2$. For $i=1,2$, the action of $A_{i,j}^\dagger$ is defined as
\begin{align*}
(A_{i,1}^\dagger h_1)_k &= 
\begin{cases}
\bigl(D_{a_1} f_i^{(m)}(\bx) h_1^{(m)} \bigr)_k &\quad\text{for }   0 \leq k < m ,  \\
\delta_{i,2} \mu_k (h_1)_k  &\quad\text{for }  k \ge m,
\end{cases}
\\
(A_{i,2}^\dagger h_2)_k &= 
\begin{cases}
\bigl(D_{a_2} f_i^{(m)}(\bx) h_2^{(m)} \bigr)_k &\quad\text{for }   0 \leq k < m,   \\
\delta_{i,1} i  \mu_k (h_2)_k  &\quad\text{for }  k \ge m,
\end{cases}
\end{align*}
where $\delta_{i,j}$ is the Kronecker $\delta$. 
Consider now a matrix $A^{(m)} \in M_{2m+1}(\C)$ computed so that $A^{(m)} \approx {DF^{(m)}(\bx)}^{-1}$. We decompose it into nine blocks:
\[
A^{(m)}=
\begin{pmatrix}
A_{0,0}^{(m)} & A_{0,1}^{(m)} & A_{0,2}^{(m)}\\
A_{1,0}^{(m)} & A_{1,1}^{(m)} & A_{1,2}^{(m)}\\
A_{2,0}^{(m)} & A_{2,1}^{(m)} & A_{2,2}^{(m)}
\end{pmatrix}.
\]
This allows defining the linear operator $A$ as
\begin{equation} \label{eq:A_eigenpair}
A=
\begin{pmatrix}
A_{0,0} & A_{0,1} & A_{0,2}\\
A_{1,0} & A_{1,1} & A_{1,2}\\
A_{2,0} & A_{2,1} & A_{2,2}
\end{pmatrix} ,
\end{equation}
whose action on an element $h=(h_0,h_1,h_2) \in X$ is defined by $(Ah)_i = A_{i,0} h_0 +A_{i,1} h_1 + A_{i,2} h_2$, for $i=0,1,2$. 
For $i,j=1,2$, $A_{i,j}$  is defined as
 \begin{align*}
 (A_{i,1} h_1)_k &=
 \begin{cases}
 \left(A_{i,1}^{(m)} h_1^{(m)} \right)_k & \text{for }   0 \leq k < m    \\
\delta_{i,2} \frac{1}{\mu_k} (h_1)_k  & \text{for }  k \ge m 
 \end{cases}
 \\
 (A_{i,2} h_2)_k &=
 \begin{cases}
 \left( A_{i,2}^{(m)}  h_2^{(m)} \right)_k & \text{for }   0 \leq k < m   \\
 \delta_{i,1} \frac{1}{i  \mu_k} (h_2)_k  & \text{for }  k  \ge m.
 \end{cases}
 \end{align*}
Having obtained an approximate solution $\ba$ and the linear operators $\dagA$ and $A$, the next step is to construct the bounds $Y_0$, $Z_0$, $Z_1$ and $Z_2(r)$ satisfying \eqref{eq:general_Y_0}, \eqref{eq:general_Z_0}, \eqref{eq:general_Z_1} and \eqref{eq:general_Z_2}, respectively.  

Since the bounds $Y_0$ and $Z_0$ are standard, we only present the derivation of the bounds $Z_1$ and $Z_2$.

\subsubsection{The \boldmath $Z_1$ \unboldmath bound} \label{sec:Z1_Appendix}

Recall that we look for the bound 
$
\| A[DF(\bx) - A^{\dagger} ] \|_{B(X)} \le Z_1
$.
Given $h=(h_0,h_1,h_2) \in X$ with $\|h\|_X \le 1$, set
\[
z \bydef [DF(\bx) -  A^{\dagger} ] h.
\]
Denote $z=(z_0,z_1,z_2)$. Recalling \eqref{eq:f=0} and \eqref{eq:g=0}, then 
$f_1(a) = \mu a + a^2$ and $f_2(\lambda,a,b) =  i  \left(\mu b + 2 a*b  \right) - \lambda b$, where $\mu$ is a diagonal operator with diagonal entries $\mu_k$, $k \ge 0$.
Since in $z$ some of the terms involving $((h_1)_k)_{k=0}^{m-1}$ will cancel, it is useful to introduce $\widehat h_1$ as follows:
\[
 (\widehat h_1)_k \bydef \begin{cases} 0 & \text{if } k<m,\\
  (h_1)_k & \text{if } k \geq m. \end{cases}
\]
Then, $z_0=0$, and 
 \begin{align*}
 (z_1)_k &=
 \begin{cases}
 \displaystyle 2 (\ba*\widehat h_1)_k & \text{for } k=0,\dots,m-1 \\
 \displaystyle 2 (\ba* h_1)_k & \text{for } k \ge m \\
 \end{cases}
 \\
 (z_2)_k &=
 \begin{cases}
 	2 i  \left[ (\ba*\widehat h_2)_k + (\bb*\widehat h_1)_k \right] & \text{for }k=0,\dots,m-1\\
 	2 i  \left[ (\ba* h_2)_k + (\bb*h_1)_k \right] - \bar \lambda (h_2)_k & \text{for }k \ge m.
 \end{cases}
 \end{align*}
By Corollary~\ref{cor:psi_k}, for $k =0,\dots,m-1$, $| (z_1)_k| \le 2 \Psi_k(\ba)$ and
$| (z_2)_k| \le 2 \left( \Psi_k(\ba) + \Psi_k(\bb)  \right)$. Hence,
\begin{align*}
|(Az)_0 | & \le Z_1^{(0)} \bydef 2 |A^{(m)}_{0,1}| \Psi^{(m)}(\ba) + 2 |A^{(m)}_{0,2}| \left( \Psi^{(m)}(\ba) + \Psi^{(m)}(\bb) \right).
\end{align*}
Moreover, for $\ell=1,2$
\begin{align*}
|(Az)_\ell |_{\nu} & \le \sum_{j=1}^2 | A_{\ell,j} z_j |_{\nu}  
\\& 
=  \sum_{k=0}^{m-1} \bigl| \bigl(A^{(m)}_{\ell,1} z_1^{(m)} \bigr)_k \bigr| \omega_k
+  \sum_{k \ge m}  \frac{\delta_{\ell,1}}{|\mu_k|} | (z_1)_k | \omega_k
\sum_{k=0}^{m-1} \bigl| \bigl(A^{(m)}_{\ell,2} z_2^{(m)} \bigr)_k \bigr| \omega_k
+  \sum_{k \ge m}  \frac{\delta_{\ell,2}}{|i  \mu_k|} | (z_2)_k | \omega_k \\
& \le 2 \sum_{k=0}^{m-1} \bigl| \bigl( |A^{(m)}_{\ell,1}| \Psi^{(m)}(\ba) \bigr)_k \bigr| \omega_k 
+ \frac{\delta_{\ell,1}}{4 m^2 \pi^2} \left(  \sum_{k \ge m} | (z_1)_k | \omega_k  \right) \\
& \quad + 2 \sum_{k=0}^{m-1} \bigl| \bigl( |A^{(m)}_{\ell,2}| (\Psi^{(m)}(\ba)+\Psi^{(m)}(\bb)) \bigr)_k \bigr| \omega_k + \frac{\delta_{\ell,2}}{4 m^2 \pi^2} \left(  \sum_{k \ge m} | (z_2)_k | \omega_k  \right) \\
& \le 2 \sum_{k=0}^{m-1} \bigl| \bigl( |A^{(m)}_{\ell,1}| \Psi^{(m)}(\ba) \bigr)_k \bigr| \omega_k 
+ \frac{\delta_{\ell,1}}{4 m^2 \pi^2} (2 |\ba * h_1 |_{\nu}) \\
& \quad + 2 \sum_{k=0}^{m-1} \bigl| \bigl( |A^{(m)}_{\ell,2}| (\Psi^{(m)}(\ba)+\Psi^{(m)}(\bb)) \bigr)_k \bigr| \omega_k + \frac{\delta_{\ell,2}}{4 m^2 \pi^2} \left( 2 | \ba* h_2|_{\nu} + 2 | \bb*h_1|_{\nu} + |\bar \lambda| |h_2|_{\nu} \right) \\
&\le 2 \sum_{k=0}^{m-1} \bigl| \bigl( |A^{(m)}_{\ell,1}| \Psi^{(m)}(\ba) \bigr)_k \bigr| \omega_k 
+ 2 \sum_{k=0}^{m-1} \bigl| \bigl( |A^{(m)}_{\ell,2}| (\Psi^{(m)}(\ba)+\Psi^{(m)}(\bb)) \bigr)_k \bigr| \omega_k 
\\
& \quad + 
 \frac{\delta_{\ell,1}}{4 m^2 \pi^2}  | \ba|_{\nu}+
\frac{\delta_{\ell,2}}{4 m^2 \pi^2} \left( 2 | \ba|_{\nu} + 2 | \bb|_{\nu} + |\bar \lambda| \right) \,\bydef Z_1^{(\ell)}.
\end{align*}
We thus define 
\begin{equation} \label{eq:Z1_equilibria}
Z_1 \bydef \max\left(Z_1^{(0)},Z_1^{(1)},Z_1^{(2)} \right).
\end{equation}

\subsubsection{The \boldmath $Z_2$ \unboldmath bound}

Let $r>0$ and $c=(c_0,c_1,c_2) \in B_r(\ba)$, that is 
\[
\| c- \bx \|_{X} = \max(|c_0-\bar \lambda|, |c_1- \ba |_{\nu}, |c_2- \bb |_{\nu} ) \le r. 
\]
Given $\| h \|_X \leq 1$, note that 
\begin{align*}
[Df_1(c)-Df_1(\bx)] h  & = 2 (c_1-\ba) h_1
\\
[Df_2(c)-Df_2(\bx)] h & = i  \left( 2 (c_1-\ba) h_2 + 2 (c_2-\bb) h_1 \right) - (c_0 - \bar \lambda)h_2 - (c_2 - \bb) h_0
\end{align*}
so that 
\begin{align*}
\| [Df_1(c)-Df_1(\bx)] h\|_{1,\nu}  & \le 2 r
\\
\| [Df_2(c)-Df_2(\bx)] h\|_{1,\nu} & \le 6 r.
\end{align*}
Hence, 
\begin{align*}
\| A [DF(c) - DF(\ba)]\|_{B(X)} &= \sup_{\| h \|_X \leq 1} \| A [DF(c) - DF(\ba)] h \|_{X} \\
& \le \| A \|_{B(X)} \max \{ 2r, 6r \} = \left( 6  \| A \|_{B(X)} \right) r
\end{align*}
Then, we set
\begin{align} \label{eq:Z2_equilibria}
Z_2 &\bydef 6  \| A \|_{B(X)},
\end{align}
where an upper bound for $\| A \|_{B(X)}$ can be computed using formula \eqref{eq:norm_B_ell_nu_1} from Corollary~\ref{cor:OperatorNorm} and using a similar approach than the one of Lemma~\ref{lem:normA}.

Having explicit and computable formulas for the bounds $Y_0$ (standard), $Z_0$ (standard), $Z_1$ (given in \eqref{eq:Z1_equilibria}) and $Z_2$ (given in \eqref{eq:Z2_equilibria}), and recalling \eqref{eq:general_radii_polynomial}, define the radii polynomial by 
\[
p(r) \bydef Z_2 r^2 - ( 1 - Z_1 - Z_0) r + Y_0.
\]
We wrote a MATLAB program (available at \cite{bib:codes}) which rigorously computes (i.e. controlling the floating point errors) the bounds $Y_0$, $Z_0$, $Z_1$ and $Z_2$. 
The program uses the interval arithmetic library INTLAB available at \cite{Ru99a}.
Fixing $\nu =1$ and using the MATLAB program, verify the existence of $r_0>0$ such that $p(r_0)<0$. By Theorem~\ref{thm:radii_polynomials}, there exists a unique $\tx = (\tilde \lambda, \ta,\tb) \in B_{r_0}(\bx)$ such that $F(\tx) = 0$. This approach provides a rigorous steady state $\ta$ and an eigenpair $(\tilde{\lambda},\tb)$ such that
\[
	|\ta -\ba|_\nu\le r_0,\quad |\tilde{\lambda}-\bar{\lambda}|\le r_0,\quad |\tb -\bb|_\nu\le r_0.
\]

\section{List of computational parameters}
\label{sec:ComputationalTable}

 We present in Table \ref{tab:Table} the computational parameters used to produce computer assisted proofs of all the heteroclinic orbits described in this paper. 
The variable names, as described in Section \ref{sec:FinalTheorem}, are as follows: 
    \textbf{Step 1.} 
The parameter $r_0$ denotes accuracy of our approximate steady state and eigenpair; 
the parameter $r_p$ denotes accuracy of our unstable manifold approximation.
\textbf{Step 2.} 
The parameter ``\# of times steps'' denotes number of time steps the rigorous integrator took in solving the solution from time $0$ until time  $t_{\mathrm{end}}$;  
The parameters $\varrho_0$ and $\varrho_\infty$ denote the \emph{a posteriori} error of the final solution in the zero and nonzero modes respectively.  
\textbf{Step 3.} 
The parameters $\rho_0$, $\rho_1$, and $r$ denote values for which the hypothesis of Theorem \ref{thm:HomoclinicBlowup}  may be satisfied and prove the endpoint of the rigorous integrator will converge to the zero equilibrium. 

\begin{table}[ht]
\centering
	\caption{Computational results for proving the existence of heteroclinics.}
	\resizebox{0.9\textwidth}{!}{
		\begin{tabular}{cccccccccc}
			\hline
			Orbit&$r_0$&$r_p$& \# of time steps & $t_{\mathrm{end}}$ &$\varrho_{0}$&$\varrho_{\infty}$ & $\rho_0$&$\rho_1$&$r$\\
			\hline
			$u^{i}_1\to P(1) \to 0$ & 1.61e-12 & 5.91e-10 & 20 & 5e-2 & 1.43e{-8} & 1.97e{-7} & 17.14 & 1.37e-2 & 3.56e-2 \\
			$(u^{i}_1)^*\to P(1) \to 0$ & 5.54e-12 & 8.88e-10& 20 & 5e-2 & 2.24e-8 & 3.08e-7 & 17.14 &  1.37e-2 & 3.56e-2 \\
			\hline
			$u^{i}_1\to P(-1) \to 0$ & 1.61e-12 & 5.91e-10 & 2523 & 9.09e-1 & 7.65e{-6} & 1.21e{-4} & 1.05 & 2.24e-1 & 5.96e-1 \\
			$(u^{i}_1)^*\to P(-1) \to 0$ &  5.54e-12 & 8.87e-10 & 2526 & 9.09e-1 & 1.16e-5 & 1.81e-4& 1.05 & 2.24e-1 & 6.03e-1 \\
			\hline
			$u^{ii}_1\to P(1) \to 0$ & 7.18e-10 & 2.60e-7 & 1329 & 4.21e-2& 1.32e{-2} & 2.35e{-1} & 27.01 & 8.67e-3 & 2.29e-2 \\
			$(u^{ii}_1)^*\to P(1) \to 0$ & 7.15e-10 & 2.59e-7 & 1329 & 4.21e-2 &  1.31e{-2}& 2.34e{-1} &  27.01 & 8.67e-3 & 2.28e-2\\
			\hline
		\end{tabular}}
\label{tab:Table}
\end{table}

\bibliography{Bib_NLS,Bib_complexblowup}
\bibliographystyle{abbrv}

\end{document}